\documentclass{amsart}
\usepackage{latexsym,graphicx,epstopdf}
\usepackage{amssymb,amsmath,amsfonts}
\usepackage{epsfig,amsmath,amssymb,amsthm,color}
\newtheorem{theorem}{Theorem}[section]
\newtheorem{lemma}[theorem]{Lemma}
\theoremstyle{corollary}
\newtheorem{corollary}[theorem]{Corollary}

\theoremstyle{definition}
\newtheorem{definition}[theorem]{Definition}

\newtheorem{example}[theorem]{Example}

\theoremstyle{remark}
\newtheorem{remark}[theorem]{Remark}

\numberwithin{equation}{section}
\begin{document}
\thispagestyle{empty}
\title[fixed points approximation]
{Split common fixed point problems and its variant forms}
\date{}
\author[A. K{\i}l{\i}\c{c}man and L.B. Mohammed] {A. K{\i}l{\i}\c{c}man  and L.B. Mohammed}

\address{Department of Mathematics, Faculty of Science, Universiti Putra Malaysia, 43400 Serdang Selangor, Malaysia}
 \email{akilicman@yahoo.com, lawanbulama@gmail.com}
\begin{abstract} The split common fixed point problems has found its applications in various branches of mathematics both pure and applied. It provides us a unified structure to study a large number of nonlinear mappings. Our interest here is to apply these mappings and  propose some iterative methods for solving the split common fixed point problems and its variant forms, and we prove the convergence results of these algorithms.\\

As a special case  of the split common fixed problems, we consider  the split common fixed point equality problems for the class of finite family of quasi-nonexpansive mappings. Furthermore, we consider another problem namely split feasibility and fixed point equality problems and suggest some new iterative methods and prove their convergence results for the class of quasi-nonexpansive mappings.\\

Finally, as a special case of the split feasibility and fixed point equality problems, we consider the split feasibility and fixed point problems and propose Ishikawa-type extra-gradients algorithms for solving these split feasibility and fixed point problems for the class of quasi-nonexpansive mappings in Hilbert spaces. In the end, we prove the convergence results of the proposed algorithms.\\

Results proved in this chapter continue to hold for different type of problems, such as;
convex feasibility problem, split feasibility problem and multiple-set split feasibility
problems.\\



\end{abstract}
\maketitle

\section{Introduction}
Functional analysis is an abstract branch of mathematics
that originated from classical analysis. The impetus came from; linear algebra,
problems related to ordinary and partial differential equations,
calculus of variations, approximation theory,
integral equations, and so on. Functional analysis can be defined
as the study of certain topological-algebraic structures and of
the methods by which the knowledge of these structures can be
applied to analytic problems, see  Rudin \cite{rudin1973functional}.\\

Fixed point theory (FPT) is one of the most powerful and fruitful tools of modern
mathematics and may be considered a core subject of nonlinear analysis.
It has been a  nourishing area of research for many mathematicians. The origins of the theory, which date to the later part of the nineteenth century, rest in the use of successive approximations to establish the existence and uniqueness of the solutions, particularly to differential equations, for example, see \cite{arino1984fixed,yamamoto1998numerical,taleb2000fixed,nieto2005contractive,pathak2007common,sestelo2015existence} and references therein.\\

The classical importance of fixed point theory in functional analysis is due to
its usefulness in the theory of ordinary and partial differential equations.
The existence or construction of a solution to a differential equation often reduces to
the existence or location of a fixed point for an operator defined on a subset of
a space of functions. Fixed point theory had also been used to determine the
existence of periodic solutions for functional differential equations when solutions
are already known to exist,  for example, see \cite{chow1974existence,grimmer1979existence,torres2003existence,kiss2012computational} and references therein.\\

Related to the FPT, we have the split common fixed point problems (SCFPP). The SCFPP was introduced and studied by Censor and Segal \cite{censor2009split} as a generalization of many existing problems in nonlinear sciences, both pure and applied. Moreover, Censor and Segal \cite{censor2009split} had shown that the problem of fixed point, convex feasibility, multiple-set split feasibility,  split feasibility and much more can be studied more conveniently as SCFPP. The results and conclusions that are true for the SCFPP continue to hold for these problems, and it shows the significance and range of applicability of the SCFPP. One of the important applications of SCFPP can be seen in intensity modulation radiation therapy  (IMRT), for more details, see Censor et al., \cite{censor2006unified}.\\

This research work falls within the general area of  ``Nonlinear Functional Analysis'', an area with the vast amount of applicability in the recent years,  as such becoming the object of an increasing amount of study. We  focus on an important topic within this area ``\textbf{ A note on the split common fixed fixed point problem and its variant forms}.'' In this regard, we discuss the SCFPP and its variant forms. We show that already known problems are special cases of the split common fixed point problems (SCFPP). We use approximation methods to suggest different iterative algorithms for solving  SCFPP and its variant forms. In the end, we give the convergence results of these algorithms.

\section{Basic Concepts and Definitions}
\subsection{Introduction}
In this section,  we give some definitions and basic results.  We start from the definition of vector space and end with some results from Hilbert spaces.    Those results that are commonly used in all the chapters are given in this section, and those results that are relevant to a particular chapter are provided at the beginning of each chapter.   In short, this section works as a foundation for the structure of this thesis.\\

\subsection{Vector Spaces} Vector spaces play a vital role in many branches of mathematics.  In fact, in various practical (and theoretical) problems we have a set V whose elements may be vectors in three-dimensional space, or sequences of numbers, or
functions, and these elements can be added and multiplied by constants (numbers) in a natural way, the result being again an element of V. Such concrete situations suggest the concept of a vector space as defined below. The definition will involve a general field $\mathbb{F},$ but in functional analysis, $\mathbb{F}$ will be $\mathbb{R}$ or  $\mathbb{C}$. The elements of $\mathbb{F}$ are called scalars,   while in this thesis they will be real or complex numbers.

\begin{definition} \normalfont A vector space over a field $\mathbb{F}$  is a nonempty set denoted by $V$ together with addition (+) and scalar multiplication $(.)$ satisfies the following conditions:
\begin{enumerate}
\item[(i)] x+y=y+x, for all $x,y \in V;$
\item[(ii)]  x+(y+w)=(x+y)+w, for all $x,y, w \in V;$
\item[(iii)] there exists  a vector denoted by $\theta$ such that   $x+\theta=x,$  for all $x\in V;$
\item[(iv)] for all $x\in V$, there exists a unique vector denoted by (-x) such that    $x+(-x)=\theta;$
\item[(v)] $\alpha.(\beta.x)=(\alpha.\beta).x,$ for all $\alpha,\beta\in\mathbb{F}$ and $x\in V;$
\item[(vi)] $\alpha.(x+y)=\alpha.x+\alpha.y,$ for all $x,y\in V$ and $\alpha\in\mathbb{F}$;
\item[(vii)] $(\alpha+\beta).x=\alpha.x+\beta.x,$ for all $\alpha,\beta\in\mathbb{F}$ and $x\in V;$
\item[(viii)] there exists $1\in\mathbb{F}$ such that   $1.x=x,$ $\forall x\in V.$
\end{enumerate}
\end{definition}

\begin{remark}\normalfont  From now we will drop the dot $(.)$ in  the scalar multiplication and denote $\alpha.\beta$
 as $\alpha\beta.$
\end{remark}

\noindent Let ${\rm v_{1},v_{2},v_{3},...,v_{n}\in V}$ and
${\rm\alpha_{1},\alpha_{2},\alpha_{3},...,\alpha_{n}}$ be scalars. Consider the equation:
\begin{equation}
{\rm\alpha_{1}v_{1}+\alpha_{2}v_{2}+\alpha_{3}v_{3}+...+\alpha_{n}v_{n} = 0}.\label{1.1a}
\end{equation}
 Trivially, \normalfont$\alpha_{1}=\alpha_{2}=\alpha_{3}=...=\alpha_{n}=0$  solves Equation (\ref{1.1a}). If it is possible to have the solution of Equation (\ref{1.1a}) with at least one of the $\alpha_{i}'s$ non zero, then the vectors ${\rm v_{1},v_{2},v_{3},...,v_{n}}$ are called \textbf{``Linearly Dependent''} otherwise they are called \textbf{``Linearly Independent''}.\\

If $\mathbb{M}\subseteq$V consist of a linearly independent set of vectors; we say that $\mathbb{M}$ is a linearly independent set.

\begin{definition}\normalfont   Span of $\mathbb{M}$ (Span$\mathbb{M}$) is defined as the set of all linear combination of  $\mathbb{M},$  i.e., Span$\mathbb{M}$ = $\{{\rm\alpha_{1}v_{1}+\alpha_{2}v_{2}+\alpha_{3}v_{3}+...,    v_{1},v_{2},v_{3},...\in V},$ where  $\alpha_{1},\alpha_{2},...$  are scalars\}.
\end{definition}

\begin{definition}\normalfont
 Let  $\mathbb{M}\subseteq$V. $\mathbb{M}$ is said to be  basis for the space V, if
\begin{enumerate}
\item[(i)] $\mathbb{M}$ is a linearly independent set,
\item[(ii)] Span$\mathbb{M}$ = V.
\end{enumerate}
\end{definition}

\begin{definition} \normalfont Let V be a vector space, the dimension of V (dim V) is the number of vectors of the basis of V.  V is of finite dimension if its dimension is finite. Otherwise, it is said to be of infinite dimensional space.
\end{definition}

\begin{definition}\normalfont  Let  $C$ be a  subset of  V.  $C$ is said to be convex, if for all $x,y\in C$, $\gamma\in[0,1],$  $(1-\gamma)x+\gamma y\in C.$ In general, for all $x_{1},x_{2},x_{3},...,x_{n}\in C$ and for $\gamma_{j}\geq 0$  such that $\sum_{j=1}^{n}\gamma_{j}=1,$ the combination $\sum_{j=1}^{n}\gamma_{j}x_{j}\in C$ is called the convex combination.
\end{definition}

\begin{definition}\normalfont\label{linear}  A mapping $T:V_{1}\to V_{2}$ is said to be linear, if  $\forall u,v\in V_{1}$ and $\alpha ,\beta $ scalars,  $$T(\alpha u+\beta v)=\alpha T(u)+\beta T(v).$$

 Limits (of convergent sequences), differentiation and integration, are examples of a linear map.
\end{definition}

\begin{remark}\normalfont If in Definition \ref{linear}, the linear space $V_{2}$ is replaced by a scalar field
 $\mathbb{F},$ then the linear map $T$ is called  linear functional on $V_{1}$.
\end{remark}

\subsection{Hilbert Space and its Properties}

\begin{definition}\normalfont
Let $Y$ be a linear space.  An inner product on $Y$ is a function
$\left\langle ., .\right\rangle :Y\times Y\rightarrow \mathbb{F}$
 such  that the following conditions are satisfies:
\begin{enumerate}
\item[(i)]$\left\langle y, y\right\rangle\geq 0$  $\forall y\in Y;$
\item[(ii)] $\left\langle y, y\right\rangle=0$ if  $y=0,$  $\forall y\in Y;$
\item[(iii)] $\left\langle y,z\right\rangle$=$\overline{\left\langle z, y\right\rangle},$ $\forall y,z\in Y,$ where the ``bar'' indicates the complex conjugation;
\item[(iv)] $\left\langle\alpha x+\beta y, z\right\rangle=\overline{\alpha}\left\langle x, z\right\rangle+\overline{\beta}\left\langle y, z\right\rangle,$ for all $x,y,z\in Y$ and $\alpha,\beta\in\mathbb{C}.$
\end{enumerate}
\end{definition}
\begin{remark}\normalfont The pair $(Y, \left\langle ., .\right\rangle)$ is called an inner product space. We shall simply write $Y$
for the inner product space $(Y, \left\langle ., .\right\rangle)$ when the inner product $\left\langle ., .\right\rangle$  is known.
Furthermore, if  $Y$ is a real vector space, then condition (iii) above reduces to $\left\langle x, z\right\rangle=\left\langle z, x\right\rangle$ (Symmetry).
\end{remark}

\begin{definition}\normalfont Let Y be a linear space over  $\mathbb{F}$ $ (\mathbb{R}$ or $\mathbb{C}).$ A norm on Y is a real-valued function $\|.\|:Y\to  \mathbb{R}$ such that  the following conditions are satisfies:
\begin{enumerate}
\item[(i)]$\left\|x\right\|\geq 0,$ $\forall x\in Y;$
\item[(ii)] $\left\|x\right\|= 0$ if  $x = 0,$  $\forall x\in Y;$
\item[(iii)] $\left\|\alpha x\right\|=\left|\alpha\right|\left\|x\right\|$, $\forall x\in Y$ and $\alpha \in \mathbb{R};$
\item[(iv)] $\left\|x+z\right\|\leq\left\|x\right\|+\left\|z\right\|,$ $\forall x,z\in Y.$
\end{enumerate}
\end{definition}
\begin{remark}\normalfont A linear space $Y$ with a norm defined on it i.e., $(Y, \|.\|)$ is called a normed linear space.  If $Y$ is a normed linear space, the norm $\|.\|$ always induces a metric $d$ on Y given by $d(z, x)=\|z-x\|$ for each $x,z\in Y,$ with this,
$(Y,d)$ become a metric space. For a quick review of metric space the reader may consult Dunford et al., \cite{dunford1971linear}.
\end{remark}

\begin{lemma}\normalfont  Let $Y$ be an inner product space. For arbitrary $x, z \in Y,$
\begin{equation}
|\left\langle x, z\right\rangle|^{2}\leq \left\langle x, x\right\rangle \left\langle z, z\right\rangle.\label{reduces}
\end{equation}
If x and z are linearly dependent, then Equation (\ref{reduces}) reduces to
\begin{equation*}
|\left\langle x, z\right\rangle|^{2}= \left\langle x, x\right\rangle \left\langle z, z\right\rangle.
\end{equation*}
\end{lemma}
This lemma is known as Cauchy-Schwartz Inequality. For more details about the proof, one is referred to Chidume \cite{chidume2006applicable}.

\begin{lemma}\normalfont \label{2.13}
A mapping $\|.\|:Y\to\mathbb{R}$ defined by $$\|x\|=\sqrt{\left\langle x, x\right\rangle},\forall x\in Y$$
is a norm on Y.
\end{lemma}
\begin{remark}\normalfont
As the consequence of Lemma \ref{2.13}, Equation (\ref{reduces}) reduces to the following inequality:
 \begin{equation*}|\left\langle x, z\right\rangle|\leq \| x\|\|z\|, \forall x,z\in Y.
\end{equation*}
\end{remark}

\begin{definition}\normalfont
A sequence   $\{ y_{n}\} $  in a  normed linear space $Y$ is said to   converge to  $y\in Y,$ if $\forall \epsilon >0$, there exists $N_{\epsilon}\in\mathbb{N}$, such that  $\left\|y_{n}-y\right\|<\epsilon$, $ \forall n \geq N_{\epsilon}.$ The vector $y\in Y$ is called the limit of  the sequence $\{y_{n}\}$ and is written as $\underset{n\to\infty}{\lim}y_{n}=y$ or $y_{n}\to y,$ as $n \rightarrow \infty$.
\end{definition}

\begin{definition}\normalfont
A sequence   $\{y_{n}\}$  in a  normed linear space $Y$ is said to  converge weakly to  $y\in Y,$ if for all $h\in Y^{*}$  such that   $\underset{n\to\infty}{\lim}h(y_{n})= h(y),$ where $Y^{*}$ denote the dual space of $Y$.
\end{definition}

Next, we give some results regards to the weak convergence of a sequence. For more details about the proof, see Chidume \cite{chidume2006applicable}.

\begin{lemma}\normalfont\label{lemma} Let   $\{ y_{n}\}\subseteq E$ (Banach space).  Then the following results are satisfies:
\begin{enumerate}
\item[(i)] $y_{n}\rightharpoonup y$ $\Leftrightarrow$  $h(y_n) \rightarrow h(y)$ for each $h\in E^{*};$
\item[(ii)] $y_{n} \rightarrow y$ $\Rightarrow$ $ y_{n} \rightharpoonup y;$
\item[(iii)] $y_{n} \rightharpoonup y$  $\Rightarrow$ $\{y_{n}\}$ is bounded and $$\left\| y \right\| \leq \liminf_{n\rightarrow\infty}\left\|y_{n}\right\|;$$
\item[$(iv)$] $y_{n} \rightharpoonup y$ (in E),  $h_{n} \rightarrow h$  (in $E^{*}$) $\Rightarrow$  $h_{n}(y_{n}) \rightarrow h(y)$  ( in $\mathbb{R}).$
\end{enumerate}
\end{lemma}

\begin{remark}\normalfont Lemma \ref{lemma} (ii) Shows that strong convergence implies weak convergence. However, the converse may not necessarily be true, that is, in an infinite dimensional  space, weak convergence does not always imply strong convergence, while they are the same if the  dimension is finite. For the example of weak convergence which is not strong convergence, see  Chidume \cite{chidume2006applicable} and the references therein.
\end{remark}

\begin{definition}\normalfont Let $C$ be a  subset of  $H.$ A sequence $\{y_{n}\}$ in $H$ is said to be Fejer monotone,  if
$$\left\|y_{n+1}-z\right\|\leq\left\|y_{n}-z\right\|, \forall n\geq 1, z\in C.$$
\end{definition}

\begin{definition}\normalfont
A sequence $\{y_{n}\}$ in a normed linear space $Y$ is said to be Cauchy, if  $\forall\epsilon >0,$  $\exists N_{\epsilon}\in\mathbb{N}$ such that  $\left\|y_{n}-y_{m}\right\|<\epsilon,$ $\forall n, m\geq N_{\epsilon}.$
\end{definition}

\begin{definition}\normalfont A normed linear   space $Y$  is said to be  complete if and only if every Cauchy sequence in Y converges.
\end{definition}

\begin{remark}\normalfont
With respect to the norm defined in Lemma \ref{2.13}, we can define the Cauchy sequence
in an inner product space $Y$. A sequence $\{y_{n}\}$ in $Y$ is said to be  Cauchy if and only
if $\left\langle y_{n}-y_{m}, y_{n}-y_{m}\right\rangle^{1/2}:=\left\|y_{n}-y_{m} \right\|\to 0$ as $n,m \to \infty.$
\end{remark}
\begin{definition}\normalfont An inner product space $Y$ is said to be complete if and only if every Cauchy sequence  converges.
\end{definition}

\begin{definition}\normalfont A complete inner product space is called a ''Hilbert Space'' and
that of normed linear space is known as a ''Banach Space''.
\end{definition}

\subsection{Bounded Linear Map and its Properties}

\begin{definition}\normalfont Let  $T:H\to H$ be a linear map. T is said to be bounded, if  there exists a  constant $M\geq 0$ such that
$$\|T(y)\|\leq M\|y\|, \forall y\in H.$$
\end{definition}

Next, we give some results of a linear map that are continuous. For more details about the proof, see Chidume \cite{chidume2006applicable}.

\begin{lemma}\normalfont\label{cont} Let $X$ and $Y$ be normed linear spaces and   $T:X \to Y$ be a linear operator.  Then the following results are equivalent:
\begin{itemize}
\item [(i)] $T$ is continuous;
\item [(ii)] $T$ is continuous at the origin i.e., if $\{x_{n}\}$ is a sequence in $X$ such that
$$\underset{n\to\infty}{\lim}x_{n}= 0, {\rm~ then~} \underset{n\to\infty}{\lim}Tx_{n}= 0 {\rm~in~} Y;$$
\item [(iii)] $T$ is Lipschitz, i.e., in the sense that there exists $M\geq 0$ such that
                     $$\|Tx\|\leq M\|x\|, \forall x\in X;$$
\item [(iv)] $T(\Delta)$ is bounded \Big(in the sense that there exists $M\geq 0$ such that $\|Tx\|\leq M$ for all $x\in \Delta$, where  $\Delta:=\{x\in X:\|x\|\leq 1\}$\Big).
\end{itemize}
\end{lemma}

\begin{remark}\normalfont In the light of Lemma \ref{cont}, we have that a linear map $T:X\to Y$ is continuous iff it is bounded.
\end{remark}

\begin{definition}\normalfont Let  $A:H\to H$ be a bounded linear map. Define a mapping $A^{*}:H\to H$ by
$$ \left\langle Ay, z\right\rangle=\left\langle y, A^{*}z\right\rangle ,\forall y,z \in H.$$
The mapping $A^{*}$ is called the adjoint of A.
\end{definition}

The following results are fundamental for the adjoint operator on Hilbert space. For the proof, see Chidume \cite{chidume2006applicable}.

\begin{lemma}\normalfont Let $A:H\to H$ be  a bounded linear map with its adjoint $A^{*}.$  Then the following hold:
\begin{enumerate}
\item[(i)] $(A^{*})^{*}=A;$
\item[(ii)] $\|A\|=\|A^{*}\|;$
\item[(iii)] $\|A^{*}A\|=\|A\|^{2}.$
\end{enumerate}
\end{lemma}

\subsection{Some  Nonlinear Operators}

Let   $T:H\to H$ be a map. A point $x\in H$ is called a {\bf fixed point} of $T$ provided $Tx=x.$ We denote  the set of fixed point  of $T$ by  $Fix(T),$  that is
\begin{equation*}
Fix(T)=\{x\in H:Tx=x\}.
\end{equation*}
The $Fix(T)$ is closed and convex, for more details, see Goebel and Kirk \cite{goebel1990topics}.\\

$T$ is said to be $\eta-$\textbf{strongly monotone}, if there exists a  constant $\eta>0$ such that
 $$\left\langle Tx-Ty, x-y\right\rangle \geq\eta\left\|x-y\right\|, \forall x,y\in H,$$

and it is said to be  \textbf{contraction}, if
\begin{equation}
\left\|Tx-Tz\right\|\leq k\left\|x-z\right\|,  \forall x,z\in H,\label{1contraction}
\end{equation}
where  $k\in (0,1).$

\begin{remark}\normalfont\label{l101A} If  $T:H\to H$ is a contraction mapping with coefficient $k\in (0,1)$, then $(I-T)$ is
$(1-k)-$strongly monotone, that is
\begin{align}
\left\langle (I-T)w-(I-T)z, ~w-z\right\rangle&\geq (1-k)\left\|w-z\right\|^{2}, \forall w,z \in H.\nonumber
\end{align}
\end{remark}
\begin{proof}

\begin{align}
\left\langle (I-T)w-(I-T)z, ~w-z\right\rangle&=\left\langle ~w-z, ~w-z\right\rangle+\left\langle Tz-Tw, ~w-z\right\rangle\nonumber
\\&=\left\langle ~w-z, ~w-z\right\rangle-\left\langle Tw-Tz, ~w-z\right\rangle.\label{bava}
\end{align}
On the other hand,

\begin{align}
\left\langle Tw-Tz, ~w-z\right\rangle&\leq\|Tw-Tz\|\|w-z\|\nonumber
\\&\leq k\|w-z\|^{2} {\rm~~since~f~is ~a ~contraction~ mapping}.\label{baa}
\end{align}
By (\ref{bava}) and (\ref{baa}), we deduce that
\begin{align}
\left\langle (I-T)w-(I-T)z, ~w-z\right\rangle&\geq (1-k)\left\|w-z\right\|^{2}.\nonumber
\end{align}
And the proof completed.
\end{proof}

Equation (\ref{1contraction}) reduces to the following equation as $k=1.$
\begin{equation*}
\left\|Tx-Tz\right\|\leq\left\|x-z\right\|,  \forall x,z\in H.
\end{equation*}
This is known as nonexpansive mapping. As a generalization of nonexpansive mapping, we have   \textbf{asymptotically nonexpansive} (see Goebel and Kirk \cite{goebel1972fixed}), this mapping is defined as:
\begin{equation*}
\left\|T^{n}x-T^{n}z \right\|\leq k_{n}\left\|x-z\right\|,  \forall n\geq 1 {\rm~and~} x,z\in H,
\end{equation*}
where   $k_{n}\subset [1, \infty)$ such that  $\underset{n\to\infty}{\lim}k_{n}=1.$\\

The map $T$ is said to be \textbf{total asymptotically nonexpansive} (see Alber \cite{alber2006approximating}), if
\begin{equation*}
\left\|T^{n}x-T^{n}z\right\|^{2}\leq\left\|x-z\right\|^{2}+v_{n}\eta(\left\|x-z\right\|)+\mu_{n},
\forall n\geq 1 {\rm~and~} x,z\in H.
\end{equation*}
where $\{v_{n}\}$ and $\{\mu_{n}\}$ are sequences  in $[0,\infty)$ such that $\underset{n\to\infty}{\lim}v_{n}= 0,$
$\underset{n\to\infty}{\lim}\mu_{n}= 0,$    and $\eta:\Re^{+} \rightarrow \Re^{+}$  is a strictly increasing continuous function
 with $\eta(0)=0.$ This  class of  mapping generalizes  the class of nonexpansive and
 asymptotically nonexpansive mappings  (for more details  see \cite{chidume2007approximation, chidume2009new} and references therein).  And it is said to be $(k, \{\mu_{n}\}, \{\xi_{n}\}, \phi)$- total asymptotically strict pseudocontraction, if there exists a constant $k\in [0,1),$ $\mu_{n}\subset[0,\infty)$, $\xi_{n}\subset [0,\infty)$  with $\mu_{n}\rightarrow 0$ and $\xi_{n}\rightarrow 0$ as $n\rightarrow \infty$, and continuous strictly increasing function $\phi:[0,\infty)\rightarrow [0,\infty)$ with $\phi(0)=0$ such that
\begin{align*}
\left\|T^{n}x-T^{n}y\right\|^{2}&\leq\left\|x-y\right\|^{2}+ k\left\|(I-T^{n})x-(I-T^{n})y\right\|^{2}
\\&+\mu_{n}\phi(\left\|x-y\right\|) +\xi_{n}, \forall x,y\in H.
\end{align*}

$T$ is said to be  \textbf{strictly pseudocontractive} (see Browder and Petryshyn \cite{browder1967construction}), if
\begin{equation*}
\left\|Tx-Tz\right\|^{2}\leq \left\|x-z\right\|^{2}+ k\left\|(I-T)x-(I-T)z\right\|^{2},  \forall x,z\in H,
\end{equation*}
where  $k\in [0, 1)$. And it is said to  \textbf{ pseudocontractive} if
\begin{equation*}
\left\|Tx-Tz\right\|^{2}\leq \left\|x-z\right\|^{2}+ \left\|(I-T)x-(I-T)z\right\|^{2},  \forall x,z\in H.
\end{equation*}
It is obvious that all nonexpansive mappings and strictly pseudocontractive mappings are pseudocontractive
mappings but the converse does not hold.\\

$T$ is said to be {\bf quasi-nonexpansive} (see Diaz and Metcalf\cite{diaz1967structure}), if  $Fix(T)\neq\emptyset$ and
\begin{equation*}
\Vert Tx-z\Vert\leq\Vert x-z\Vert,~~\forall x\in H~{\rm and}~z\in Fix(T).
\end{equation*}
This is equivalent to
\begin{equation}2\left\langle x-Tx, z-Tx\right\rangle\leq\left\|Tx-x\right\|^{2},  \forall x\in H~{\rm and}~z\in Fix(T).\label{vwx}
\end{equation}
\begin{remark}\normalfont \label{3a}
Every nonexpansive mapping with  $Fix(T)\neq\emptyset$ is a quasi-nonexpansive; however, the converse may not necessarily be true. Thus, the class of quasi- nonexpansive mapping generalizes the class of nonexpansive mapping.
\end{remark}

The following is an example of a quasi-nonexpansive mapping which is not nonexpansive mapping,  for more details, see He and Du\cite{he2012nonlinear} and references therein.

\begin{example}\normalfont\label{1.2b} Let $H=\mathbb{R},$ defined  $T:Q:=[0,\infty)\to \mathbb{R}$ by $$Ty = \frac{y^{2}+2}{1+y} {\rm~for~all~} y\in Q.$$
\end{example}

$T$ is said to be  $k-$demicontractive, if
\begin{equation}\normalfont
\|Ty-z\|^{2}\leq\|y-z\|^{2}+k\|Ty-z\|^{2}, \forall y\in H {\rm~ and~} z\in Fix(T),\label{2.15}
\end{equation}
where $k\in [0,1).$ Trivially, the class of demicontractive mapping generalizes the class of quasi-nonexpansive mapping for $k\geq 0.$ \\

The following is an example of a demicontractive mapping  which is not quasi-nonexpansive mapping, for more details,
see Chidume et al., \cite{chidume2015split} and references therein.


\begin{example}\normalfont Define a map $T:l_{2}\to l_{2}$ by
$$T(x_{1},x_{2},x_{3},...)=-\frac{5}{2}(x_{1},x_{2},x_{3},...), {\rm~for~arbitrary~vector~}(x_{1},x_{2},x_{3},...)\in l_{2} .$$
\end{example}

\begin{remark} \normalfont
If $k=-1,$ Equation (\ref{2.15})  reduces to
\begin{equation*}\normalfont
\|Ty-z\|^{2}\leq\|y-z\|^{2}-\|Ty-y\|^{2}, \forall y\in H {\rm~ and~} z\in Fix(T).
\end{equation*}
This is  known as  \textbf{firmly quasi-nonexpansive mapping}.   
    Every strictly pseudocontractive  mapping with $Fix(T)\neq\emptyset$ is a demicontractive mapping; however, the converse may
 not necessarily be true. Thus, the class of  demicontractive mapping is more general than the class of strictly pseudocontractive mapping.
\end{remark}

The following is  an example of demicontractive mapping which is not 
 strictly pseudocontractive mapping, for more details, see Browder and Petryshyn \cite{browder1967construction} and references therein.
\begin{example}\normalfont Let $C=[-1,1]$ be a sub set of a real Hilbert space $H.$ Define $T$ on $C$ by
\begin{equation*}
T(x)= {} \left\{ \begin{array}{ll}\frac{2}{3}x\sin(\frac{1}{x}), {\rm~if~} x\neq 0,\label{2.1b}
\\ 0, \hspace{0.1cm}x=0. & \textrm{ $  $}
 \end{array}  \right.
\end{equation*}
Clearly, 0 is the only fixed point of $T$ . For $x\in C,$ we have
\begin{align*}
\left|Tx-0\right|^{2}&=|Tx|^{2}
\\&=\left|\frac{2}{3}x\sin(\frac{1}{x})\right|^{2}
\\&\leq \left|\frac{2x}{3}\right|^{2}
\\&\leq |x|^{2}
\\&\leq |x-0|^{2}+k|Tx-x|^{2}, {\rm~for~ any~} k<1.
\end{align*}
Thus, $T$ is demicontractive mapping. Next, we see that $T$ is not strictly pseudocontractive mapping.
Let $x=\frac{2}{\pi}$ and $z=\frac{2}{3\pi}$,
then $|Tx-Tz|^{2}=\frac{256}{81\pi^{2}}.$ However, $$|x-z|^{2}+|(I-T)x-(I-T)z|^{2}=\frac{160}{81\pi^{2}}.$$
\end{example}
$T$ is said to be   \textbf{asymptotically quasi-nonexpansive}, if $Fix(T)\neq \emptyset$ such that for each $n\geq 1,$
\begin{equation*}
\left\|T^{n}x-z\right\|^{2}\leq t_{n}\left\|x-z\right\|^{2}, \forall z\in Fix(T) {\rm~and~} x\in H,
\end{equation*}
where  $\{t_{n}\}\subseteq [1,\infty)$ with $\underset{n\to\infty}{\lim}t_{n}= 1.$ It is clear from this definition that every  asymptotically nonexpansive mapping with $Fix(T)\neq \emptyset$ is asymptotically quasi-nonexpansive mapping.\\

Also $T$ is said to be   $(\{r_{n}\},\{k_{n}\},\eta)$-\textbf{total quasi-asymptotically nonexpansive mapping,} if
\begin{align*}
\left\|T^{n}y-z\right\|^{2}&\leq\left\|y-z\right\|^{2}+r_{n}\eta(\left\|y-z\right\|)\nonumber
\\&+k_{n}, \forall  n\geq 1, z\in Fix(T){\rm~and~}y\in H,
\end{align*}
where  $\{r_{n}\}, \{k_{n}\}$ are sequences in $[0,\infty)$ such that $\underset{n\to\infty}{\lim}r_{n}=0,$  $\underset{n\to\infty}{\lim}k_{n}= 0$  and  $\eta:\Re^{+}\rightarrow \Re^{+}$ is a strictly continuous function with $\eta(0)=0.$  This class of mapping, generalizes the class of;  quasi-nonexpansive, asymptotically quasi-nonexpansive and  total asymptotically nonexpansive mapping.\\

$T$ is said to be  $K$-Lipschitzian, if
\begin{equation*}
\left\|Ty-Tz\right\|\leq K\left\|y-z\right\|, \forall y,z\in H.
\end{equation*}
It  is said to be uniformly $K$-Lipschitzian, if
\begin{equation*}
\left\|T^{n}y-T^{n}z\right\|\leq K\left\|y-z\right\|, \forall y,z\in H.
\end{equation*}
\begin{definition}\normalfont A mapping $T:H\rightarrow H$ is said to be   class$-\tau $ operator,
if
  $$\left\langle z-Ty, y-Ty \right\rangle\leq 0, \forall z\in Fix(T) {\rm~ and~} y\in H.$$
\end{definition}
It is important to note that,  class$-\tau $ operator is also called directed operator, see Zaknoon \cite{zaknoon2003algorithmic} and Censo and Segal \cite{censor2009split}, separating operator, see Cegielski \cite{cegielski2010generalized} or cutter operator, see Cegielski and Censor \cite{cegielski2011opial} and references therein.

\begin{definition}\normalfont A self mapping   $T$ on $H_{1}$   is said to be   semi-compact if for any bounded sequence $\{x_{n}\}\subset H$ with $(I-T)x_{n}$ converges strongly  to 0,  there exists a sub-sequence say  $\{x_{n_{k}}\}$ of $\{x_{n}\}$ such that $\{x_{n_{k}}\}$ converges strongly to x.
\end{definition}

\begin{definition}\normalfont\label{l11}
 A self mapping $T$ on $C$ is said to be demiclosed, if for any sequence $\{y_{n}\}$ in $C$ such that  $y_{n}\rightharpoonup y$  and if the sequence $Ty_{n}\to z,$ then $Ty = z.$
\end{definition}

\begin{remark}\normalfont In Definition \ref{l11}, if
$z=0,$ the zero vector in $C,$  then $T$ is called demiclosed at zero, for more details, see Moudafi \cite{moudafi2011note} and references therein.
\end{remark}

\begin{lemma}\normalfont( Goebel and Kirk \cite{goebel1990topics}) If a self mapping  $T$ on $C$ is a nonexpansive mapping, then $T$ is demiclosed at zero. 
\end{lemma}

\begin{lemma}\normalfont(Acedo and Xu \cite{acedo2007iterative}) If a self mapping  $T$ on $C$ is a $k-$strictly pseudocontractive, then  $(T-I)$ is demiclosed at zero.
\end{lemma}

\begin{lemma}\normalfont\label{l5}
Let $C$ be a subset of  $H_{1},$ and $P_{C}$ be a metric projection from $H_{1}$ onto $C.$ Then $\forall y\in C ~{\rm~and~} x\in H_{1},$
\begin{equation*}
\|x-P_{C}(x)\|^{2}\leq \|y-x\|^{2}-\|y-P_{C}(x)\|^{2}.
\end{equation*}
\end{lemma}
For the proof of this lemma, see Li and He \cite{li2015new} and references therein.

\begin{lemma}\normalfont\label{l2}
 For each $x,y\in H_{1}$,   the following results hold.
\begin{itemize}
\item[(i)] $\left\|x+y\right\|^{2}=\left\|x\right\|^{2}+2\left\langle x, y\right\rangle+\left\|y\right\|^{2},$
\item[(ii)] $\left\|\alpha x+(1-\alpha)y\right\|^{2}=\alpha\left\|x\right\|^{2}+(1-\alpha)\left\|y\right\|^{2}-\alpha(1-\alpha)\left\|x-y\right\|^{2}$, $\forall$ $\alpha\in [0,1].$
\end{itemize}
\end{lemma}
For the proof of this lemma, see Acedo and Xu \cite{acedo2007iterative} and references therein.

\begin{lemma}\normalfont\label{l6} Let $\{a_{n}\}$ be a sequence of nonnegative real number such that
\begin{equation*}
a_{n+1}\leq(1-\gamma_{n})a_{n}+\sigma_{n}, n\geq 0,
\end{equation*}
where $\gamma_{n}$ is a sequence in $(0,1)$ and $\sigma_{n}$ is a sequence of real number such that;
\begin{itemize}
\item[(i)] $\underset{n\rightarrow\infty}{\lim}\gamma_{n}=0~~  {\rm and}~~ \sum \gamma_{n}=\infty;$
\item[(ii)]  $\underset{n\rightarrow\infty}{\lim}\frac{\sigma_{n}}{\gamma_{n}}\leq 0$ \ {\rm or}\ $\sum |\sigma_{n}|<\infty.$
  {\rm Then} $\underset{n\rightarrow\infty}{\lim}a_{n}=0.$
\end{itemize}
\end{lemma}
\noindent For the proof, see Xu \cite{xu2002iterative}.

\begin{lemma}\normalfont\label{2.2s}
 Let $\{x_{n}\}, \{y_{n}\}, \{z_{n}\}$ be  sequences of nonnegative real numbers satisfying
\begin{equation*}
x_{n+1}\leq(1+z_{n})x_{n} + y_{n}.
\end{equation*}
If $\sum z_{n}<\infty$ and $\sum y_{n}<\infty$, then  $\underset{n\rightarrow\infty}{\lim}x_{n}$ exist.
\end{lemma}
For  the proof of this lemma, see Tan and Xu \cite{xu1993approximating}.

\begin{lemma}\normalfont\label{df}
Let $\{x_{n}\}$ be a Fejer monotone with respect to $C,$ then the following are satisfied:
\begin{enumerate}
\item[$(i)$] $x_{n}\rightharpoonup x^{*}\in C$ if and only if $\omega_{\omega}\subset C;$
\item[$(ii)$]  $\{P_{C}x_{n }\}$ converges strongly to some vector in C;
\item[$(iii)$] if $x_{n}\rightharpoonup x^{*}\in C$, then $x^{*}=\underset{n\rightarrow \infty}{\lim} P_{C}x_{n }$.
\end{enumerate}
\end{lemma}
For the proof, see Bauschke and Borwein \cite{bauschke1996projection}.

\subsection{Problem Formulation}
 The SCFPP is formulated as follows:
\begin{align}
{\rm Find~~} x^{*}\in C:=\bigcap_{i=1}^{N}Fix(T_{i}) {\rm ~such ~that~} Ax^{*}\in Q:=\bigcap_{j=1}^{M}Fix(G_{j})\label{4.1zy}.
\end{align}
In this chapter, we consider $T_{i}: H_{1}\to H_{1} ,$ for $i=1,2,3,...,N$ and $G_{j}: H_{2}\to H_{2},$ for $j=1,2,3,...,M,$ to be total quasi-asymptotically nonexpansive  and or demicontractive mappings.\\

 We denote the solution set of SCFPP (\ref{4.1zy})  by
 \begin{align}
\Gamma=\left\{x^{*}\in C{\rm ~such~ that~} Ax^{*}\in Q\right\}.\label{4.2@}
\end{align}
 In sequel, we assume that $\Gamma\neq\emptyset.$

\subsection{Preliminary Results}
A Banach space E  satisfies {\it Opial's condition} (see Opial \cite{opial1967weak}), if for
any sequence $\{x_{n}\}$ in E such that $x_{n}\rightharpoonup x,~{\rm as}~n\to\infty$ implies that
\begin{displaymath}
\liminf_{n \rightarrow\infty}\Vert x_{n}-x\Vert <\liminf_{n
\rightarrow \infty} \Vert x_{n}-y \Vert, ~~ \forall y\in E , y \ne x.
\end{displaymath}
T it is said to have {\it Kadec-Klee property} (see Opial \cite{opial1967weak}), if for
any sequence $\{x_{n}\}$ in E such that $x_{n}\rightharpoonup x$ and $\left\|x_{n}\right\|\rightarrow\left\|x\right\|,~{\rm as}~n\to\infty$ implies that
\begin{displaymath}
x_{n}\rightarrow x  ~{\rm and}~{\rm as}~n\to\infty.
\end{displaymath}

\begin{remark}\normalfont
Each Hilbert space satisfies the  Opial and Kadec-Klee's properties.
\end{remark}
The following lemma were taken from  Wang et al., \cite{wang2012split}, we include the proof here for the sake of completeness.
\begin{lemma}\normalfont\label{2.1}
Let $G:H_{1}\rightarrow H_{1}$  be a $(\{v_{n}\},\{\mu_{n}\},\xi)$-total quasi-asymptotically nonexpansive mapping with $Fix(G)\neq\emptyset.$ Then, for each $y\in Fix(G),$   $x\in H_{1}$ and  $n\geq 1$, the following inequalities are equivalent:
\begin{align}
\left\|G^{n}x-y\right\|^{2}&\leq \left\|x-y\right\|^{2}+v_{n}\xi(\left\|x-y\right\|)+\mu_{n};\label{4.3}
\\2 \left\langle x-G^{n}x, x-y \right\rangle &\geq \left\|G^{n}x-x\right\|^{2}-v_{n}\xi(\left\|x-y\right\|)-\mu_{n};\label{4.4}
\\2 \left\langle x-G^{n}x, y-G^{n}x \right\rangle&\leq\left\|G^{n}x-x\right\|^{2}+v_{n}\xi(\left\|x-y\right\|)+\mu_{n}.\label{4.5}
\end{align}
\end{lemma}

\begin{proof}

$(i)\Rightarrow (ii)$
\begin{align}
\left\|G^{n}x-y\right\|^{2}&=\left\|G^{n}x-x+x-y\right\|^{2}\nonumber
\\&=\left\|G^{n}x-x\right\|^{2}+2\left\langle G^{n}x-x, x-y\right\rangle+\left\|x-y\right\|^{2},\nonumber
\end{align}
this imply that
\begin{align}
2\left\langle G^{n}x-x, x-y\right\rangle&=\left\|G^{n}x-y\right\|^{2}-\left\|G^{n}x-x\right\|^{2}-\left\|x-y\right\|^{2}\nonumber
\\&\leq \left\|x-y\right\|^{2}+v_{n}\xi(\left\|x-y\right\|)+\mu_{n}-\left\|G^{n}x-x\right\|^{2}-\left\|x-y\right\|^{2}.\nonumber
\end{align}
Thus, we deduce that
\begin{align}
2\left\langle x-G^{n}x, x-y\right\rangle&\geq \left\|G^{n}x-x\right\|^{2}-v_{n}\xi(\left\|x-y\right\|)-\mu_{n}.\nonumber
\end{align}
$(ii)\Rightarrow (iii)$
\begin{align}
\left\langle x-G^{n}x, x-y\right\rangle&=\left\langle x-G^{n}x, x-G^{n}x+G^{n}x-y\right\rangle\nonumber
\\&=\left\langle x-G^{n}x, x-G^{n}x\right\rangle+\left\langle x-G^{n}x, G^{n}x-y\right\rangle.\nonumber
\end{align}
This tends to imply that
\begin{align}
\left\langle x-G^{n}x, G^{n}x-y\right\rangle&=-\left\| x-G^{n}x\right\|^{2}+\left\langle x-G^{n}x, x-y\right\rangle \nonumber
\\&\geq -\left\| x-G^{n}x\right\|^{2} +\frac{1}{2}\left\|G^{n}x-x\right\|^{2}\nonumber
\\&-\frac{1}{2}v_{n}\xi(\left\|x-y\right\|)-\frac{1}{2}\mu_{n}.\nonumber
\end{align}
Thus, we deduce that
\begin{align}
2\left\langle x-G^{n}x, y-G^{n}x\right\rangle&\leq \left\| x-G^{n}x\right\|^{2} +v_{n}\xi(\left\|x-y\right\|)+\mu_{n}.\nonumber
\end{align}
$(iii)\Rightarrow (i)$
\begin{align}2 \left\langle x-G^{n}x, y-G^{n}x \right\rangle&\leq\left\|G^{n}x-x\right\|^{2}+v_{n}\xi(\left\|x-y\right\|)+\mu_{n}\nonumber
\\&=\left\|G^{n}x-y\right\|^{2}+2 \left\langle G^{n}x-y, y-x \right\rangle+\left\|x-y\right\|^{2}\nonumber
\\&+v_{n}\xi(\left\|x-y\right\|)+\mu_{n},\nonumber
\end{align}
thus, we deduce that
\begin{align}
\left\|G^{n}x-y\right\|^{2}&\leq \left\|x-y\right\|^{2}+v_{n}\xi(\left\|x-y\right\|)+\mu_{n}.\nonumber
\end{align}
And thus  completes the  proof.  \end{proof}

\begin{lemma}(Mohammed and Kilicman \cite{lbm})\normalfont\label{fd}
Let   $P_{C}:H\to C$ be a metric projection   such that $$\left\langle x_{n}-x^{*}, x_{n}-P_{C}x_{n}\right\rangle\leq 0.$$ Then   for each $ n\geq 1,$
$$\left\|P_{C}x_{n}-x_{n}\right\|\leq\left\|P_{C}x_{n}-x^{*}\right\|,   \forall x^{*}\in C.$$
\end{lemma}

\begin{proof}

Let $x^{*}\in C$, then
\begin{align*}
\left\| x_{n} - P_{C}x_{n} \right\|^2
&=\left\|x_{n}-x^{*}+x^{*}-P_{C}x_{n}\right\|^2
\\&=\left\|x_{n}-x^{*}\right\|^{2}+\left\| x^{*}- P_{C}x_{n}\right\|^{2}
\\&+2\left\langle x_{n}-x^{*}, x^{*}-P_{C}x_{n}\right\rangle
\\&=\left\|x_{n}-x^{*}\right\|^{2}+\left\|x^{*}-P_{C}x_{n}  \right\|^{2}
\\&+2\left\langle x_{n}-x^{*}, x^{*}-x_{n}+x_{n}-P_{C}x_{n} \right\rangle
\\&=\left\|x_{n}-x^{*}\right\|^{2}+\left\|x^{*}-P_{C}x_{n}  \right\|^{2}-2\left\|x_{n}-x^{*}\right\|^{2}
\\&+2\left\langle x_{n}-x^{*}, x_{n}-P_{C}x_{n}\right\rangle
\\&= \left\| x^{*}- P_{C}x_{n}\right\|^{2}-\left\|x_{n}-x^{*} \right\|^{2}+2\left\langle x_{n}-x^{*}, x_{n}-P_{C}x_{n}\right\rangle
\\& \leq\left\|x^{*}-P_{C}x_{n}\right\|^2.
\end{align*}
Thus, we conclude that
 $$\left\| x_{n} - P_{C}x_{n} \right\|\leq\left\|x^{*}-P_{C}x_{n}\right\|.$$
\end{proof}

\subsection{Strong Convergence  for the Split Common Fixed Point Problems for Total Quasi-asymptotically Nonexpansive Mappings}

\begin{theorem}\normalfont\label{thm1}
Let  $G:H_{1}\rightarrow H_{1}$, $T:H_{2}\rightarrow H_{2}$ be $(\{v_{n_{1}}\},$ $\{\mu_{n_{1}}\},$ $\xi_{1})$, $(\{v_{n_{2}}\},$
$ \{\mu_{n_{2}}\},$  $\xi_{2})$-total quasi-asymptotically nonexpansive  mappings and  uniformly $L_{1}, L_{2}$-Lipschitzian continuous  mappings such that $(G-I)$ and $(T-I)$ are  demiclosed at zero. Let $A:H_{1}\rightarrow H_{2}$ be a bounded linear operator with its adjoint   $A^{*}.$  Also let  $M$ and $M^{*}$ be  positive constants such that $\xi(k)\leq\xi(M)+M^{*}k^{2},\forall k\geq 0.$  Assume that $\Gamma\neq\emptyset,$ and let $P_{\Gamma}$ be the metric projection of $H_{1}$ onto $\Gamma$ satisfying
$$\left\langle x_{n}-x^{*}, x_{n}-P_{\Gamma}x_{n}\right\rangle\leq 0.$$
Define a sequence  $\{x_{n}\}$  by
\begin{equation}
 {} \left\{\begin{array}{ll} x_{0} \in H_{1},
 \\ u_{n} =x_{n}+\gamma A^{*}(T^{n}-I)Ax_{n}, & \textrm{ $  $}
  \\ x_{n+1} = \alpha_{n}u_{n} + (1-\alpha_{n})G^{n}u_{n},\forall n\geq 0,\label{3.1}
\end{array}\right.
\end{equation}
where  $\gamma,$ $L$, $\{v_{n}\}$, $\{\mu_{n}\}$, $\{\xi_{n}\}$ and $\{\alpha_{n}\}$ satisfies the following conditions:
\begin{enumerate}
\item[(i)]$0<k<\alpha_{n}<1$, $\gamma\in (0,\frac{1}{L^{*}})$ with $L^{*}=\left\|AA^{*}\right\|$ {\rm ~and~} $L=\max{\{L_{1},L_{2}\}}$;
\item[(ii)] $v_{n}=\max{\{v_{n_{1}},v_{n_{2}}\}}$, $\mu_{n}=\max{\{\mu_{n_{1}},\mu_{n_{2}}\}}$ and $\xi=\max{\{\xi_{1},\xi_{2}\}}$.
\end{enumerate}
Then $x_{n}\to x^{*}\in\Gamma $.
\end{theorem}

\begin{proof}

To show that $x_{n}\rightarrow x^{*},$ as $n\rightarrow\infty,$ it suffices to show that
$$x_{n}\rightharpoonup x^{*} {\rm~and~} \left\|x_{n}\right\|\rightarrow \left\|x^{*}\right\|, {\rm~as~} n\rightarrow\infty.$$

 We divided the proof into five steps as follows:\\

\textbf{Step 1.} In this step, we show  that for each $x^{*}\in\Gamma,$ the following limit exists.
\begin{equation}
\underset{n\rightarrow\infty}{\lim}\left\|x_{n}-x^{*}\right\|=\underset{n\rightarrow\infty}{\lim}\left\|u_{n}-x^{*}\right\|.
\end{equation}
Now, let $x^{*}\in\Gamma.$  By (\ref{3.1}) and Lemma \ref{2.1}, we have
\begin{align}
\left\|x_{n+1}-x^{*}\right\|^{2}
&=\left\|\alpha_{n}u_{n} +(1-\alpha_{n})G^{n}u_{n}-x^{*}\right\|^{2}\nonumber
\\&=\left\|\alpha_{n}(u_{n}-G^{n}u_{n})\right\|^{2}+2\alpha_{n}\left\langle u_{n}-G^{n}u_{n}, G^{n}u_{n}-x^{*}\right\rangle+\left\|G^{n}u_{n}-x^{*}\right\|^{2}\nonumber
\\&=\alpha_{n}^{2}\left\|u_{n}-G^{n}u_{n}\right\|^{2}+2\alpha_{n}\left\langle u_{n}-x^{*}+x^{*}-G^{n}u_{n}, G^{n}u_{n}-x^{*}\right\rangle\nonumber
\\&+\left\|G^{n}u_{n}-x^{*}\right\|^{2}\nonumber
\\&=\alpha_{n}^{2}\left\|u_{n}-G^{n}u_{n}\right\|^{2}+2\alpha_{n}\left\langle u_{n}-x^{*}, G^{n}u_{n}-x^{*}\right\rangle\nonumber
\\&+(1-2\alpha_{n})\left\|G^{n}u_{n}-x^{*}\right\|^{2}\nonumber
\\&=\alpha_{n}^{2}\left\|u_{n}-G^{n}u_{n}\right\|^{2}+2\alpha_{n}\left\langle u_{n}-x^{*}, G^{n}u_{n}-u_{n}+u_{n}-x^{*}\right\rangle\nonumber
\\&+(1-2\alpha_{n})\left\|G^{n}u_{n}-x^{*}\right\|^{2}\nonumber
\\&=\alpha_{n}^{2}\left\|u_{n}-G^{n}u_{n}\right\|^{2}+2\alpha_{n}\left\langle u_{n}-x^{*}, G^{n}u_{n}-u_{n}\right\rangle \nonumber
\\&+2\alpha_{n}\left\langle u_{n}-x^{*}, u_{n}-x^{*}\right\rangle +(1-2\alpha_{n})\left\|G^{n}u_{n}-x^{*}\right\|^{2}\nonumber
\\&\leq -\alpha_{n}(1-\alpha_{n})\left\|u_{n}-G^{n}u_{n}\right\|^{2} +2\alpha_{n}\left\|u_{n}-x^{*}\right\|^{2}+\alpha_{n}v_{n}\xi(\left\|u_{n}-x^{*}\right\|) \nonumber
\\&+\alpha_{n}\mu_{n}+(1-2\alpha_{n})\Big(\left\|u_{n}-x^{*}\right\|^{2}+ v_{n}\xi(\left\|u_{n}-x^{*}\right\|) +\mu_{n}\Big) \nonumber
\\&\leq -\alpha_{n}(1-\alpha_{n})\left\|u_{n}-G^{n}u_{n}\right\|^{2} +\left\|u_{n}-x^{*}\right\|^{2} \nonumber
\\&+(1-\alpha_{n})\Big(v_{n}\xi(\left\|u_{n}-x^{*}\right\|) +\mu_{n}\Big) \nonumber
\\&=-\alpha_{n}(1-\alpha_{n})\left\|u_{n}-G^{n}u_{n}\right\|^{2}+ \Big(1+(1-\alpha_{n})v_{n}M^{*}\Big)\left\|u_{n}-x^{*}\right\|^{2} \nonumber
\\&+(1-\alpha_{n})\Big(v_{n}\xi (M)+\mu_{n}\Big).\label{(a)}
\end{align}
On the other hand,
\begin{align}
\left\|u_{n}-x^{*}\right\|^{2}&=\left\|x_{n}-x^{*}+\gamma A^{*}(T^{n}-I)Ax_{n}\right\|^{2}\nonumber
\\&=\left\|x_{n}-x^{*}\right\|^{2}+\gamma^{2}\left\|A^{*}(T^{n}-I)Ax_{n}\right\|^{2}\nonumber
\\&+2\gamma\left\langle x_{n}-x^{*}, A^{*}(T^{n}-I)Ax_{n}\right\rangle,\label{(bDD)}
\end{align}
and
\begin{align}
\gamma^{2}\left\| A^{*}(T^{n}-I)Ax_{n}\right\|^{2} &=\gamma^{2}\left\langle A^{*}(T^{n}-I)Ax_{n}, A^{*}(T^{n}-I)Ax_{n}\right\rangle\nonumber
\\&= \gamma^{2}\left\langle AA^{*}(T^{n}-I)Ax_{n}, (T^{n}-I)Ax_{n}\right\rangle\nonumber
\\&\leq \gamma^{2}L^{*}\left\| (T^{n}-I)Ax_{n}\right\|^{2}. \label{(cAW)}
\end{align}
By Lemma \ref{2.1}, it follows that
\begin{align}
2\gamma\left\langle x_{n}-x^{*}, A^{*}(T^{n}-I)Ax_{n}\right\rangle &=2\gamma\left\langle Ax_{n}-T^{n}Ax_{n}+T^{n}Ax_{n}-Ax^{*},  T^{n}Ax_{n}-Ax_{n}\right\rangle\nonumber
\\& = 2\gamma\left\langle T^{n}Ax_{n}-Ax^{*},  T^{n}Ax_{n}-Ax_{n}\right\rangle\nonumber
\\&-2\gamma\left\|(T^{n}-I)Ax_{n}\right\|^{2}\nonumber
\\&\leq \gamma v_{n}M^{*}L^{*}\left\|x_{n}-x^{*}\right\|^{2}+\gamma(v_{n}\xi(M)+\mu_{n})\nonumber
\\&-\gamma\left\|(T^{n}-I)Ax_{n}\right\|^{2}.\label{(dSD)}
\end{align}
Substituting  (\ref{(cAW)}) and (\ref{(dSD)}) into (\ref{(bDD)}), we obtain that
\begin{align}
\left\|u_{n}-x^{*}\right\|^{2}&\leq (1+\gamma v_{n}M^{*}L^{*})\left\| x_{n}-x^{*}\right\|^{2}\nonumber
\\&-\gamma(1-\gamma L)\left\| (T^{n}-I)Ax_{n}\right\|^{2} + \gamma (v_{n}\xi(M) +\mu_{n}).\label{(e)}
 \end{align}
By (\ref{(e)}) and (\ref{(a)}), we deduce that
 \begin{align}
\left\|x_{n+1}-x^{*}\right\|^{2}&\leq (1+(1-\alpha_{n})v_{n}M^{*})\Big( (1+\gamma v_{n}M^{*}L^{*})\left\| x_{n}-x^{*}\right\|^{2}\nonumber
\\&-\gamma(1-\gamma L^{*})\left\| (T^{n}-I)Ax_{n}\right\|^{2} + \gamma (v_{n}\xi(M) +\mu_{n})\Big)\nonumber
\\&-\alpha_{n}(1-\alpha_{n})\left\|x_{n}-G^{n}u_{n}\right\|^{2}+ (1-\alpha_{n})(v_{n}\xi (M)+\mu_{n})\nonumber
\\&=\big(1+(1-\alpha_{n})v_{n}M^{*}\big)\big(1+\gamma v_{n}M^{*}L\big)\left\| x_{n}-x^{*}\right\|^{2}\nonumber
\\& -\gamma(1-\gamma L^{*})(1+(1-\alpha_{n})v_{n}M^{*})\left\| (T^{n}-I)Ax_{n}\right\|^{2}\nonumber
\\&-\alpha_{n}(1-\alpha_{n})\left\|x_{n}-G^{n}u_{n}\right\|^{2}+(1+(1-\alpha_{n})v_{n}M^{*})\gamma (v_{n}\xi(M) +\mu_{n}) \nonumber
\\&+(1-\alpha_{n})(v_{n}\xi (M)+\mu_{n})\nonumber
\\&\leq\big(1+(1-\alpha_{n})v_{n}M^{*}\big)\big(1+\gamma v_{n}M^{*}L\big)\left\| x_{n}-x^{*}\right\|^{2}\nonumber
\\& -\gamma(1-\gamma L^{*})\left\| (T^{n}-I)Ax_{n}\right\|^{2}-\alpha_{n}(1-\alpha_{n})\left\|x_{n}-G^{n}u_{n}\right\|^{2}\nonumber
\\&+(1+(1-\alpha_{n})v_{n}M^{*})\gamma (v_{n}\xi(M) +\mu_{n}) \nonumber
\\&+(1-\alpha_{n})(v_{n}\xi (M)+\mu_{n}).\label{(f)}
\end{align}
Thus, we deduce that
\begin{align}
\left\|x_{n+1}-x^{*}\right\|^{2}&\leq \Big(1+\gamma v_{n}M^{*}L^{*} + (1-\alpha_{n})v_{n}M^{*}(1+\gamma v_{n}M^{*}L^{*})\Big)\left\| x_{n}-x^{*}\right\|^{2}\nonumber
\\& + (1+(1-\alpha_{n})v_{n}M^{*})\gamma (v_{n}\xi(M) +\mu_{n})+(1-\alpha_{n})(v_{n}\xi (M)+\mu_{n}).\nonumber
\end{align}
This implies that
\begin{align}
\left\|x_{n+1}-x^{*}\right\|^{2}&\leq (1+\beta_{n})\left\| x_{n}-x^{*}\right\|^{2} +\eta_{n},\label{(h)}
\end{align}
\begin{align*}
{\rm where~}\beta_{n}&= \gamma v_{n}M^{*}L^{*} + \big(1-\alpha_{n}\big)v_{n}M^{*}\big(1+\gamma v_{n}M^{*}L^{*}\big)
\\\eta_{n}&= \big(1+(1-\alpha_{n})v_{n}M^{*}\big)\gamma \big(v_{n}\xi(M) +\mu_{n}\big) +\big(1-\alpha_{n}\big)\big(v_{n}\xi (M)+\mu_{n}\big).
\end{align*}
Clearly,  $\sum{\beta_{n}}<\infty$ and $\sum{\eta_{n}}<\infty$. Moreover, $\beta_{n}\to 0$ and $\eta_{n}\to 0.$ Hence, by Lemma \ref{2.2s}, we conclude that $\underset{n\rightarrow\infty}{\lim}\left\|x_{n}-x^{*}\right\|$ exist.\\

We now prove that for each $x^{*}\in\Gamma$,  $\underset{n\rightarrow\infty}{\lim}\left\|u_{n}-x^{*}\right\|$ exist. \\

By (\ref{(f)}), we deduce that
\begin{align}
\gamma(1-\gamma L^{*})\left\| (T^{n}-I)Ax_{n}\right\|^{2}&\leq \left\| x_{n}-x^{*}\right\|^{2}-\left\|x_{n+1}-x^{*}\right\|^{2}\nonumber
 \\&+\beta_{n}\left\| x_{n}-x^{*}\right\|^{2} +\eta_{n}\label{(i)},
\end{align}
and
\begin{align}
\alpha_{n}(1-\alpha_{n})\left\|u_{n}-G^{n}u_{n}\right\|^{2}&\leq \left\| x_{n}-x^{*}\right\|^{2}-\left\|x_{n+1}-x^{*}\right\|^{2}\nonumber
 \\&+\beta_{n}\left\| x_{n}-x^{*}\right\|^{2} +\eta_{n}\label{(ii)}.
\end{align}
Thus, as $n\to\infty,$ we deduce from  (\ref{(i)}) and (\ref{(ii)})  that
\begin{align}
\underset{n\rightarrow\infty}{\lim}\left\|u_{n}-G^{n}u_{n}\right\|=0  {\rm ~and~} \underset{n\rightarrow\infty}{\lim}\left\|Ax_{n}-T^{n}Ax_{n}\right\|=0. \label{(j)}
\end{align}
Given (\ref{(e)}), (\ref{(j)}) and the fact that  $\underset{n\rightarrow\infty}{\lim}\left\|x_{n}-x^{*}\right\|$ exists, we obtain that $$\underset{n\rightarrow\infty}{\lim}\left\|u_{n}-x^{*}\right\| {\rm~exist}.$$

Moreover, by (\ref{(a)}) and (\ref{(e)}), we deduce that

\begin{align}\left\|x_{n+1}-x^{*}\right\|^{2}&\leq \Big(1+(1-\alpha_{n})v_{n}M^{*}\Big)\left\|u_{n}-x^{*}\right\|^{2} \nonumber
\\&+(1-\alpha_{n})\Big(v_{n}\xi (M)+\mu_{n}\Big),\label{hh1}
\end{align}
and
\begin{align}
\left\|u_{n}-x^{*}\right\|^{2}&\leq (1+\gamma v_{n}M^{*}L^{*})\left\| x_{n}-x^{*}\right\|^{2}\nonumber
\\&+ \gamma (v_{n}\xi(M) +\mu_{n}).\label{hh2}
 \end{align}
The fact that  $\underset{n\rightarrow\infty}{\lim}\left\|x_{n}-x^{*}\right\|$ and $\underset{n\rightarrow\infty}{\lim}\left\|u_{n}-x^{*}\right\|$  exists, it follows from (\ref{hh1}) and (\ref{hh2}) that

$$\underset{n\rightarrow\infty}{\lim}\left\|u_{n}-x^{*}\right\| =\underset{n\rightarrow\infty}{\lim}\left\|x_{n}-x^{*}\right\|.$$

\textbf{Step 2.} In this step, we show that

\begin{equation}
\underset{n\rightarrow\infty}{\lim}\left\|x_{n+1}-x_{n}\right\|=0 {\rm ~and~} \underset{n\rightarrow\infty}{\lim}\left\|u_{n+1}-u_{n}\right\|=0.
\end{equation}
By (\ref{3.1}), we have that
\begin{align}
\left\|x_{n+1}-x_{n}\right\| &=\left\|\alpha_{n}u_{n} +(1-\alpha_{n})G^{n}u_{n}-x_{n}\right\|\nonumber
\\&= \left\|(1-\alpha_{n})(G^{n}u_{n}-u_{n})+u_{n}-x_{n}\right\|\nonumber
\\&= \left\|(1-\alpha_{n})(G^{n}u_{n}-u_{n})+A^{*}(T^{n}-I)Ax_{n}\right\|.\label{4.20}
\end{align}
In view of (\ref{(j)}), we deduce from (\ref{4.20}) that
\begin{align}
\underset{n\rightarrow\infty}{\lim}\left\|x_{n+1}-x_{n}\right\|=0.\label{(k)}
\end{align}
On the other hand,
\begin{align*}
\left\|u_{n+1}-u_{n}\right\| &=\left\|(I+\gamma A^{*}(T^{n+1}-I)A)x_{n+1}+ (I+\gamma A^{*}(T^{n}-I)A)x_{n}\right\|
\\&=\left\|x_{n+1}-x_{n} +\gamma A^{*}(T^{n+1}-I)Ax_{n+1}- \gamma A^{*}(T^{n}-I)Ax_{n}\right\|.
\end{align*}
Thus, by (\ref{(j)}) and (\ref{(k)}) we obtain that
$$\underset{n\rightarrow\infty}{\lim}\left\|u_{n+1}-u_{n}\right\|=0.$$

\textbf{Step 3.} In this step,  we show that

\begin{equation}
\underset{n\rightarrow\infty}{\lim}\left\|u_{n}-Gu_{n}\right\|=0 {\rm~and~}  \underset{n\rightarrow\infty}{\lim}\left\|Ax_{n}-Tx_{n}\right\|=0.\label{eq2}
\end{equation}

The fact that  $\underset{n\rightarrow\infty}{\lim}\left\|u_{n}-G^{n}u_{n}\right\|=0,$  $\underset{n\rightarrow\infty}{\lim}\left\|u_{n+1}-u_{n}\right\|= 0 $ and $G$ is uniformly L-Lipschitzian mapping, we have  that
\begin{align*}
\left\|u_{n}-Gu_{n}\right\|&\leq \left\|u_{n}-G^{n}u_{n}\right\|+ \left\|Gu_{n}-G^{n}u_{n}\right\|
\\&\leq \left\|u_{n}-G^{n}u_{n}\right\| + L\left\|u_{n}-G^{n-1}u_{n}\right\|
\\&\leq \left\|u_{n}-G^{n}u_{n}\right\|+ L\left\|G^{n-1}u_{n}-G^{n-1}u_{n-1}\right\|
\\&+L\left\|u_{n}-G^{n-1}u_{n-1}\right\|
\\&\leq \left\|u_{n}-G^{n}u_{n}\right\|+ L^{2}\left\|u_{n}-u_{n-1}\right\|
\\&+L\left\|u_{n}-u_{n-1}+u_{n-1}-G^{n-1}u_{n-1}\right\|
\\&\leq \left\|u_{n}-G^{n}u_{n}\right\|+L(L+1)\left\|u_{n}-u_{n-1}\right\|
\\& +L\left\|u_{n-1}-G^{n-1}u_{n-1}\right\|.
\end{align*}
Thus, as $n\rightarrow\infty,$ we have that
$$\underset{n\rightarrow\infty}{\lim} \left\|u_{n}-Gu_{n}\right\|= 0.$$

Similarly, from the fact that,  $\underset{n\rightarrow\infty}{\lim}\left\|Ax_{n}-T^{n}Ax_{n}\right\|= 0,$ $\underset{n\rightarrow\infty}{\lim}\left\|x_{n+1}-x_{n}\right\|= 0 $ and $T$ is uniformly L-Lipschitzian mapping, we deduce  that
$$\underset{n\rightarrow\infty}{\lim}\left\|Ax_{n}-TAx_{n}\right\|= 0.$$

\textbf{Step 4.} In this step, we show that

\begin{equation}
x_{n}\rightharpoonup x^{*} {\rm~and~} u_{n}\rightharpoonup x^{*} {\rm ~as~} n\to\infty. \label{eq1}
\end{equation}

Since $\{u_{n}\}$ is bounded, then there exists a sub-sequence  $u_{n_{i}}\subset u_{n}$ such that
\begin{equation}
u_{n_{i}}\rightharpoonup x^{*}, {\rm~as~} i\rightarrow\infty. \label{law}
\end{equation}
 By (\ref{eq2}) and (\ref{law}), we have that
\begin{equation}
\underset{i\rightarrow\infty}{\lim}\left\|u_{n_{i}}-Gu_{n_{i}}\right\|= 0. \label{l}
\end{equation}
From (\ref{law}), (\ref{l}) and the fact that  $(G-I)$ is demiclosed at zero, we have that $x^{*}\in Fix(G)$.
By (\ref{3.1}), (\ref{law}) and the fact $\underset{n\rightarrow\infty}{\lim}\left\|Ax_{n}-T^{n}Ax_{n}\right\|= 0,$ we deduce that  $$x_{n_{i}}=u_{n_{i}}-\gamma A^{*}(T^{n_{i}}-I)Ax_{n_{i}}\rightharpoonup x^{*}.$$
By the definition of $A$, we get
\begin{equation}
Ax_{n_{i}}\rightharpoonup Ax^{*} {\rm~as~} i\to\infty.\label{m}
\end{equation}
In view of (\ref{eq2}), we get
 \begin{equation}
\underset{i\rightarrow\infty}{\lim}\left\|Ax_{n_{i}}-TAx_{n_{i}}\right\|= 0.\label{n}
 \end{equation}
From (\ref{m}), (\ref{n}) and the fact that   $(T-I)$ is demiclosed at zero, we have that $Ax^{*}\in Fix(T).$ Thus, $x^{*}\in Fix(G)$ and $Ax^{*}\in Fix(T),$ and this implies that $x^{*}\in\Gamma$.\\

Now, we show that   $x^{*}$  is  unique. Suppose to the contrary that there exists another sub-sequence $u_{n{j}}\subset u_{n}$  such that $u_{n{j}}\rightharpoonup y^{*}\in \Gamma$ with $x^{*} \neq y^{*}.$ By  opial's property of Hilbert space, we have
that
\begin{align*}
\underset{j\rightarrow\infty}{\liminf}\left\|u_{n_{j}}-x^{*}\right\|
&<\underset{j\rightarrow\infty}{\liminf} \left\|u_{n_{j}}-y^{*}\right\|
\\&= \underset{n\rightarrow\infty}{\liminf}\left\|u_{n}-y^{*}\right\|
\\&=\underset{j\rightarrow\infty}{\liminf}\left\|u_{n_{j}}-y^{*}\right\|
 \\&< \underset{j\rightarrow\infty}{\liminf} \left\|u_{n_{j}}-x^{*}\right\|
\\&= \underset{n\rightarrow\infty}{\liminf}\left\|u_{n}-x^{*}\right\|
\\&= \underset{j\rightarrow\infty}{\liminf}\left\|u_{n_{j}}-x^{*}\right\|.
\end{align*}

Thus,  we have
$$\underset{j\rightarrow\infty}{\liminf}\left\|u_{n_{j}}-x^{*}\right\|<\underset{j\rightarrow\infty}{\liminf}\left\|u_{n_{j}}-x^{*}\right\|.
$$
This  is a contradiction, therefore, $u_{n}\rightharpoonup x^{*}$. By using (\ref{3.1}) and (\ref{(j)}) , we have
\begin{equation*}
x_{n}=u_{n}-\gamma A^{*}(T^{n}-I)Ax_{n} \rightharpoonup x^{*}, {\rm ~as~} n\to\infty.
\end{equation*}

\textbf{Step 5.} In this step, we show that
\begin{equation}
\left\|x_{n}\right\|\rightarrow  \left\|x^{*}\right\|, {\rm ~as~} n\to\infty.\label{eq3}
\end{equation}
To show this, it suffices to show that $\left\|x_{n+1}\right\|\rightarrow \left\|x^{*}\right\| {\rm ~as~} n\to\infty.$\\

By Equation  (\ref{(h)}),  Lemma \ref{fd} and \ref{df}, and the fact that $\beta_{n}\to 0$ and $\eta_{n}\to 0$, we have
\begin{align*}
\Big|\left\|x_{n+1}\right\|-\left\|x^{*}\right\|\Big|^{2}
&\leq \left\|x_{n+1}-x^{*}\right\|^{2}
\\&\leq (1+\beta_{n})\left\| x_{n}-x^{*}\right\|^{2} +\eta_{n}
\\&=\left\| x_{n}-x^{*}\right\|^{2}+\beta_{n}\left\| x_{n}-x^{*}\right\|^{2}+\eta_{n}
\\&=\left\| x_{n}-P_{\Gamma}x_{n}+P_{\Gamma}x_{n}-x^{*}\right\|^{2}+\beta_{n}\left\| x_{n}-x^{*}\right\|^{2}+\eta_{n}
\\&\leq 4\left\| P_{\Gamma}x_{n}-x^{*}\right\|^{2}+\beta_{n}\left\| x_{n}-x^{*}\right\|^{2}+\eta_{n}.
\end{align*}
Thus, as $n\to\infty,$ we have that
\begin{align*}\underset{n\rightarrow\infty}{\lim}\Big|\left\|x_{n+1}\right\|-\left\|x^{*}\right\|\Big|^{2}
&\leq 4\underset{n\rightarrow\infty}{\lim}\left\| P_{\Gamma}x_{n}-x^{*}\right\|^{2}
\\&+\underset{n\rightarrow\infty}{\lim} \beta_{n}\left\| x_{n}-x^{*}\right\|^{2} +\underset{n\rightarrow\infty}{\lim}(\eta_{n}).
\end{align*}
And this implies that
$$ \underset{n\rightarrow\infty}{\lim}\Big|\left\|x_{n+1}\right\|-\left\|x^{*}\right\|\Big|= 0.$$

By  (\ref{eq1}) and (\ref{eq3}), we conclude that $x_{n}\rightarrow x^{*}, {\rm ~as~} n\to\infty.$

\end{proof}

\subsection{Strong Convergence  for the Split Common Fixed Point Problems for  Demicontractive Mappings }
In this section, we considered an algorithm for solving the SCFPP for  demicontractive mappings without any prior information on the norm on the bounded linear operator and established the strong convergence results of the proposed algorithm. In the end, we provides some special cases of our suggested methods.

\begin{theorem}\normalfont\label{maths}
Let  $U:H_{1}\rightarrow H_{1}$ and $T:H_{2}\rightarrow H_{2}$ be $k_{1},k_{2}-$demicontractive mappings  such that $(U-I)$ and $(T-I)$ are demiclosed at zero, 
 $A:H_{1}\rightarrow H_{2}$ be a bounded linear operator with its adjoint  $A^{*}.$ Assume that  $\Gamma\neq\emptyset$ and let $P_{\Gamma}$ be a metric projection from $H_{1}$ onto $\Gamma$ satisfying
$$\left\langle x_{n}-x^{*}, x_{n}-P_{\Gamma}x_{n}\right\rangle\leq 0.$$

Define   $\{x_{n}\}$  by
\begin{equation}
 {} \left\{ \begin{array}{ll} x_{0}\in H_{1} {\rm ~is~ arbitrary~ chosen},
\\ u_{n}=x_{n}+\rho_{n}A^{*}(T-I)Ax_{n},
\\ x_{n+1}=(1-\alpha_{n})u_{n}+\alpha_{n}Uu_{n}, \forall n\geq 0,& \textrm{ $  $}
 \end{array}  \right.\label{5}
\end{equation}
  where $0<c<\alpha_{n}<1-k,$ with $k:=\max\{k_{1}, k_{2}\},$  and
\begin{equation}
 \rho_{n}={} \left\{ \begin{array}{ll} \frac{(1-k)\left\|(I-T)Ax_{n}\right\|^{2}}{2\left\|A^{*}(I-T)Ax_{n}\right\|^{2}},TAx_{n}\neq Ax_{n},
\\\hspace{0.2cm} 0,\hspace{2.3cm} {\rm~~otherwise}. & \textrm{ $  $}
 \end{array}  \right.\label{6}
\end{equation}
Then $x_{n}\to x^{*}\in\Gamma$.
\end{theorem}

\begin{proof}

To show that $x_{n}\rightarrow x^{*}$, it suffices to show $x_{n}\rightharpoonup x^{*}$ and $\left\|x_{n}\right\|\rightarrow \left\|x^{*}\right\|$.\\

We divided the proof into four steps as follows.\\

\textbf{Step 1.} In this step, we show that $\{x_{n}\}$ is a Fejer monotone.  This is divided into two cases.\\

Case 1. If $\rho_{n}= 0$ and   Case 2. If $\rho_{n}\neq 0$.\\

Now, let $x^{*}\in\Gamma$.\\

\textbf{Case 1.} If $\rho_{n}=0.$ The fact that $U$ is demicontractive, we have
\begin{eqnarray}
\left\|x_{n+1}-x^{*}\right\|^{2}&=& \left\|x_{n}-x^{*}+\alpha_{n}(Ux_{n}-x_{n})\right\|^{2}\nonumber
\\&=&\left\|x_{n}-x^{*}\right\|^{2}+2\alpha_{n}\left\langle x_{n}-x^{*}, Ux_{n}-x_{n}\right\rangle
+\alpha_{n}^{2}\left\|Ux_{n}-x_{n}\right\|^{2}\nonumber
\\&\leq&\left\|x_{n}-x^{*}\right\|^{2}+\alpha_{n}(k-1)\left\|Ux_{n}-x_{n}\right\|^{2}+\alpha_{n}^{2}\left\|Ux_{n}-x_{n}\right\|^{2}\nonumber
\\&\leq&\left\|x_{n}-x^{*}\right\|^{2}-\alpha_{n}\big(1-k-\alpha_{n}\big)\left\|Ux_{n}-x_{n}\right\|^{2}\label{01}.
\end{eqnarray}
The fact that $0<\alpha_{n}<1-k$, it follows from (\ref{01}) that $\{x_{n}\}$ is Fejer monotone. \\

\textbf{Case 2.} If $\rho_{n}\neq 0.$ Since $U$ and $T$ are  demicontractive mappings, we have
\begin{eqnarray}
\left\|x_{n+1}-x^{*}\right\|^{2}&= &\left\|u_{n}-\alpha_{n}u_{n}+\alpha_{n}Uu_{n}-x^{*}\right\|^{2}\nonumber
\\&=&\left\|u_{n}-x^{*}\right\|^{2}+2\alpha_{n}\left\langle u_{n}-x^{*}, Uu_{n}-u_{n}\right\rangle
+\alpha_{n}^{2}\left\|Uu_{n}-u_{n}\right\|^{2}\nonumber
\\&\leq&\left\|u_{n}-x^{*}\right\|^{2} -\alpha_{n}(1-k)\left\|Uu_{n}-u_{n}\right\|^{2}+\alpha_{n}^{2}\left\|Uu_{n}-u_{n}\right\|^{2}\nonumber
\\&\leq&\left\|u_{n}-x^{*}\right\|^{2}-\alpha_{n}\big(1-k-\alpha_{n}\big)\left\|Uu_{n}-u_{n}\right\|^{2}\label{03}.
\end{eqnarray}
On the other hand,
\begin{eqnarray}
\left\|u_{n}-x^{*}\right\|^{2}&=&\left\|x_{n}+\rho_{n}A^{*}(T-I)Ax_{n}-x^{*}\right\|^{2}\nonumber
\\&=&\left\|x_{n}-x^{*}\right\|^{2}+2\rho_{n}\left\langle TAx_{n}-Ax_{n}, Ax_{n}-Ax^{*}\right\rangle+\rho^{2}_{n}\left\|A^{*}(I-T)Ax_{n}\right\|^{2}\nonumber
\\&\leq& \left\|x_{n} - x^{*}\right\|^{2}-\rho_{n}(1-k)\left\|(T-I)Ax_{n}\right\|^{2}+\rho^{2}_{n}\left\|A^{*}(I-T)Ax_{n}\right\|^{2}\nonumber
\\&=& \left\|x_{n} - x^{*}\right\|^{2}-\frac{(1-k)^{2}\left\|(I-T)Ax_{n}\right\|^{2}}{2\left\|A^{*}(I-T)Ax_{n}\right\|^{2}}\left\|(T-I)Ax_{n}\right\|^{2}\nonumber
\\&+&\frac{(1-k)^{2}\left\|(I-T)Ax_{n}\right\|^{4}}{4\left\|A^{*}(I-T)Ax_{n}\right\|^{4}}\left\|A^{*}(I-T)Ax_{n}\right\|^{2}\nonumber
\\&=&\left\|x_{n}-x^{*}\right\|^{2}-\frac{(1-k)^{2}}{4}\frac{\left\|(T-I)Ax_{n}\right\|^{4}}{\left\|A^{*}(T-I)Ax_{n}\right\|^{2}}.\label{02}
\end{eqnarray}
Substituting (\ref{02}) into (\ref{03}), we deduce  that
\begin{eqnarray}
\left\|x_{n+1}-x^{*}\right\|^{2}&\leq&\left\|x_{n}-x^{*}\right\|^{2}-\frac{(1-k)^{2}\left\|(T-I)Ax_{n}\right\|^{4}}{4\left\|A^{*}(T-I)Ax_{n}\right\|^{2}} \nonumber
\\&-& \alpha_{n}\big(1-k-\alpha_{n}\big)\left\|Uu_{n}-u_{n}\right\|^{2}.\label{04}
\end{eqnarray}
Thus, $\{x_{n}\}$ is Fejer monotone. Therefore,   $\underset{n\to\infty}{\lim}\left\|x_{n}-x^{*}\right\|$ exist.\\

\textbf{Step 2.} In this step, we show that
\begin{equation}
\lim_{n\to\infty} \left\|(I-T)Ax_{n}\right\|=0 {\rm ~and~} \lim_{n\to\infty}\left\|(I-U)x_{n}\right\|=0\label{eqn1}.
\end{equation}

\textbf{Case 1.} If $\rho_{n}=0.$ By (\ref{6}), we see that $\underset{n\to\infty}{\lim}\left\|(I-T)Ax_{n}\right\|=0.$ Also by
 (\ref{01}) and the fact   $\underset{n\to\infty}{\lim}\left\|x_{n}-x^{*}\right\|$ exist, it follows that
$\underset{n\to\infty}{\lim}\left\|(I-U)x_{n}\right\|=0.$\\

\textbf{Case 2.} If $\rho_{n}\neq 0.$ The fact   $\underset{n\to\infty}{\lim}\left\|x_{n}-x^{*}\right\|$ exist, by   (\ref{04}), we deduce  that
\begin{equation}
\lim_{n\to\infty}\left(\frac{(1-k)^{2}\left\|(T-I)Ax_{n}\right\|^{4}}{4\left\|A^{*}(T-I)Ax_{n}\right\|^{2}}\right)\leq \lim_{n\to\infty}\left( \left\|x_{n}-x^{*}\right\|-\left\|x_{n+1}-x^{*}\right\|\right)= 0,\label{10a}
\end{equation}

\begin{equation}
{\rm ~and~}\lim_{n\to\infty} \left\|(I-U)u_{n}\right\|\leq \lim_{n\to\infty}\left(\frac{\left\|x_{n}-x^{*}\right\|-\left\|x_{n+1}-x^{*}\right\|}{c\big(1-k-\alpha_{n}\big)}\right)= 0.\label{10}
\end{equation}
By (\ref{10a}), we have that
\begin{equation}
\lim_{n\to\infty}\left(\frac{\left\|(T-I)Ax_{n}\right\|^{2}}{\left\|A^{*}(T-I)Ax_{n}\right\|}\right)=0.\label{ert}
\end{equation}
On the other hand,
\begin{eqnarray}
\left\|(T-I)Ax_{n}\right\|&=&\left\|A\right\|\frac{\left\|TAx_{n}-Ax_{n}\right\|^{2}}{\left\|A\right\|\left\|TAx_{n}-Ax_{n}\right\|}\nonumber
\\&\leq&\left\|A\right\|\frac{\left\|TAx_{n}-Ax_{n}\right\|^{2}}{\left\|A^{*}(T-I)Ax_{n}\right\|}\nonumber.
\end{eqnarray}
Thus, by (\ref{ert}) we deduce that
$$\lim_{n\to\infty}\left\|(T-I)Ax_{n}\right\|=0.$$
Since
\begin{eqnarray}
\rho_{n}\left\|A^{*}(T-I)Ax_{n}\right\|&=&\left\|u_{n}-x_{n}\right\|\nonumber
\\&=&\frac{(1-k)\left\|(T-I)Ax\right\|^{2}}{2\left\|A^{*}(T-I)Ax_{n}\right\|}.\nonumber
\end{eqnarray}
Thus, by (\ref{ert}) we have
\begin{eqnarray}
\lim_{n\to\infty}\rho_{n}\left\|A^{*}(T-I)Ax_{n}\right\|=0.\label{21}
\end{eqnarray}
By (\ref{5}), we have that

\begin{align}
\left\|(U-I)x_{n}\right\|&=\left\|(U-I)u_{n}-(U-I)\rho_{n}A^{*}(T-I)Ax_{n}\right\|\nonumber
\\&\leq\left\|(U-I)u_{n}\right\|+\left\|(U-I)\rho_{n}A^{*}(T-I)Ax_{n}\right\|.\label{nonumber}
\end{align}
	Given   (\ref{10}), (\ref{21}) and the fact that $U$ is bounded. It follows from (\ref{nonumber})  that
 \begin{equation*}
\lim_{n\to\infty}\left\|(I-U)x_{n}\right\|=0.\label{11.}
\end{equation*}
Hence, in both case, Equation (\ref{eqn1}) hold.\\

\textbf{step 3.} In this step, we show that
\begin{equation}
x_{n}\rightharpoonup x^{*}, {\rm ~as~} n\to \infty.\label{b.}
\end{equation}
To show this, it suffices to show that $\omega_{\omega} \subseteq \Gamma,$ see Lemma \ref{df} (i).\\

Now, let $q\in\omega_{\omega}$, this implies that, there exists $\{x_{n_{j}}\}$ of $\{x_{n}\}$ such that   $x_{n_{j}}\rightharpoonup q.$ Since $\underset{j\to\infty}{\lim}\left\|Ux_{n_{j}}-x_{n_{j}}\right\|= 0,$ together with the demiclosed  of $(U-I)$ at zero, we conclude that, $q\in Fix(U).$\\

 On the other hand, since $A$ is bounded, we have that   $Ax_{n_{j}}\rightharpoonup Aq.$  By (\ref{eqn1}) and together with the demiclosed  of $(T-I)$ at zero, we have that $Aq\in Fix(T).$ Thus,  $q\in\Gamma,$ this implies that $\omega_{\omega} \subseteq \Gamma.$ Hence, by Lemma \ref{df}, we conclude that $x_{n}\rightharpoonup x^{*},$ as $n\to \infty.$\\

\textbf{Step 4.} In this step, we show that
\begin{equation}
\left\|x_{n}\right\|\rightarrow  \left\|x^{*}\right\|, {\rm ~as~} n\to \infty\label{c}.
\end{equation}
To show (\ref{c}), it suffices to show that $\left\|x_{n+1}\right\|\rightarrow\left\|x^{*}\right\|$.\\

By Lemma \ref{l2} and the fact that $\{x_{n}\}$ is a Fejer monotone, we have
\begin{eqnarray}
\Big|\left\|x_{n+1}\right\|-\left\|x^{*}\right\|\Big|^{2}\nonumber
&\leq &\left\|x_{n+1}-x^{*}\right\|^{2}\nonumber
\\&\leq& \left\| x_{n}-x^{*}\right\|^{2} \nonumber
\\&=&\left\| x_{n}-P_{\Gamma}x_{n}+P_{\Gamma}x_{n}-x^{*}\right\|^{2}\nonumber
\\&\leq &4\left\| P_{\Gamma}x_{n}-x^{*}\right\|^{2}.\label{3.17}
\end{eqnarray}
Thus, we deduce that
$\left\|x_{n+1}\right\|\to\left\|x^{*}\right\|.$  By  Equation (\ref{b.}) and (\ref{c}), we conclude  that $x_{n}\to x^{*}$ as $n\to \infty.$
\end{proof}

\begin{corollary}\normalfont\label{3.2q}

Let   $G:H_{1}\to H_{1}$ and $T:H_{2}\to H_{2}$ be $(k_{n_{1}}, k_{n_{2}})$- quasi - asymptotically nonexpansive mappings such that $(G-I)$ and $(T-I)$ are  demiclosed at zero, and $A:H_{1}\to H_{2}$ be a  bounded linear operator with its adjoint $A^{*}.$ Also  let $L^{*}=\left\|AA^{*}\right\|,$ $M $ and $M^{*}$ be positive  constants such that $\xi(k)\leq \xi(M)+M^{*}k^{2},\forall k\geq 0 $. Assume that,    $\Gamma\neq\emptyset,$   and let $P_{\Gamma}$ be a metric projection of $H_{1}$ onto $\Gamma$ satisfying
$$\left\langle x_{n}-x^{*}, x_{n}-P_{\Gamma}x_{n}\right\rangle\leq 0.$$
Define  $\{x_{n}\}$  by
\begin{equation}
 {} \left\{\begin{array}{ll} \ x_{0} \in H_{1},
 \\ u_{n} = x_{n}+\gamma A^{*}(T^{n}-I)Ax_{n},& \textrm{ $  $}
  \\ x_{n+1} = \alpha_{n}u_{n} + (1-\alpha_{n})G^{n}u_{n},\forall n\geq 0,\label{3.1xxx}
\end{array}\right.
\end{equation}
where  $\alpha_{n}\subset(0,1),$ $\gamma\in (0,\frac{1}{L^{*}})$,  $L=\max{\{L_{1},L_{2}\}}$ and
 $k_{n}=\max{\{k_{n_{1}},k_{n_{2}}\}}.$ Then  $x_{n}\to x^{*}\in\Gamma $.

\begin{proof}

 $G$ and $T$ are $(\{v_{n}\}, \{\mu_{n}\}, \xi)-$ total quasi-asymptotically nonexpansive mappings with $\{v_{n}\}=\{k_{n}-1\}, \mu_{n}=0$ and $\xi(k)=k^{2}, \forall k\geq 0$. Moreover, $G$ and $T$ are uniformly $k_{n_{1}}, k_{n_{2}}-$ Lipschitzian mappings. Therefore, all the conditions in Theorem \ref{3.1} are satisfied. Hence, the conclusion of this corollary follows directly from Theorem \ref{3.1}.
\end{proof}
\end{corollary}

\begin{corollary}\normalfont
Let   $G:H_{1}\to H_{1} $ and $T:H_{2}\to H_{2}$ be two quasi-nonexpansive  mappings such that $(G-I)$ and $(T-I)$ are  demiclosed at zero. And let $A$ be a bounded linear operator with its adjoint $A^{*}.$  Assume that,  $\Gamma\neq\emptyset,$   and let $P_{\Gamma}$ be a metric projection of H onto $\Gamma$ satisfying
$$\left\langle x_{n}-x^{*}, x_{n}-P_{\Gamma}x_{n}\right\rangle\leq 0.$$
Define   $\{x_{n}\}$  by
\begin{equation}
 {} \left\{\begin{array}{ll} \ x_{0} \in H_{1};
 \\ u_{n} = x_{n}+\gamma A^{*}(T-I)Ax_{n}; & \textrm{ $  $}
  \\ x_{n+1} = \alpha_{n}u_{n} + (1-\alpha_{n})Gu_{n},\forall n\geq 0,\label{3.1xxxa}
\end{array}\right.
\end{equation}
where   $\{\alpha_{n}\}\subset(0,1)$ and $\gamma\in (0,\frac{1}{L^{*}})$ with $L^{*}=\left\|AA^{*}\right\|.$ Then  $x_{n}\to x^{*}\in\Gamma $.

\begin{proof}

$G$ and $T$ are $(1)-$quasi-asymptotically nonexpansive mappings. Moreover, $G$ and $T$ are  uniformly $1-$ Lipschitzian mappings.  Therefore, all the conditions of Corollary \ref{3.2q} are satisfied. Hence, the conclusions of this corollary follow directly from  Corollary \ref{3.2q}.

\end{proof}
\end{corollary}

\begin{corollary}\normalfont
Let $H_{1}, H_{2},$  $A$, $A^{*},$ $P_{\Gamma}$ and $\{x_{n}\}$ be as in Theorem \ref{maths}. Also let $U:H_{1}\rightarrow H_{1}$ and $T:H_{2}\rightarrow H_{2}$ be  quasi nonexpansive mappings  such that $(U-I)$ and $(T-I)$ are demiclosed at zero. Assume that $\Gamma\neq\emptyset$.  Then  $x_{n}\to x^{*}\in\Gamma$.

\begin{proof}

Since T is quasi-nonexpansive, clearly T is 0-demicontractive. Hence, all the hypothesis of Theorem \ref{maths} are satisfied. Therefore, the proof of this corollary follows trivially from Theorem \ref{maths}.
\end{proof}
\end{corollary}

\begin{corollary}\normalfont
Let $H_{1}, H_{2},$  $A$, $A^{*},$ $P_{\Gamma}$ and $\{x_{n}\}$ be as in Theorem \ref{maths}. Also let $U:H_{1}\rightarrow H_{1}$ and $T:H_{2}\rightarrow H_{2}$ be firmly quasi nonexpansive mappings  such that $(U-I)$ and $(T-I)$ are demiclosed at zero. And assume that $\Gamma\neq\emptyset$.  Then  $x_{n}\to x^{*}\in\Gamma$.
\end{corollary}

\begin{corollary}\normalfont
Let $H_{1}, H_{2},$  $A$, $A^{*},$ $P_{\Gamma}$ and $\{x_{n}\}$ be as in Theorem \ref{maths}. Also let $U:H_{1}\rightarrow H_{1}$ and $T:H_{2}\rightarrow H_{2}$ be  directed operators  such that $(U-I)$ and $(T-I)$ are demiclosed at zero and assume that $\Gamma\neq\emptyset$.  Then,  $x_{n}\to x^{*}\in\Gamma$.

\begin{proof}

 Since $T$ is directed operator, clearly $T$  is $(-1)-$demicontractive. Hence, all the hypothesis of Theorem \ref{maths} are satisfied.  Therefore, the proof of this corollary follows trivially from  Theorem \ref{maths}.

\end{proof}
\end{corollary}

\subsection{Application  to Variational Inequality Problems}

Let $T:C\to H_{1}$ be a nonlinear mapping. The variational inequality problem  with respect to $C$ consist as finding  a vector $x^{*}\in C$ such that
 \begin{equation}
 \langle Tx^{*}, x-x^{*}\rangle\geq 0,~\forall ~x\in C\label{3.1n}.
 \end{equation}
We denote the solution set of Variational Inequality Problem (\ref{3.1n}) by $VI(T,C).$

It is easy to see that
\begin{equation}
 {\rm~find~} x^{*}\in VI(T,C) {\rm~ if ~and~only~if~} x^{*}\in Fix(P_{C}(I-\beta T)),\label{3.1o}
 \end{equation}
where $P_{C}$ is the metric projection from $H_{1}$ onto $C$ and $\beta$ is a positive constant. \\

Let $Q:=Fix(P_{C}(I-\beta T))$ ( the fixed point set of $P_{C}(I-\beta T)$) and $A=I$ (the identity operator on $H_{1}$), then Equation (\ref{3.1n}) can be written as;
\begin{equation}
 {\rm~find~} x^{*}\in C {\rm~such~that~} Ax^{*}\in Q.\label{3.1p}
 \end{equation}

\subsection{On Synchronal Algorithms for Fixed and Variational Inequality Problems in Hilbert Spaces}
The aim of this section is to expand the general approximation method proposed by Tian and Di \cite{di} to the class of $(k, \{\mu_{n}\}, \{\xi_{n}\}, \phi)$- total asymptotically strict pseudocontraction and  uniformly M-Lipschitzian mappings to solve the fixed point problem as well as variational inequality problem in the frame work of Hilbert space. The results presented in this paper extend, improve and generalize several known results in the literature.

\subsection{Preliminaries}
 In the sequel we shall make use of the following lemmas in proving the main results of this section.

\begin{lemma}\cite{g}\normalfont Let $H$ be a Hilbert space, there hold the following identities;
\begin{enumerate}
\item[(i)] $\left\|x-y\right\|^{2}=\left\|x\right\|^{2}-\left\|y\right\|^{2}-2\left\langle x-y,y\right\rangle, \forall x,y\in H;$
\item[(ii)] $\left\|tx+(1-t)y\right\|^{2}=t\left\|x\right\|^{2}+(1-t)\left\|y\right\|^{2}-t(1-t)\left\|x-y\right\|^{2}, \forall t\in [0,1]{\rm~ and~}\\ x,y\in H$;
\item[(iii)] if $\{x_{n}\}$ is a sequence in $H$ such that $x_{n}\rightharpoonup z,$
 then
 $$\underset{n\rightarrow\infty}{\limsup}\left\|x_{n}-y\right\|^{2}=\underset{n\rightarrow\infty}{\limsup}\left\|x_{n}-z\right\|^{2}+\underset{n\rightarrow\infty}{\limsup}\left\|z-y\right\|^{2},\forall y\in H.$$
\end{enumerate}
\end{lemma}

\begin{lemma}\cite{li}\normalfont Let $C$ be a nonempty closed convex subset of a real Hilbert space $H$ and let  $T:C\rightarrow C$  be a  ($k,~ \{\mu_{n}\},~\{\xi_{n}\},~ \phi)$- total  asymptotically strict pseudocontraction mapping and uniformly L-Lipschitzian. Then $I-T$ is demiclosed at zero in the sense that if $\{x_{n}\}$ is a sequence in  $C$ such that  $x_{n}\rightharpoonup x^{*}$, and $\underset{n\rightarrow\infty}{\limsup}\left\|(T^{n}-I)x_{n}\right\|=0$, then $(T-I)x^{*}=0.$
\end{lemma}

\begin{lemma}\cite{di} Assume that $\{a_{n}\}$ is a sequence of nonnegative real number such that
\begin{equation*}
a_{n+1}\leq(1-\gamma_{n})a_{n}+\sigma_{n}, n\geq 0,
\end{equation*}
where $\gamma_{n}$ is a sequence in $(0,1)$ and $\sigma_{n}$ is a sequence of real number such that;
\begin{itemize}
\item[(i)] $\underset{n\rightarrow\infty}{\lim}\gamma_{n}=0~~  {\rm and}~~ \sum \gamma_{n}=\infty$;
\item[(ii)]  $\underset{n\rightarrow\infty}{\lim}\frac{\sigma_{n}}{\gamma_{n}}\leq 0$ \ {\rm or}\ $\sum |\sigma_{n}|<\infty.$ \ {\rm Then}\ $\underset{n\rightarrow\infty}{\lim}a_{n}=0.$
\end{itemize}
\end{lemma}

\begin{lemma}\label{2.4}\cite{di}\normalfont Let $F:H\rightarrow H $ be a $\eta$- strongly monotone and $L$-Lipschitzian operator with $L>0$ and $\eta>0$. Assume that $\displaystyle{0<\mu<\frac{2\eta}{L^{2}}}$, \ $\displaystyle{\tau =\mu \left(\eta-\frac{L^{2}\mu}{2}\right)}$ and $0<t<1$. Then $$\left\|(I-\mu tF)x-(I-\mu tF)y\right\|\leq (1-\tau t)\left\|x-y\right\|.$$
\end{lemma}

\begin{lemma}\normalfont (Bulama and Kilicman \cite{kbul})\label{2.5} \normalfont Let $S:C\rightarrow H$ be a uniformly $L$-Lipschitzian mapping with $L\in(0,1]$. Define  $T:C\rightarrow H$ by $T^{\beta_{n}}x=\beta_{n} x + (1-\beta_{n} )S^{n}x$ with $\beta_{n}\in(0,1)$ and  $\forall x\in C$.  Then $T^{\beta_{n}}$  is nonexpansive and $Fix(T^{\beta_{n}})= Fix(S^{n})$.
\begin{proof}Let $x,y\in C,$ from lemma (2.1(ii)), we have
\begin{align*}\left\|T^{\beta_{n}}x-T^{\beta_{n}}y\right\|^{2}&=\left\|\beta_{n}(x-y)+ (1-\beta_{n})(S^{n}x-S^{n}y)\right\|^{2}
\\&=\beta_{n}\left\| x -y \right\|^{2}+(1-\beta_{n})\left\|S^{n}x-S^{n}y\right\|^{2}
\\&-\beta_{n}(1-\beta_{n})\left\|(x-y)- (S^{n}x-S^{n}y) \right\|^{2}
\\&\leq\beta_{n}\left\| x -y \right\|^{2}+(1-\beta_{n})\left\|S^{n}x-S^{n}y\right\|^{2}
\\&\leq(L^{2}+\beta_{n}(1-L^{2}))\left\|x-y\right\|^{2},
\end{align*}
since $L\in (0,1]$ and $\beta_n \in (0,1)$,  it follow that, $T^{\beta_{n}}$ is nonexpansive, and it is not difficult to see that $Fix(T^{\beta_{n}})= Fix(S^{n}).$
\end{proof}
\end{lemma}

\begin{lemma}\cite{t}\normalfont Let $H$ be a real Hilbert space, $f:H\rightarrow H$ be a contraction  with coefficient $0<\alpha<1$ and $F:H\rightarrow H$ be a $L$-Lipschitzian continuous operator and $\eta$-strongly monotone operator  with $L>0$ and $\eta>0$. Then for $0<\gamma<\frac{\mu\eta}{\alpha}$, $$\left\langle x-y, (\mu F-\gamma f)x-(\mu F-\gamma f)y\right\rangle\geq (\mu\eta-\gamma\alpha)\left\|x-y\right\|^{2}.$$
\end{lemma}

\begin{theorem}\label{thm133}\normalfont
Let  $T:H\rightarrow H$ be a  $(k, \{\mu_{n}\}, \{\xi_{n}\}, \phi)$- total asymptotically strict pseudocontraction mapping and uniformly M-Lipschitzian   with $\phi(t)=t^{2}, \forall t \geq 0$ and $M\in (0,1]$. Assume that  $Fix(T^{n})\neq\emptyset,$ and let $f$ be a contraction with coefficient $\beta\in(0,1)$,  $G:H\rightarrow H$ be a $\eta$-strongly monotone and $L$-Lipschitzian operator with $L>0$ and $\eta>0$ respectively. Assume that $\displaystyle{0<\gamma<\mu(\eta-\frac{\mu L^{2}}{2})/\beta =\frac{\tau}{\beta}}$ and let  $x_{0}\in H$ be chosen arbitrarily, $\{\alpha_{n}\}$ and $\{\beta_{n}\}$ be two sequences in (0,1) satisfying the following conditions:
\begin{equation}
\left\{
\begin{array}{l} {\rm (i)}\displaystyle{\underset{n\rightarrow \infty}\lim \alpha_{n}=0 {\rm ~and~} \sum{\alpha_{n}}=\infty}; \\
{\rm (ii)}\displaystyle{\sum{\left|\alpha_{n+1}-\alpha_{n}\right|}<\infty, \sum{\left|\beta_{n+1}-\beta_{n}\right|}<\infty}~{\rm and}~\sum{\left|\beta_{n}\right|}<\infty; \\
{\rm (iii)} 0\leq k\leq\beta_{n}<a<1, \forall n\geq 0.
\end{array}
\right.\label{con2}
\end{equation}
Let $\{x_{n}\}$ be a sequence defined by
\begin{equation}
 {} \left\{ \begin{array}{ll} T^{\beta_{n}}=\beta_{n}I +(1-\beta_{n})T^{n};
\\ x_{n+1}=\alpha_{n}\gamma f(x_{n})+(I-\alpha_{n}\mu G)T^{\beta_{n}}{x_{n}},& \textrm{ $  $}
 \end{array} \label{3.2} \right.
\end{equation}
then $\{x_{n}\}$ converges strongly to a common fixed of $T^{n}$ which solve the variational inequality problem
\begin{equation}
\langle (\gamma f-\mu G)x^{*}, x-x^{*}\rangle\leq 0, \forall x\in Fix(T^{n}).\label{3.3}
\end{equation}
\begin{proof} The proof is divided into five steps as follows.

\textbf{Step 1}. In this step, we show that
 \begin{equation}
T^{\beta_{n}}~~ {\rm is~nonexpansive~and}~ Fix(T^{\beta_{n}})=Fix(T^{n}).
\end{equation}
 The proof follows directly from lemma (\ref{2.5}).

\textbf{Step 2}. In this step, we show that
\begin{equation}
\{x_{n}\},\{T^{n}x_{n}\},\{f(x_{n})\}~{\rm and}~ \{GT^{n}x_{n}\} ~{\rm ~are~ all~ bounded}.
\end{equation}
Let $x^{*}\in Fix(T^{n}),$ from (\ref{3.2}) and lemma (\ref{2.4}), and the fact that $f$ is a contraction, we have
\begin{align*}\left\|x_{n+1}-x^{*}\right\|&= \left\|\alpha_{n}\gamma f(x_{n})+(I-\alpha_{n}\mu G)T^{\beta_{n}}{x_{n}}-x^{*}\right\|
\\&=\left\|\alpha_{n}(\gamma f(x_{n})-\mu Gx^{*})+(I-\alpha_{n}\mu G)T^{\beta_{n}}{x_{n}}-(I-\alpha_{n}\mu G)x^{*}\right\|
\\&\leq (1-\alpha_{n}\tau)\left\|x_{n}-x^{*}\right\|+\alpha_{n}\left\|\gamma( f(x_{n})-f(x^{*}))+\gamma f(x^{*})-\mu Gx^{*})\right\|
\\&\leq (1-\alpha_{n}(\tau-\gamma\beta))\left\|x_{n}-x^{*}\right\|+\alpha_{n}\left\|\gamma f(x^{*})-\mu Gx^{*})\right\|
\\&\leq \max\left\{\left\|x_{n}-x^{*}\right\|,\frac{\left\|\gamma f(x^{*})-\mu Gx^{*})\right\|}{(\tau-\gamma\beta)}\right\}.
\end{align*}
By using induction, we have
\begin{equation}
\left\|x_{n+1}-x^{*}\right\|\leq \max\left\{\left\|x_{0}-x^{*}\right\|,\frac{\left\|\gamma f(x^{*})-\mu Gx^{*})\right\|}{(\tau-\gamma\beta)}\right\}.
\end{equation}
Hence $\{x_{n}\}$ is bounded, and also
\begin{align}
\left\|T^{n}x_{n}-x^{*}\right\|^{2}&\leq \left\|x_{n}-x^{*}\right\|^{2} + k\left\|x_{n}-x^{*}-(T^{n}x_{n}-x^{*})\right\|^{2}+\mu_{n}\phi (\left\|x_{n}-x^{*}\right\|)+\xi_{n}\nonumber
\\&=\left\|x_{n}-x^{*}\right\|^{2} + k\left\|x_{n}-x^{*}\right\|^{2}+k\left\|T^{n}x_{n}-x^{*}\right\|^{2}\nonumber
\\&+2k\left\|x_{n}-x^{*}\right\|\left\|T^{n}x_{n}-x^{*}\right\|+\mu_{n} \left\|x_{n}-x^{*}\right\|^{2}+\xi_{n}\nonumber
\\&\leq (1+k+\mu_{n})\left\|x_{n}-x^{*}\right\|^{2}+2k\left\|x_{n}-x^{*}\right\|\left\|T^{n}x_{n}-x^{*}\right\|\nonumber
\\&+k\left\|T^{n}x_{n}-x^{*}\right\|^{2}+\xi_{n}.\label{3.7}
\end{align}
From (\ref{3.7}), we deduce that
\begin{align*}
(1-k)\left\|T^{n}x_{n}-x^{*}\right\|^{2}-2k\left\|x_{n}-x^{*}\right\|\left\|T^{n}x_{n}-x^{*}\right\|
\\-(1+k+\mu_{n})\left\|x_{n}-x^{*}\right\|^{2}-\xi_{n}\leq 0.
\end{align*}
This implies that
\begin{align*}
\left\|T^{n}x_{n}-x^{*}\right\|&\leq \frac{k\left\|x_{n}-x^{*}\right\|}{(1-k)}
\\&+ \frac{\sqrt{4k^{2}\left\|x_{n}-x^{*}\right\|^{2}+4(1-k)\{(1+k+\mu_{n})\left\|x_{n}-x^{*}\right\|^{2}+\xi_{n}\}}}{2(1-k)}
\\&=\frac{k\left\|x_{n}-x^{*}\right\|+ \sqrt{(1+(1-k)\mu_{n})\left\|x_{n}-x^{*}\right\|^{2}+(1-k)\xi_{n}}}{(1-k)}
\\&\leq\frac{k\left\|x_{n}-x^{*}\right\|+ (1+(1-k)\mu_{n})\left\|x_{n}-x^{*}\right\|^{2}+(1-k)\xi_{n}}{(1-k)}
\end{align*}
\begin{align}\left\|T^{n}x_{n}-x^{*}\right\|&\leq M^{*},\label{3.8}
\end{align}
where $M^{*}$ is chosen arbitrarily such that $$\sup\left(\frac{k\left\|x_{n}-x^{*}\right\|+ (1+(1-k)\mu_{n}))\left\|x_{n}-x^{*}\right\|^{2}+(1-k)\xi_{n}}{(1-k)}\right)\leq M^{*}.$$  It follows from (\ref{3.8}) that $\{T^{n}x_{n}\}$  is bounded.
Since $G$ is $L$-Lipschitzian, $f$ is contraction and the fact that  $\{x_{n}\}, \{T^{n}x_{n}\}$ are bounded, it is easy to see that $\{GT^{n}x_{n}\}$ and $\{f(x_{n})\}$ are also bounded.

\textbf{Step 3.} In this step, we show that
\begin{equation}
\underset{n\rightarrow\infty}{\lim}\left\|x_{n+1}-x_{n}\right\|=0.\label{3.9}
\end{equation}
Now,
\begin{align*}
\left\|x_{n+2}-x_{n+1}\right\|&=\Big(\alpha_{n+1}\gamma f(x_{n+1})+(I-\alpha_{n+1}\mu G)T^{\beta_{n+1}}{x_{n+1}}\Big)
\\&-\Big(\alpha_{n}\gamma f(x_{n})+(I-\alpha_{n}\mu G)T^{\beta_{n}}{x_{n}}\Big)
\\&=\alpha_{n+1}\gamma( f(x_{n+1})-f(x_{n}))+(\alpha_{n+1}-\alpha_{n})\gamma f(x_{n})
\\&+(I-\alpha_{n+1}\mu G)T^{\beta_{n+1}}{x_{n+1}}-(I-\alpha_{n+1}\mu G)T^{\beta_{n}}{x_{n}}
\\&+(\alpha_{n}-\alpha_{n+1})\mu GT^{\beta_{n}}{x_{n}},
\end{align*}
this turn to implies that
\begin{eqnarray} \left\| x_{n+2}-x_{n+1}\right\|&\leq& \alpha_{n+1}\gamma \beta\left\|x_{n+1}-x_{n}\right\|+(1-\alpha_{n+1}\tau)\left\|T^{\beta_{n+1}}{x_{n+1}}-T^{\beta_{n}}{x_{n}}\right\|
\nonumber \\& +&  \left|\alpha_{n+1}-\alpha_{n}\right|\Big(\gamma\left\|f(x_{n})\right\|+\mu\left\|GT^{\beta_{n}}x_{n}\right\|\Big)
\nonumber \\ &\leq& \alpha_{n+1}\gamma \beta\left\|x_{n+1}-x_{n}\right\|+(1-\alpha_{n+1}\tau)\left\|T^{\beta_{n+1}}{x_{n+1}}-T^{\beta_{n}}{x_{n}}\right\|
\nonumber \\& +&  \left|\alpha_{n+1}-\alpha_{n}\right|N_{1},\label{a}
\end{eqnarray}
where $N_{1}$ is chosen arbitrarily so that $\underset{n\geq 1}{\sup}\Big(\gamma\left\|f(x_{n})\right\|+\mu\left\|GT^{\beta_{n}}x_{n}\right\|\Big)\leq N_{1}$.

On the other hand,
\begin{align}
\left\|T^{\beta_{n+1}}x_{n+1}-T^{\beta_{n}}x_{n}\right\|&\leq \left\|T^{\beta_{n+1}}{x_{n+1}}-T^{\beta_{n+1}}{x_{n}}\right\|+\left\|T^{\beta_{n+1}}{x_{n}}-T^{\beta_{n}}{x_{n}}\right\|
\nonumber
\\ &\leq \left\|x_{n+1}-x_{n}\right\|+\left|\beta_{n+1}-\beta_{n}\right|\left\|x_{n}\right\|+\left|\beta_{n+1}\right|\left\|T^{n+1}x_{n}\right\|
\nonumber
\\+& \left|\beta_{n}\right|\left\|T^{n}x_{n}\right\|
\nonumber
\\\leq \left\|x_{n+1}-x_{n}\right\|&+\left|\beta_{n+1}-\beta_{n}\right|N_{2}+\left|\beta_{n+1}\right|N_{3}
 +\left|\beta_{n}\right|N_{4}, \label{b}
\end{align}
where $N_{2,3,4}$ satisfy the following relations: $$N_{2}\geq \underset{n\geq 1}{\sup}\left\|x_{n}\right\|,\ N_{3}\geq \underset{n\geq 1}{\sup}\left\|T^{n+1}x_{n}\right\|\ {\rm and} \ N_{4}\geq \underset{n\geq 1}{\sup}\left\|T^{n}x_{n}\right\|$$ respectively.

\noindent Now substituting (\ref{b}) into (\ref{a}), yields
\begin{align*}\left\| x_{n+2}-x_{n+1}\right\|&\leq\alpha_{n+1}\gamma \beta\left\|x_{n+1}-x_{n}\right\|+(1-\alpha_{n+1}\tau)\Big(\left\|x_{n+1}-x_{n}\right\|
\\&+\left|\beta_{n+1}-\beta_{n}\right|N_{2}+\left|\beta_{n+1}\right|N_{3}
 +\left|\beta_{n}\right|N_{4}\Big)
\\& + \left|\alpha_{n+1}-\alpha_{n}\right|N_{1}
\\&=(1+\alpha_{n+1}(\gamma \beta-\tau))\left\|x_{n+1}-x_{n}\right\|+ \left|\alpha_{n+1}-\alpha_{n}\right|N_{1}
\\&+(1-\alpha_{n+1}\tau)\Big(\left|\beta_{n+1}-\beta_{n}\right|N_{2}+\left|\beta_{n+1}\right|N_{3}
 +\left|\beta_{n}\right|N_{4}\Big)
\\&\leq (1-\alpha_{n+1}(\tau-\gamma \beta)\left\|x_{n+1}-x_{n}\right\|+
\\+(1-\alpha_{n+1}\tau)&\Big(\left|\beta_{n+1}-\beta_{n}\right|+\left|\beta_{n+1}\right|
 +\left|\beta_{n}\right|+ \left|\alpha_{n+1}-\alpha_{n}\right|\Big)N_{5},
\end{align*}
where $N_{5}$ choosing appropriately such that $ N_{5}\geq \max\{N_{1},N_{2},N_{3},N_{4}\}$.

\noindent By lemma (2.3)  and (ii), it follows that
\begin{equation*}
\underset{n\rightarrow\infty}{\lim}\left\|x_{n+1}-x_{n}\right\|=0.
\end{equation*}

From equation (\ref{3.2}), we have,
\begin{align*}\left\| x_{n+1}-T^{\beta_{n}}x_{n}\right\|&=\left\|\alpha_{n}\gamma f(x_{n})+(I-\alpha_{n}\mu G)T^{\beta_{n}}x_{n}-T^{\beta_{n}}x_{n}\right\|
\\&\leq\alpha_{n}\left\|\gamma f(x_{n})-\mu GT^{\beta_{n}}x_{n}\right\|\rightarrow 0.
\end{align*}
\noindent On the other hand,
\begin{align*}\left\| x_{n+1}-T^{\beta_{n}}x_{n}\right\|&=\left\|x_{n+1}-(\beta_{n}+(1-\beta_{n})T^{n})x_{n}\right\|
\\&=\left\|(x_{n+1}-x_{n})+(1-\beta_{n})(x_{n}-T^{n}x_{n})\right\|
\\&\geq(1-\beta_{n})\left\|x_{n}-T^{n}x_{n}\right\|-\left\|x_{n+1}-x_{n}\right\|,
\end{align*}
this implies that
 \begin{align*}\left\|x_{n}-T^{n}x_{n}\right\|&\leq\frac{\left\| x_{n+1}-T^{\beta_{n}}x_{n}\right\|+\left\|x_{n+1}-x_{n}\right\|}{(1-\beta_{n})}
\\&\leq\frac{\left\| x_{n+1}-T^{\beta_{n}}x_{n}\right\|+\left\|x_{n+1}-x_{n}\right\|}{(1-a)}\rightarrow 0.
\end{align*}
From the boundedness of  $\{x_{n}\}$, we deduce that $\{x_{n}\}$ converges weakly. Now assume that  $x_{n}\rightharpoonup p$, by lemma (2.2) and the fact that $\left\|x_{n}-T^{n}x_{n}\right\|\rightarrow 0$, we obtain $p\in Fix(T^{n})$. So, we have
\begin{equation}
\omega_{\omega}(x_{n})\subset Fix(T^{n}).\label{3.12}
\end{equation}
By lemma (2.6) it follows that $(\gamma f-\mu G)$ is strongly monotone, so the variational inequality (\ref{3.3}) has a unique solution $x^{*}\in Fix(T^{n})$.

\textbf{Step 4}. In this step, we show that
\begin{equation}
\underset{n\rightarrow\infty}{\limsup}\left\langle (\gamma f-\mu G)x^{*}, x_{n}-x^{*}\right\rangle\leq 0.
\end{equation}
The fact that $\{x_{n}\}$ is bounded, we have $\{x_{n_{i}}\}\subset \{x_{n}\}$ such that
  $$\underset{n\rightarrow\infty}{\limsup}\left\langle (\gamma f-\mu G)x^{*}, x_{n}-x^{*}\right\rangle=\underset{i\rightarrow\infty}{\limsup}\left\langle (\gamma f-\mu G)x^{*}, x_{n_{i}}-x^{*}\right\rangle\leq 0.$$
\noindent Suppose without loss of generality that $x_{n_{i}}\rightharpoonup x$, from (\ref{3.12}), it follows that $x\in Fix (T^{n})$. Since $x^{*}$ is the unique solution of (\ref{3.2}), implies that
\begin{align*}
\underset{n\rightarrow\infty}{\limsup}\left\langle (\gamma f-\mu G)x^{*}, x_{n}-x^{*}\right\rangle&=\underset{i\rightarrow\infty}{\limsup}\left\langle (\gamma f-\mu G)x^{*}, x_{n_{i}}-x^{*}\right\rangle.
\\&=\left\langle (\gamma f-\mu G)x^{*}, x-x^{*}\right\rangle\leq 0.
\end{align*}

\textbf{Step 5.} In this step, we show that
\begin{equation}
\underset{n\rightarrow\infty}{\lim}\left\|x_{n}-x^{*}\right\|=0.
\end{equation}
By lemma (2.4) and the fact that $f$ is a contraction, we have
\begin{align*}
\left\|x_{n+1}-x^{*}\right\|^{2}&=\left\|\alpha_{n}(\gamma f(x_{n})-\mu Gx^{*})+(I-\alpha_{n}\mu G)T^{\beta_{n}}x_{n}-(I-\alpha_{n}\mu G)x^{*}\right\|^{2}
\\&\leq \left\|(I-\alpha_{n}\mu G)T^{\beta_{n}}x_{n}-(I-\alpha_{n}\mu G)x^{*}\right\|^{2}
\\&+2\alpha_{n}\left\langle \gamma f(x_{n})-\mu Gx^{*},x_{n+1}-x^{*}\right\rangle
\\&\leq (1-\alpha_{n}\tau)^{2} \left\|x_{n}-x^{*}\right\|^{2}+2\alpha_{n}\gamma\left\langle  f(x_{n})-f(x^{*}),x_{n+1}-x^{*}\right\rangle
\\&+2\alpha_{n}\left\langle \gamma f(x^{*})-\mu Gx^{*},x_{n+1}-x^{*}\right\rangle
\\&\leq (1-\alpha_{n}\tau)^{2} \left\|x_{n}-x^{*}\right\|^{2}+2\alpha_{n}\beta\gamma \left\|x_{n}-x^{*}\right\|\left\|x_{n+1}-x^{*}\right\|
\\&+2\alpha_{n}\left\langle \gamma f(x^{*})-\mu Gx^{*},x_{n+1}-x^{*}\right\rangle
\\&\leq (1-\alpha_{n}\tau)^{2} \left\|x_{n}-x^{*}\right\|^{2}+\alpha_{n}\beta\gamma\Big( \left\|x_{n}-x^{*}\right\|^{2}+\left\|x_{n+1}-x^{*}\right\|^{2}\Big)
\\&+2\alpha_{n}\left\langle \gamma f(x^{*})-\mu Gx^{*},x_{n+1}-x^{*}\right\rangle,
\end{align*}
this implies that
\begin{align*}\left\|x_{n+1}-x^{*}\right\|^{2}&\leq \frac{\Big( (1-\alpha_{n}\tau)^{2} +\alpha_{n}\beta\gamma\Big)\left\|x_{n}-x^{*}\right\|^{2}}{(1-\alpha_{n}\gamma\beta)}
\\&+\frac{2\alpha_{n}\left\langle \gamma f(x^{*})-\mu Gx^{*},x_{n+1}-x^{*}\right\rangle}{(1-\alpha_{n}\gamma\beta)}
\\&\leq\Big( 1-(2\tau-\gamma\beta)\alpha_{n}\Big)\left\|x_{n}-x^{*}\right\|^{2}+\frac{(\alpha_{n}\tau)^{2}}{(1-\alpha_{n}\gamma\beta)}\left\|x_{n}-x^{*}\right\|^{2}
\\&+\frac{2\alpha_{n}\left\langle \gamma f(x^{*})-\mu Gx^{*},x_{n+1}-x^{*}\right\rangle}{(1-\alpha_{n}\gamma\beta)},
\end{align*}
this implies that $$\left\|x_{n+1}-x^{*}\right\|^{2}\leq (1-\gamma_{n})\left\|x_{n}-x^{*}\right\|^{2}+\sigma_{n},$$
where \begin{eqnarray*}
\gamma_{n}&:=& (2\tau-\gamma\beta)\alpha_{n} \ \ {\rm and} \\
\sigma_{n}&:=&\frac{\alpha_{n}}{(1-\alpha_{n}\gamma\beta)}\Big(\alpha_{n}\tau^{2}\left\|x_{n}-x^{*}\right\|^{2}+2\left\langle \gamma f(x^{*})-\mu Gx^{*}, x_{n+1}-x^{*}\right\rangle\Big).
\end{eqnarray*}
 From (3.1 (i)), it follows that
\begin{eqnarray*}
 \underset{n\rightarrow\infty}{\lim}\gamma_{n}&=& 0, \\ \sum\gamma_{n}&=&\infty,
\end{eqnarray*}
$$\frac{\sigma_{n}}{\gamma_{n}}=\frac{1}{(2\tau-\gamma\beta)(1-\alpha_{n}\gamma\beta)}\Big(\alpha_{n}\tau^{2}\left\|x_{n}-x^{*}\right\|^{2}+2\left\langle \gamma f(x^{*})-\mu Gx^{*},x_{n+1}-x^{*}\right\rangle\Big).$$
 Thus $\displaystyle{\underset{n\rightarrow\infty}{\lim}\frac{\sigma_{n}}{\gamma_{n}}\leq 0}$.

Hence by Lemma (2.3), it follows that $x_{n}\rightarrow x^{*}$ as $n\rightarrow \infty$.
\end{proof}
\end{theorem}

\begin{corollary}\normalfont Let $B$ be a unit ball is a real Hilbert space $l_{2}$, and let the mapping $T:B\rightarrow B$ be defined by
$$T:(x_{1}, x_{2}, x_{3},\ldots)\rightarrow (0, x_{1}^{2}, a_{2}x_{2}, a_{3}x_{3},\ldots), (x_{1}, x_{2}, x_{3},\ldots)\in B,$$
where $\{a_{i}\}$ is a sequence in $(0,1)$ such that $\displaystyle{\prod^{\infty}_{i=2}(a_{i})=\frac{1}{2}}$. Let, $f, G,\gamma,  \{\alpha_{n}\}, \{\beta_{n}\}$   be as in theorem (3.1). Then the sequence $\{x_{n}\}$ define by algorithm (\ref{3.2}), converges strongly to a common fixed point of  $~T^{n}$ which solve the variational inequality problem (3.3).
 \begin{proof}
By example (1.1), it follows that $T$ is $(k, \{\mu\}, \{\xi_{n}\}, \phi)$- total asymptotically strict pseudocontraction mapping and uniformly $M$-Lipschitzian with $M=\displaystyle{2\prod^{n}_{i=2}(a_{i})}$. Hence, the conclusion of this corollary, follows directly from   theorem (3.1).
\end{proof}
\end{corollary}

\begin{corollary}\normalfont
Let $H$ be a real Hilbert space  and $T:H\rightarrow H$ be a $(k,  \{k_{n}\})$- asymptotically strict pseudocontraction mapping and uniformly $M$-Lipschitzian with  $M\in (0,1]$. Assume that $Fix(T^{n})\neq\emptyset$, and Let $f, G, \gamma$  $\{\alpha_{n}\}$  and  $\{\beta_{n}\}$ be as in theorem (3.1).
Then, the sequence $\{x_{n}\}$ generated by algorithm (\ref{3.2}), converges strongly to a common fixed point of $~T^{n}$ which solve the variational inequality problem (\ref{3.3}).
\end{corollary}
\begin{corollary}\cite{t}\normalfont Let the sequence $\{x_{n}\}$ be generated by the mapping
 $$x_{n+1}=\alpha_{n} \gamma f(x_{n})+(I-\mu\alpha_{n}F)Tx_{n},$$
where $T$ is nonexpansive, $\alpha_{n}$ is a sequence in  (0,1) satisfying the following conditions:
\begin{equation}
\left\{
\begin{array}{l} {\rm (i)} \ \displaystyle{\underset{n\rightarrow \infty}\lim \alpha_{n}=0}, \ \sum{\alpha_{n}}=\infty;\\
 {\rm (ii)} \ \displaystyle{\sum{|\alpha_{n+1}-\alpha_{n}|}<\infty, \  ~~\sum{|\beta_{n+1}-\beta_{n}|}<\infty}; \\
 {\rm (iii)} \ \displaystyle{0\leq \max_{i}  k_{i}\leq\beta_{n}<a<1, \forall n\geq 0}. \end{array}
\right.  \label{con1}
\end{equation}

It was proved in \cite{t}  that $\{x_{n}\}$ converged  strongly to  the common fixed point $x^{*}$ of $T$, which is the solution of variational inequality problem
\begin{equation}
\langle (\gamma f-\mu F)x^{*},x-x^{*}\rangle\leq 0, \forall x\in Fix(T).
\end{equation}
\begin{proof}
Take n=1, $k=\mu_{n}=\xi_{n}=0$ and $F=G$ in theorem (3.1). Therefore all the conditions in theorem (3.1) are satisfied. Hence the conclusion of this corollary follows directly from theorem (3.1).
\end{proof}
 \end{corollary}
\begin{corollary}\cite{marino}\normalfont
Let the sequence $\{x_{n}\}$ be  generated by
 $$x_{n+1}=\alpha_{n} \gamma f(x_{n})+(I-\alpha_{n}A)Tx_{n},$$
where $T$ is nonexpansive and the sequence  $\alpha_{n}\subset (0,1)$ satisfy the conditions in equation (\ref{con2}).
Then it was proved in \cite{marino} that $\{x_{n}\}$ converged strongly to $x^{*}$ which solve the variational inequality
\begin{equation}
\langle (\gamma f- A)x^{*},x-x^{*}\rangle\leq 0, \forall x\in Fix(T).
\end{equation}
 \begin{proof}Take n=1, $\mu_{n}=\xi_{n}=0$  and $\mu=1$ and $G=A$ in theorem (3.1). Therefore all the conditions in theorem (3.1) are satisfied. Hence the conclusion of this corollary follows directly from theorem (3.1).
\end{proof}
\end{corollary}
\begin{corollary}\cite{yamada} \normalfont
Let the sequence $\{x_{n}\}$ be generated by
 $$x_{n+1}=Tx_{n}-\mu\lambda_{n}F(Tx_{n}),$$
where $T$ is nonexpansive mapping on $H$, $F$ is L-Lipschitzian and $\eta$-strongly monotone  with $L>0, \eta >0$ and $0<\mu<\frac{2\eta}{L^{2}}$, if the sequence  $\lambda_{n}\subset (0,1)$ satisfies the following conditions:
\begin{equation}
\left\{
\begin{array}{l} {\rm (i)} \ \displaystyle{\underset{n\rightarrow \infty}\lim \lambda_{n}=0}, \displaystyle{\sum{\lambda_{n}}=\infty};\\
 {\rm (ii) \ either} \displaystyle{\sum{|\lambda_{n+1}-\lambda_{n}|}=0 ~~{\rm or}~~\underset{n\rightarrow \infty}\lim\frac{\lambda_{n+1}}{\lambda_{n}}=1}. \end{array}
\right.  \label{con}
\end{equation}
 Then, it was proved by Yamada in \cite{yamada} that $\{x_{n}\}$ converged strongly to the unique solution of the variational inequality
\begin{equation}
\langle Fx^{*},x-x^{*}\rangle\geq 0,\forall x\in Fix(T).
\end{equation}
 \begin{proof}Take  $n=1$,  $k=\mu_{n}=\xi_{n}=0$ and also take $\gamma=0$, $\beta_{n}=0$ and $G=F$.  Therefore all the conditions in theorem (3.1) are satisfied. Hence the result follows directly from theorem (3.1).
\end{proof}
\end{corollary}

\section{A note on the Split Equality Fixed Point  Problems  in Hilbert Spaces }
In this section, we propose the split feasibility and fixed point equality problems (SFFPEP) and split common fixed point equality problems (SCFPEP). Furthermore, we formulate and analyse the algorithms for solving these problems
for the class of quasi-nonexpansive mappings in Hilbert spaces. In the end, we study the convergence results  of the proposed algorithms.

\subsection{Problem Formulation}
The split feasibility and fixed point equality problems (in short, SFFPEP)  formulated as follows:
\begin{equation}
{\rm~Find~} x^{*}\in C\cap Fix(U) {\rm~and~} y^{*}\in Q\cap Fix(T) {\rm~such~ that~} Ax^{*}=By^{*}.\label{7.1}
\end{equation}
While the split common fixed  point equality problems (in short, SCFPEP)  obtained as follows:
\begin{equation}
{\rm~Find~} x^{*}\in\bigcap_{i=1}^{N}Fix(U_{i}) {\rm~and~} y^{*}\in \bigcap_{j=1}^{M}Fix(T_{j}) {\rm~such~ that~} Ax^{*}=By^{*},\label{7.2}
\end{equation}
where $U_{i=1}:H_{1}\to H_{1}, i=1,2,3,...,N,$ and $T_{j=1}:H_{2}\to H_{2}, j=1,2,3,...,M,$  are quasi-nonexpansive mappings with  $Fix(U_{i})\neq\emptyset$ and $Fix(T_{j})\neq\emptyset,$ respectively, $A:H_{1}\to H_{3}$ and $B:H_{2}\to H_{3}$ are bounded linear operators.\\

Note that if $C:=Fix(U),$  $Q:=Fix(T),$  $H_{2}=H_{3}$ and $B=I.$ Then Problem (\ref{7.1}) reduces to the following problems:
\begin{equation}
{\rm~Find~} x^{*}\in C {\rm~and~} y^{*}\in Q {\rm~such~ that~} Ax^{*}=By^{*},\label{7.3}
\end{equation}
and
  \begin{equation}
  x^{*}\in C\cap Fix(U) {\rm~such~ that~} Ax^{*}\in Q\cap Fix(T).\label{7.4}
\end{equation}
Equation (\ref{7.3}) and (\ref{7.4}) are called the split equality fixed point problems (SEFPP) and split feasibility and fixed point problems (SFFPP), respectively. In the light of this, it is worth to mention here that the SFFPEP generalizes the SFP,  SFFPP, and SEFPP. Therefore, the results and conclusions that are true for the SFFPEP  continue to hold for these problems (SFP, SFFPP, and SEFPP), and it  shows the significance and the range of applicability of the SFFPEP.\\

Furthermore, Problem (\ref{7.2}) reduces to Problem (\ref{4.1zy}) as $H_{2}=H_{3}$ and $B=I.$ This shows that the SCFPEP generalizes the SCFPP. Therefore, the results and conclusions that are true for the SCFPEP  continue to hold for the SCFPP.\\

We denote the solution of sets SFFPEP (\ref{7.1}) and SCFPEP (\ref{7.2}) by

\begin{equation}
\Phi=\Big\{ x^{*}\in C\cap Fix(U) {\rm~and~} y^{*}\in Q\cap Fix(T) {\rm~such~ that~} Ax^{*}=By^{*}\Big\},\label{7.6}
\end{equation}
and
\begin{equation}
\Psi=\Big\{ x^{*}\in\bigcap_{i=1}^{N}Fix(U_{i}) {\rm~and~} y^{*}\in \bigcap_{j=1}^{M}Fix(T_{j}) {\rm~such~ that~} Ax^{*}=By^{*}\Big\}, \label{7.7}
\end{equation}
respectively.  In sequel,  we assume that $\Phi$ and $\Psi$ are nonempty.

\subsection{Preliminaries}
In this section, we present some lemmas used in proving our main result.
\begin{lemma}\normalfont\label{opial} Let $C\subset H$ and $\{x_{n}\}$ be a
sequence in $H$ such that the following conditions are satisfied:
\begin{enumerate}
\item[(i)] For each $x\in C$, $\underset{n\to\infty}{\lim}{\|x_{n}-x\|}$ exists,
\item[(ii)] Any weak-cluster point of the sequence $\{x_{n}\}$ belongs to $C.$
\end{enumerate}
Then, there exists $y\in C$ such that $\{x_{n}\}$  converges weakly to $y.$
\end{lemma}
For the proof, see \cite{opial1967weak} and references therein.

\begin{lemma}\normalfont\label{2T} Let $T_{i}:H\to H,$ for i=1,2,3,...,N be N-quasi-nonexpansive mappings. Defined $U=\sum^{N}_{i=1}{\delta_{i}U_{\beta_{i}}},$ where $U_{\beta_{i}}=(1-\beta_{i})I+\beta_{i}T_{i},$ and $\delta_{i}\in (0,1)$ such that $\sum^{N}_{i=1}{\delta_{i}}=1.$  Then
\begin{enumerate}
\item[(i)] U is a quasi-nonexpansive mapping,
\item[(ii)] $Fix(U)=\bigcap^{N}_{i=1}Fix(U_{\beta_{i}})= \bigcap^{N}_{i=1}Fix(T_{i}),$
\item [(iii)] in addition, if $(T_{i}-I)$ for i=1,2,3,...,N is demiclosed at zero, then $(U-I)$ is also demiclosed at zero.
\end{enumerate}
 \end{lemma}
For the proof, see Li and He \cite{li2015new} and the references therein.

\subsection{The Split Feasibility and Fixed Point Equality Problems for Quasi-Nonexpansive Mappings in Hilbert Spaces}

To approximate the solution of the split feasibility and fixed point equality problems (\ref{7.6}), we make the following assumptions:
\begin{enumerate}
\item [($B_{1}$)] $U:H_{1}\to H_{1}$ and  $T:H_{2}\to H_{2}$ are  quasi-nonexpansive mappings with  $Fix(U)\neq\emptyset$ and
$Fix(T)\neq\emptyset,$  respectively.
\item [($B_{2}$)]  $A:H_{1}\to H_{3}$ and $B:H_{2}\to H_{3}$  are bounded linear operators with their adjoints $A^{*}$ and $B^{*},$ respectively.
\item [($B_{3}$)] $(U-I)$ and $(T-I)$ are demiclosed at zero.
\item [($B_{4}$)] $P_{C}$ and $P_{Q}$ are metric projection of $H_{1}$ and  $H_{2}$ onto $C$ and $Q,$ respectively.
\item [($B_{5}$)] For arbitrary $x_{1}\in H_{1}$ and $ y_{1}\in H_{2},$ define a sequence $\{(x_{n}, y_{n})\}$ by:
\end{enumerate}

\begin{equation}
 {} \left\{ \begin{array}{ll}
z_{n}=P_{C}(x_{n}-\lambda_{n} A^{*}(Ax_{n}-By_{n})),
\\ w_{n}=(1-\beta_{n})z_{n}+\beta_{n}U(z_{n}),
\\ x_{n+1}=(1-\alpha_{n})z_{n}+\alpha_{n}U(w_{n}),
\\
\\u_{n}=P_{Q}(y_{n}+\lambda_{n} B^{*}(Ax_{n}-By_{n})),
\\ r_{n}=(1-\beta_{n})u_{n}+\beta_{n}T(u_{n}),
\\ y_{n+1}=(1-\alpha_{n})u_{n}+\alpha_{n}T(r_{n}),  \forall n\geq 1, & \textrm{ $  $}
 \end{array}  \right.\label{BUL}
\end{equation}
where $0<a<\beta_{n}<1,$ $0<b<\alpha_{n}<1,$ and $\lambda_{n}\in\left(0, \frac{2}{L_{1}+L_{2}}\right),$ where $L_{1}=A^{*}A$  and $L_{2}=B^{*}B,$ respectively.\\

We are now in the position to state and prove the main result of this chapter.

\begin{theorem}\normalfont\label{TA}
Suppose that assumption $(B_{1})-(B_{5})$ are satisfied, also assume that the solution set  $\Phi\neq\emptyset.$  Then  $(x_{n}, y_{n})\rightharpoonup(x^{*}, y^{*})\in\Phi.$

\begin{proof}

Let $(x^{*}, y^{*})\in\Phi.$ By (\ref{BUL}), we have
\begin{align}
\left\|x_{n+1}-x^{*}\right\|^{2}&=\left\|(1-\alpha_{n})(z_{n}-x^{*})+\alpha_{n}(Uw_{n}-x^{*})\right\|^{2}\nonumber
\\&=(1-\alpha_{n})\left\|z_{n}-x^{*}\right\|^{2}+\alpha_{n}\left\|Uw_{n}-x^{*}\right\|^{2}-\alpha_{n}(1-\alpha_{n})\left\|Uw_{n}-z_{n}\right\|^{2}\nonumber
\\&\leq (1-\alpha_{n})\left\|z_{n}-x^{*}\right\|^{2} +\alpha_{n}\left\|w_{n}-x^{*}\right\|^{2}\nonumber
\\&-\alpha_{n}(1-\alpha_{n})\left\|Uw_{n}-z_{n}\right\|^{2}.\label{7a1}
\end{align}
On the other hand,
\begin{align}
\left\|w_{n}-x^{*}\right\|^{2}&=\left\|(1-\beta_{n})(z_{n}-x^{*})+\beta_{n}(Uz_{n}-x^{*})\right\|^{2}\nonumber
\\&=(1-\beta_{n})\left\|z_{n}-x^{*}\right\|^{2}+\beta_{n}\left\|Uz_{n}-x^{*}\right\|^{2}-\beta_{n}(1-\beta_{n})\left\|Uz_{n}-z_{n}\right\|^{2}\nonumber
\\&\leq \left\|z_{n}-x^{*}\right\|^{2}-\beta_{n}(1-\beta_{n})\left\|Uz_{n}-z_{n}\right\|^{2}.\label{7a2}
\end{align}
Substituting (\ref{7a2}) into (\ref{7a1}), we have
\begin{align}
\left\|x_{n+1}-x^{*}\right\|^{2}&\leq (1-\alpha_{n})\left\|z_{n}-x^{*}\right\|^{2} +\alpha_{n}\left\|z_{n}-x^{*}\right\|^{2}\nonumber
\\&-\alpha_{n}\beta_{n}(1-\beta_{n})\left\|Uz_{n}-z_{n}\right\|^{2} -\alpha_{n}(1-\alpha_{n})\left\|Uw_{n}-z_{n}\right\|^{2}.\label{7a1x}
\end{align}

On the Other hand,
\begin{align}
\left\|z_{n}-x^{*}\right\|^{2}&=\left\|P_{C}(x_{n}-\lambda_{n} A^{*}(Ax_{n}-By_{n}))-P_{C}(x^{*})\right\|^{2}\nonumber
\\&\leq \left\|x_{n}-\lambda_{n} A^{*}(Ax_{n}-By_{n})-x^{*}\right\|^{2}\nonumber
\\&= \left\|x_{n}-x^{*}\right\|^{2}-2\lambda_{n}\left\langle  Ax_{n}-Ax^{*}, Ax_{n}-By_{n}\right\rangle\nonumber
\\&+\lambda_{n}^{2}L_{1}\left\|Ax_{n}-By_{n}\right\|^{2}.\label{7a3}
\end{align}
Substituting (\ref{7a3}) into (\ref{7a1x}), we have
\begin{align}
\left\|x_{n+1}-x^{*}\right\|^{2}&\leq\left\|x_{n}-x^{*}\right\|^{2}-2\lambda_{n}\left\langle  Ax_{n}-Ax^{*}, Ax_{n}-By_{n}\right\rangle+\lambda_{n}^{2}L_{1}\left\|Ax_{n}-By_{n}\right\|^{2}\nonumber
\\&-\alpha_{n}\beta_{n}(1-\beta_{n})\left\|U(z_{n})-z_{n}\right\|^{2}-\alpha_{n}(1-\alpha_{n})\left\|Uw_{n}-z_{n}\right\|^{2}.\label{7a4}
\end{align}
Similarly, the second equation of Equation (\ref{BUL}) gives
\begin{align}
\left\|y_{n+1}-y^{*}\right\|^{2}&\leq\left\|y_{n}-y^{*}\right\|^{2}+2\lambda_{n}\left\langle  By_{n}-By^{*}, Ax_{n}-By_{n}\right\rangle+\lambda_{n}^{2}L_{2}\left\|Ax_{n}-By_{n}\right\|^{2}\nonumber
\\&-\alpha_{n}\beta_{n}(1-\beta_{n})\left\|T(u_{n})-u_{n}\right\|^{2}-\alpha_{n}(1-\alpha_{n})\left\|Tr_{n}-u_{n}\right\|^{2}.\label{7a5}
\end{align}
By (\ref{7a4}), (\ref{7a5}) and noticing  that $Ax^{*} = By^{*},$ we deduce that
\begin{align}
\left\|x_{n+1}-x^{*}\right\|^{2}+\left\|y_{n+1}-y^{*}\right\|^{2}&\leq \left\|x_{n}-x^{*}\right\|^{2}+\left\|y_{n}-y^{*}\right\|^{2}-2\lambda_{n}\left\|Ax_{n}-By_{n}\right\|^{2}\nonumber
\\&+\lambda_{n}^{2}(L_{1}+L_{2})\left\|Ax_{n}-By_{n}\right\|^{2}\nonumber
\\&-\alpha_{n}\beta_{n}(1-\beta_{n})\left\|U(z_{n})-z_{n}\right\|^{2}\nonumber
\\&-\alpha_{n}\beta_{n}(1-\beta_{n})\left\|T(u_{n})-u_{n}\right\|^{2}.\label{f}
\end{align}
Thus, we deduce that
\begin{align}
\Phi_{n+1}&\leq \Phi_{n}-\lambda_{n}\left(2-\lambda_{n}(L_{1}+L_{2})\right)\left\|Ax_{n}-By_{n}\right\|^{2}\nonumber
\\&-\alpha_{n}\beta_{n}(1-\beta_{n})\left\|U(z_{n})-z_{n}\right\|^{2}-\alpha_{n}\beta_{n}(1-\beta_{n})\left\|T(u_{n})-u_{n}\right\|^{2},\label{g}
\end{align}
where $$\Phi_{n}:=\left\|x_{n}-x^{*}\right\|^{2}+\left\|y_{n}-y^{*}\right\|^{2}.$$
Thus, $\{\Phi_{n}\}$ is a non-increasing sequence and bounded below by 0, therefore, it converges.\\

From (\ref{g}) and the fact that $\{\Phi_{n}\}$ converges, we deduce that
\begin{align}
\underset{n\to\infty}{\lim}\left\|Ax_{n}-By_{n}\right\|=0,\label{q}
\end{align}
\begin{align}
\underset{n\to\infty}{\lim}\left\|Uz_{n}-z_{n}\right\|=0 {\rm~and~} \underset{n\to\infty}{\lim}\left\|Tu_{n}-u_{n}\right\|=0\label{K}.
\end{align}

Furthermore, since $\{\Phi_{n}\}$ converges, this ensures that  $\{x_{n}\}$ and $\{y_{n}\}$  also converges. This further implies that $x_{n}\rightharpoonup x$ and $y_{n}\rightharpoonup y$ for some $(x,y)\in\Phi.$\\

Now,  $(x,y)\in\Phi,$  implies that $x\in C\cap Fix(U)$ and $y\in Q\cap Fix(T)$ such that $Ax=By.$  The fact that $x_{n}\rightharpoonup x$ and $\underset{n\to\infty}{\lim}\left\|Ax_{n}-By_{n}\right\|=0$
together with
$$z_{n}=P_{C}(x_{n}-\lambda_{n} A^{*}(Ax_{n}-By_{n})),$$
we deduce that   $z_{n}\rightharpoonup P_{C}x.$  Since $x\in C,$ by projection theorem, we obtain that $P_{C}x=x.$ Hence,  $z_{n}\rightharpoonup x.$\\

Similarly, The fact that $y_{n}\rightharpoonup y$ and $\underset{n\to\infty}{\lim}\left\|Ax_{n}-By_{n}\right\|=0$ together with
$$u_{n}=P_{Q}(y_{n}+\lambda_{n} B^{*}(Ax_{n}-By_{n})),$$
we deduce that   $u_{n}\rightharpoonup P_{Q}y.$  Since $y\in Q,$ by projection theorem, we obtain that $P_{Q}y=y.$ Hence,  $u_{n}\rightharpoonup y.$\\



Now, $z_{n}\rightharpoonup x$, $\underset{n\to\infty}{\lim}\left\|Uz_{n}-z_{n}\right\|=0,$ and together with the demiclosed of $(U-I)$ at zero, we deduce that $Ux=x,$ this implies that $x\in Fix(U).$\\

On the other hand, $u_{n}\rightharpoonup y$ and $\underset{n\to\infty}{\lim}\left\|Tu_{n}-u_{n}\right\|=0$ together with the demiclosed of $(T-I)$ at zero, we deduce that $Ty=y,$ this implies that $y\in Fix(T).$ \\

Since $z_{n}\rightharpoonup x,$  $u_{n}\rightharpoonup y$ and the fact that $A$ and $B$ are bounded linear operators, we have
$$Az_{n}\rightharpoonup Ax {\rm~~and~~} Bu_{n}\rightharpoonup By,$$
this implies that
$$Az_{n}-Bu_{n}\rightharpoonup Ax-By,$$
which turn to implies that
$$\left\|Ax-By\right\|\leq\underset{n\to\infty}{\liminf}\left\|Az_{n}-Bu_{n}\right\|=0,$$
which further implies that $Ax=By.$ Noticing that $x\in C,$  $x\in Fix(U),$ $y\in Q$ and $y\in Fix(T)$, we have that $x\in C\cap Fix(U)$ and $y\in Q\cap Fix(T).$ Hence, we conclude that $(x,y)\in \Phi.$ \\

Summing up, we have proved that:

\begin{enumerate}
\item[(i)] for each $(x^{*}, x^{*})\in\Phi,$   the $\underset{n\to\infty}{\lim}\left(\left\|x_{n}-x^{*}\right\|^{2}+\left\|y_{n}-y^{*}\right\|^{2}\right)$ exist;
\item[(ii)] the weak cluster of the sequence $(x_{n}, y_{n})$ belongs to $\Phi.$
\end{enumerate}
Thus, by Lemma (\ref{opial}) we conclude that the  sequences $(x_{n}, y_{n})$ converges weakly to $(x^{*}, x^{*})\in\Phi.$ This completes the proof.
\end{proof}
\end{theorem}

\begin{theorem}\normalfont\label{T7a}
Suppose that all the hypothesis of Theorem \ref{TA} is satisfied. Also, assume that U and T are semi-compacts, then $(x_{n}, y_{n})\to(x^{*}, y^{*})\in\Phi.$

\begin{proof}

As in the proof of Theorem \ref{TA}, $\{u_{n}\}$ and $\{z_{n}\}$ are bounded, by (\ref{K}) and the fact that $U$ and $T$ are semi-compacts, then there exists  sub-sequences    $\{u_{n_{k}}\}$ and $\{z_{n_{k}}\}$ (suppose without loss of generality) of $\{u_{n}\}$ and $\{z_{n}\}$ such that $u_{n_{k}}\to x$ and $z_{n_{k}}\to y.$ Since, $u_{n}\rightharpoonup x^{*}$ and $z_{n}\rightharpoonup y^{*}$, we have $ x=x^{*}$ and $y= y^{*}.$ By (\ref{q}) and the fact that $u_{n_{k}}\to x^{*}$ and $z_{n_{k}}\to y^{*},$ we have
\begin{align}
\underset{n\to\infty}{\lim}\left\|Ax^{*}-Ay^{*}\right\|=\underset{n\to\infty}{\lim}\left\|Au_{n_{k}}-Bz_{n_{k}}\right\|=0,
\end{align}
which tends  to imply that $Ax^{*}=Ay^{*}$. Hence $(x^{*}, y^{*})\in\Phi$. Thus, the iterative algorithm of Theorem \ref{TA} conveges strongly to the solution of Problem \ref{7.6}.

\end{proof}
\end{theorem}

\subsection{The Split Common Fixed Point Equality Problems for Quasi - Nonexpansive Mappings in Hilbert Spaces}

To approximate the solution of  split common fixed point equality problems, we make the following assumptions:
\begin{enumerate}
\item [($A_{1}$)] $T_{1},T_{2},T_{3},...,T_{N}:H_{1}\to H_{1}$ and  $U_{1},U_{2},U_{3},...,U_{M}:H_{2}\to H_{2}$ are
quasi-nonexpansive mappings with  $\bigcap_{i=1}^{N}Fix(T_{i})\neq\emptyset$ and
$\bigcap_{i=1}^{N}Fix(U_{j})\neq\emptyset,$ respectively.
\item [($A_{2}$)] $(T_{i}-I),$ for i=1,2,3,...,N and $(U_{j}-I),$ for  i=1,2,3,...,M are demiclosed at zero.
\item [($A_{3}$)]  $A:H_{1}\to H_{3}$ and $B:H_{2}\to H_{3}$  are  bounded linear operators with their adjoints $A^{*}$ and $B^{*},$ respectively.
\item [($A_{4}$)] For arbitrary $x_{1}\in H_{1}$ and $ y_{1}\in H_{2},$ define $\{(x_{n}, y_{n})\}$ by:
\end{enumerate}
\begin{equation}
 {} \left\{ \begin{array}{ll}
z_{n}=x_{n}-\lambda_{n} A^{*}(Ax_{n}-By_{n}),
\\ w_{n}=(1-\beta_{n})z_{n}+\beta_{n}\sum^{M}_{j=1}{\delta_{j}U_{\gamma_{j}}}(z_{n}),
\\ x_{n+1}=(1-\alpha_{n})z_{n}+\alpha_{n}\sum^{M}_{j=1}{\delta_{j}U_{\gamma_{j}}}(w_{n}),
\\
\\u_{n}=y_{n}+\lambda_{n} B^{*}(Ax_{n}-By_{n}),
\\ r_{n}=(1-\beta_{n})u_{n}+\beta_{n}\sum^{N}_{i=1}{\lambda_{i}T_{\tau_{i}}}(u_{n}),
\\ y_{n+1}=(1-\alpha_{n})u_{n}+\alpha_{n}\sum^{N}_{i=1}{\lambda_{i}T_{\tau_{i}}}(r_{n}),  \forall n\geq 1, & \textrm{ $  $}
 \end{array}  \right.\label{BULA}
\end{equation}
where $U_{\gamma_{j}}=(1-\gamma_{j})I+\gamma_{j}U_{j}$ and $\gamma_{j}\in (0,1),$ for j=1,2,3,...,M,  $T_{\tau_{i}}=(1-\tau_{i})I+\tau_{i}T_{i},$ and $\tau_{i}\in (0,1),$ for i=1,2,3,...,N,  $\sum^{M}_{j=1}{\delta_{j}}=1$ and  $\sum^{N}_{i=1}{\lambda_{i}}=1,$   $0<a<\beta_{n}<1,$ $0<b<\alpha_{n}<1$ and $\lambda_{n}\in\left(0, \frac{2}{L_{1}+L_{2}}\right)$ where $L_{1}=A^{*}A$ and
$L_{2}=B^{*}B,$ respectively.

\begin{theorem}\normalfont\label{T1T} Suppose that conditions $(A_{1})-(A_{4})$  above are satisfied, also, assume that the solution set $\Psi\neq\emptyset.$  Then $(x_{n}, y_{n}) \rightharpoonup (x^{*}, y^{*})\in \Psi.$
\begin{proof}
Let  $(x^{*}, y^{*})\in\Psi$ and $U=\sum^{M}_{j=1}{\delta_{j}U_{\gamma_{j}}}$ and $T=\sum^{N}_{i=1}{\lambda_{i}T_{\tau_{i}}}.$ By Lemma \ref{2T}, we deduce that $U$ and $T$ are quasi nonexpansive mappings, $Fix(U)=\bigcap^{M}_{j=1}Fix(U_{\delta_{j}})= \bigcap^{M}_{j=1}Fix(U_{j})$ and $Fix(T)=\bigcap^{N}_{i=1}Fix(T_{\tau_{i}})= \bigcap^{N}_{i=1}Fix(T_{i}),$ respectively. By Algorithm (\ref{BULA}), we deduce the following algorithm.

\begin{equation}
 {} \left\{ \begin{array}{ll}
z_{n}=x_{n}-\lambda_{n} A^{*}(Ax_{n}-By_{n}),
\\ w_{n}=(1-\beta_{n})z_{n}+\beta_{n}U(z_{n}),
\\ x_{n+1}=(1-\alpha_{n})z_{n}+\alpha_{n}U(w_{n}),
\\
\\u_{n}=y_{n}+\lambda_{n} B^{*}(Ax_{n}-By_{n}),
\\ r_{n}=(1-\beta_{n})u_{n}+\beta_{n}T(u_{n}),
\\ y_{n+1}=(1-\alpha_{n})u_{n}+\alpha_{n}T(r_{n}),  \forall n\geq 1. & \textrm{ $  $}
 \end{array}  \right.\label{BULAL}
\end{equation}
Thus, all the hypothesis of  Theorem \ref{TA} is satisfied. Hence the proof of this theorem follows directly from  Theorem \ref{TA}.

\end{proof}
\end{theorem}

\begin{corollary}\normalfont
Suppose  that  conditions $(B_{1})-(B_{5})$ are satisfied and let  $\{(x_{n}, y_{n})\}$ be the sequence generated  by Algorithm (\ref{BUL}). Assume that $\Phi\neq\emptyset,$ and let $U$ and $T$ be the firmly of quasi-nonexpansive mappings.
Then the sequence  $\{(x_{n}, y_{n})\}$  generated  by Algorithm (\ref{BUL}) converges weakly to the solution  of Problem (\ref{7.6}).
\end{corollary}

\begin{corollary}\normalfont
Suppose  that  conditions $(B_{1})-(B_{4})$  are satisfied and let the sequence $\{(x_{n}, y_{n})\}$ be generated  by
\begin{equation}
 {} \left\{ \begin{array}{ll}
z_{n}=x_{n}-\lambda_{n} A^{*}(Ax_{n}-By_{n}),
\\  x_{n+1}=(1-\alpha_{n})z_{n}+\alpha_{n}U(z_{n}),
\\
\\u_{n}=y_{n}+\lambda_{n} B^{*}(Ax_{n}-By_{n}),
\\ y_{n+1}=(1-\alpha_{n})u_{n}+\alpha_{n}T(y_{n}),  \forall n\geq 0, & \textrm{ $  $}
 \end{array}  \right.\label{BUL3}
\end{equation}
 where $0<a<\beta_{n}<1,$  and $\lambda_{n}\in\left(0, \frac{2}{L_{1}+L_{2}}\right),$ where $L_{1}=A^{*}A$ and $L_{2}=B^{*}B.$  Assume that $\Phi\neq\emptyset.$  Then the sequence  $\{(x_{n}, y_{n})\}$  generated  by Algorithm (\ref{BUL3}) converges weakly to the solution  of SEFPP (\ref{7.3}).

\begin{proof}

Trivially, Algorithm (\ref{BUL}) reduces to Algorithm (\ref{BUL3}) as $\beta=0,$ $P_{C}=P_{Q}=I$ and SFFPEP (\ref{7.4}) reduces to SEFPP (\ref{7.3}) as $C:=Fix(U)$ and $Q:=Fix(T).$ Therefore, all the hypothesis of Theorem  \ref{TA} is satisfied. Hence, the proof of this corollary follows directly from Theorem \ref{TA}.
\end{proof}
\end{corollary}

\begin{corollary}\normalfont
Suppose  that  conditions $(A_{1})-(A_{4})$ are satisfied, and let the sequence $\{(x_{n}, y_{n})\}$ be defined  by Algorithm (\ref{BULA}).  Assume that $\Psi\neq\emptyset$ and let $U$ and $T$ be  firmly  quasi-nonexpansive mappings, where $U=\sum^{M}_{j=1}{\delta_{j}U_{\gamma_{j}}}$ and $T=\sum^{N}_{i=1}{\lambda_{i}T_{\tau_{i}}}.$
Then $(x_{n}, y_{n})\rightharpoonup (x^{*}, x^{*})\in\Psi.$
\end{corollary}

\section{Numerical Example}
In this section, we give the numerical examples that illustrates our theoretical results.
\begin{example}\normalfont\label{example}
Let $H_{1}=\Re$ with the inner product defined by $\left\langle x, y\right\rangle=xy$ for all $x, y\in \Re$ and  $\|.\|$ stands for the corresponding norm. Let $C:=[0,\infty)$  and $Q:=[0,\infty).$ Defined  $T:C\to \Re$  and  $S:Q\to \Re$ by
$Tx = \frac{x^{2}+5}{1+x},$  $\forall x\in C$  and $Sx=\frac{x+5}{5}$, $\forall x\in Q.$  Then $T$ and $S$ are quasi nonexpansive mappings.
\end{example}
\begin{proof}
\end{proof}
Trivially, $Fix(T)=5 $ and $Fix(S)=\frac{5}{4}.$

Now, \begin{align*}\left|Tx-5 \right|&=\left|\frac{x^{2}+5}{1+x}-5\right|
\\&=\frac{x}{1+x}\left|x-5\right|
\\&\leq \left|x-5\right|.
\end{align*}

On the other hand,
\begin{align*}\left|Sx-\frac{5}{4} \right|&=\left|\frac{x+5}{5}-\frac{5}{4}\right|
\\&=\frac{1}{5}\left|x-\frac{5}{4}\right|
\\&\leq \left|x-\frac{5}{4}\right|.
\end{align*}
Thus, $T$ and $S$ are quasi-nonexpansive mappings.

\begin{example}\normalfont\label{ex}Let  $H_{1}=\Re,$ $H_{2}=\Re,$  $C:=[0,\infty),$ and $Q:=[0,\infty)$ be subset of  $H_{1}$ and $H_{2},$ respectively.  Defined  $T:C\to C$ by  $Tx=\frac{x+2}{3}$  $\forall x\in C,$  and  $U:Q\to Q$ by
\begin{equation}
Ux= {} \left\{ \begin{array}{ll} \frac{2x}{x+1}, \forall x\in (1, +\infty)
\\ 0, ~~~\forall x\in[0,1]. & \textrm{ $  $} \end{array}  \right.
\end{equation}
  Then, $U$ and $T$ are quasi nonexpansive mappings.

	\begin{proof}

 Trivially, $Fix(T)=1 $ and $Fix(U)=1.$

 Now,
	\begin{align*}
\left|Tx-1\right|&=\left|\frac{x+2}{3}-1\right|
\\&\leq \left|x-1\right|.
\end{align*}
And also,
\begin{align*}
\left|Ux-1\right|&=\frac{1}{1+x}\left|x-1\right|
\\&\leq \left|x-1\right|.
\end{align*}
Thus, $U$ and $T$ are quasi nonexpansive  mappings.
\end{proof}
\end{example}

The following example is a particular case of Theorem \ref{TA}
\begin{example}\label{bnm}\normalfont Let $H_{1}=\Re$ with the inner product defined by $\left\langle x, y\right\rangle=xy$ for all $x, y\in \Re$ and  $\|.\|$ stands for the corresponding norm. Let $C:=[0,\infty)$ and  $Q:=[0,\infty).$ Defined  $U:C\to \Re$  and  $T:Q\to \Re$ by $Ux = \frac{x^{2}+5}{1+x},$ $\forall x\in C$ and $Tx=\frac{x+5}{5}$, $\forall x\in Q.$ And also let $P_{C}=P_{Q}=I,$  $Ax=x,$ $By=4y,$ $\lambda_{n}=1$, $\alpha_{n}=\frac{1}{5},$ $\beta_{n}=\frac{1}{8}$ and  $\{(x_{n},y_{n})\}$ be the sequence  generated by Algorithm (\ref{BUL}). That is
\begin{equation}
 {} \left\{ \begin{array}{ll} x_{0}\in C{\rm~~and~~} y_{0}\in Q,
\\z_{n}=P_{C}(x_{n}- A^{*}(x_{n}-4y_{n})),
\\ w_{n}=(1-\frac{1}{8})z_{n}+\frac{1}{8}U(z_{n}),
\\ x_{n+1}=(1-\frac{1}{5})z_{n}+\frac{1}{5}U(w_{n}),
\\
\\u_{n}=P_{Q}(y_{n}+ B^{*}(x_{n}-4y_{n})),
\\ r_{n}=(1-\frac{1}{8})u_{n}+\frac{1}{8}T(u_{n}),
\\ y_{n+1}=(1-\frac{1}{5})u_{n}+\frac{1}{5}T(r_{n}),  \forall n\geq 0. & \textrm{ $  $}
 \end{array}  \right.\label{BUL11}
\end{equation}
Then $(x_{n},y_{n})$ converges  to  $(5,5/4)\in\Psi$.

\begin{proof}
By  Example \ref{example} $U$ and $T$ are quasi-nonexpansive mappings. Clearly,   $A$ and $B$ are  bounded linear operator on  $\Re$  with $A=A^{*}=1$ and $B=B^{*}=4,$ respectively. Furthermore, it is easy to see that $Fix(U)=5$ and $Fix(T)=\frac{5}{4}.$ Hence,
\begin{equation*}
\Psi=\Big\{ 5\in C\cap Fix(U) {\rm~and~} 5/4\in Q\cap Fix(T)  {\rm~such ~that~} A(5)=B(5/4)\Big\}.
\end{equation*}

Simplifying Algorithm (\ref{BUL11}), we obtain the following algorithm.
\begin{equation}\label{ab1}
 {} \left\{ \begin{array}{ll} x_{0}\in C{\rm~~and~~} y_{0}\in Q,
\\z_{n}=x_{n},
\\ w_{n}=\frac{7}{8}z_{n}+\frac{1}{8}(\frac{z_{n}^{2}+5}{z_{n}+1}),
\\ x_{n+1}=\frac{4}{5}z_{n}+\frac{1}{5}(\frac{w_{n}^{2}+5}{w_{n}+1}),
\\
\\u_{n}=y_{n},
\\ r_{n}=\frac{7}{8}u_{n}+\frac{1}{8}(\frac{u_{n}+5}{5}),
\\ y_{n+1}=\frac{4}{5}u_{n}+\frac{1}{5}(\frac{r_{n}+5}{5}),  \forall n\geq 0. & \textrm{ $  $}
 \end{array}  \right.
\end{equation}

\end{proof}
\end{example}
\newpage
We used Maple and obtained the numerical values of Algorithm \ref{ab1} in the tables below.

\begin{table}[!htbp]
	\begin{center}
		\caption{Shows the numerical values of Example \ref{bnm} Algorithm (\ref{ab1}), starting with the initial values $x_{0}=10.$ and $y_{0}= 15$}
	\label{figT1T}
		\begin{tabular}{|c|c|c|}
			\hline
			n & $x_{n}$&$y_{n}$\\ [0.5ex]
			\hline
			0 & 10.00000000&  15.00000000\\
			1 & 9.898293685&  12.74500000\\
			
			2 & 9.797736851 &10.85982000\\
			
			3 & 9.698337655 & 9.283809520\\
		
			. &.   &. \\
			
			. &.  &. \\
			
			. &.	& . \\
			
			248 &5.001051418&1.250000002 \\
			
			249 & 5.001012726	 & 1.250000002\\
			
			250 & 5.000975458	&  1.250000002	   \\ [1ex]
			\hline
		\end{tabular}
	\end{center}
	\end{table}

\begin{figure}[!htbp]
\caption{Shows the convergence of Example \ref{bnm}  Algorithm (\ref{ab1}), starting with the initial value $x_{0}=10$ and $y_{0}=15.$}
	\label{fig2}
	\centering
		\includegraphics[scale=0.5]{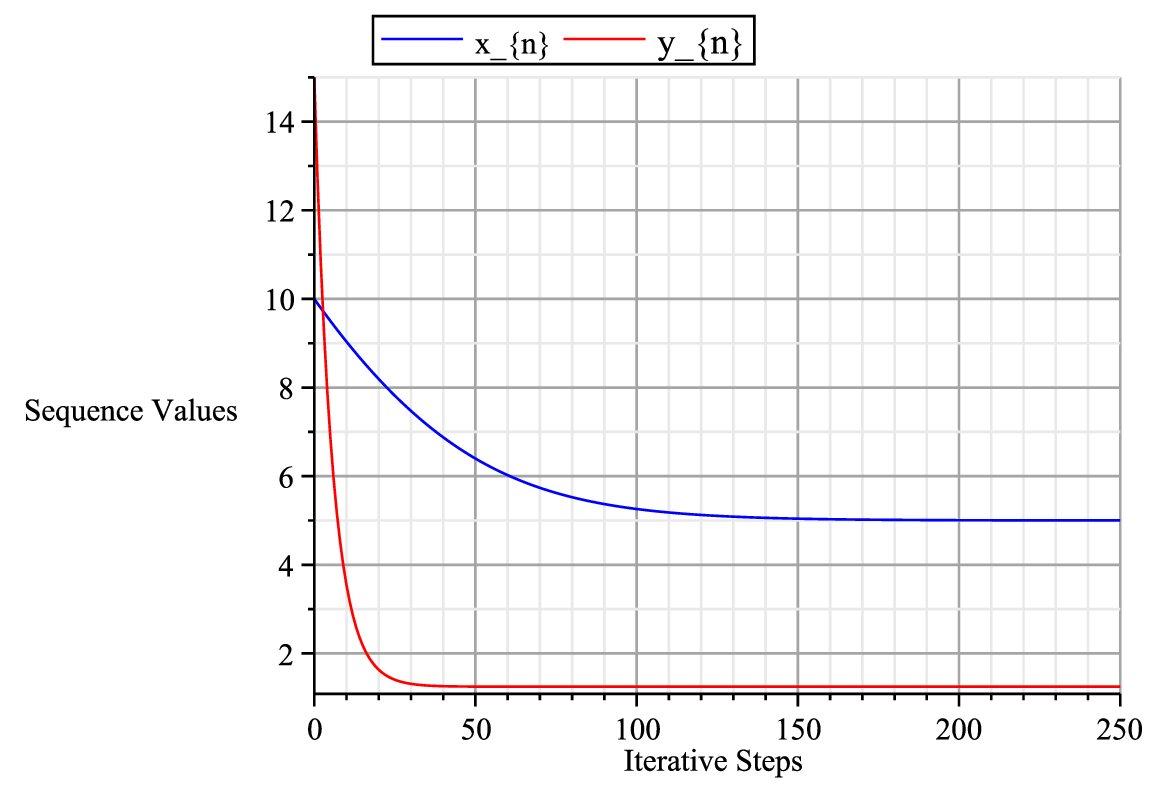}
		\end{figure}

\newpage

\begin{table}[!htbp]
	\begin{center}
		\caption{Shows the numerical values of Example \ref{bnm} Algorithm (\ref{ab1}), starting with the initial values $x_{0}=5$ and $y_{0}= 1.25.$}\label{mnb1}
			\begin{tabular}{|c|c|c|}
			\hline
			n & $x_{n}$&$y_{n}$\\ [0.5ex]
			\hline
			0 & 5.000000000&  1.250000000\\
			
			1 & 5.000000000 &1.250000000\\
			
			2 & 5.000000000 & 1.250000000\\
		
			. &.   &. \\
			
			. &.  &. \\
			
			. &.	& . \\
			
			98 &5.000000000&1.250000000 \\
			
			99 & 5.000000000	 & 1.250000000\\
			
			100 & 5.000000000	&  1.250000000	   \\ [1ex]
			\hline
		\end{tabular}
	\end{center}
	\end{table}

\begin{figure}[!htbp]
\caption{Shows the convergence of Example \ref{bnm}  Algorithm (\ref{ab1}), starting with the initial value $x_{0}=5$ and $y_{0}=1.25$}
	\label{fig42}
	\centering
		\includegraphics[scale=0.5]{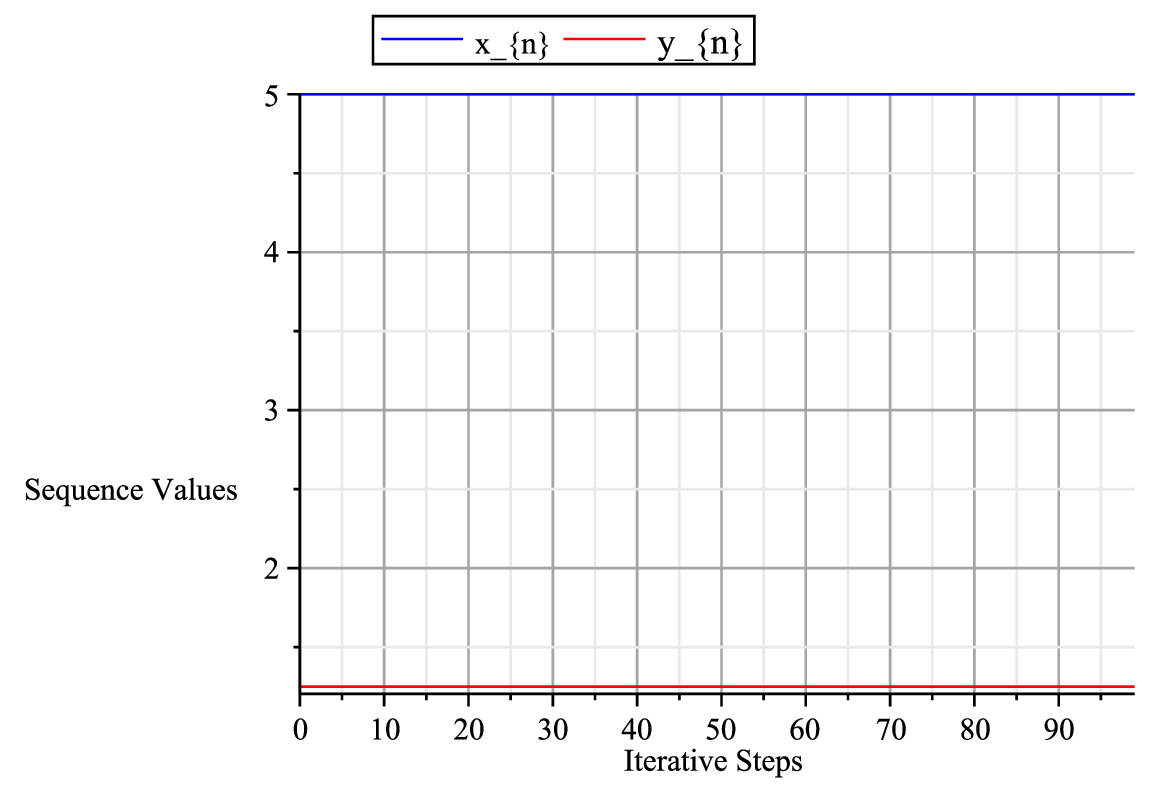}
		\end{figure}
\newpage

The following example is a particular case of Theorem \ref{T1T}
\begin{example}\normalfont\label{wq}
Let   $H_{1}=\Re$ and $H_{2}=\Re,$  $C:=[0,\infty)$ and $Q:=[0,\infty)$ be subset of  $H_{1}$ and $H_{2},$ respectively.  Define  $T:C\to C$ by  $Tx=\frac{x+2}{3}$ $\forall x\in C,$ and  $U:Q\to Q$ by
\begin{equation}
Ux= {} \left\{ \begin{array}{ll} \frac{2x}{x+1}, \forall x\in (1, +\infty)
\\ 0, ~~~\forall x\in[0,1]. & \textrm{ $  $} \end{array}  \right.
\end{equation}
Let also $\lambda_{n}=1,$ $Ax=x,$ $By=y,$ $\gamma_{j}=\frac{1}{3},$ $\tau_{i}=\frac{1}{5},$  $\alpha_{n}=\frac{1}{7}$ and
$\beta_{n}=\frac{1}{9}.$ The sequence  $\{(x_{n},y_{n})\}$  defined by Algorithm  \ref{BULA} can be written as follows:

\begin{equation}
 {} \left\{ \begin{array}{ll}
z_{n}=x_{n}- A^{*}(Ax_{n}-By_{n}),
\\ w_{n}=\frac{8}{9}z_{n}+\frac{1}{9}\left(\frac{2z_{n}}{3}+\frac{2z_{n}}{3(z_{n}+1)}\right),
\\ x_{n+1}=\frac{6}{7}z_{n}+\frac{1}{7}\left(\frac{2w_{n}}{3}+\frac{2w_{n}}{3(w_{n}+1)}\right),
\\
\\u_{n}=y_{n}+ B^{*}(Ax_{n}-By_{n}),
\\ r_{n}=\frac{8}{9}u_{n}+\frac{1}{9}\left(\frac{4u_{n}}{5}+\frac{u_{n}+2}{15}\right),
\\ y_{n+1}=\frac{6}{7}u_{n}+\frac{1}{7}\left(\frac{4r_{n}}{5}+\frac{r_{n}+2}{15}\right),  \forall n\geq 1. & \textrm{ $  $}
 \end{array}  \right.\label{exs}
\end{equation}
Then $(x_{n},y_{n})$ converges  to  $(1,1)\in\Psi.$

\begin{proof}

By Example \ref{ex}, $U$ and $T$ are quasi nonexpansive mappings with $Fix(U)=1$ and $Fix(T)=1,$ respectively. Clearly, $A, B$ are bounded linear on $\Re,$ $A=A^{*}=1$ and $B=B^{*}=1.$ Hence,
$$\Psi=\{1\in Fix(T) {\rm~and~}  1\in Fix(U) {\rm~such ~that~} A(1)=B(1)\}.$$
Simplifying Algorithm (\ref{exs}), we have

\begin{equation}
 {} \left\{ \begin{array}{ll}
z_{n}=y_{n},
\\ w_{n}=\frac{8}{9}z_{n}+\frac{1}{9}\left(\frac{2z_{n}}{3}+\frac{2z_{n}}{3(z_{n}+1)}\right),
\\ x_{n+1}=\frac{6}{7}z_{n}+\frac{1}{7}\left(\frac{2w_{n}}{3}+\frac{2w_{n}}{3(w_{n}+1)}\right),
\\
\\u_{n}=x_{n},
\\ r_{n}=\frac{8}{9}u_{n}+\frac{1}{9}\left(\frac{4u_{n}}{5}+\frac{u_{n}+2}{15}\right),
\\ y_{n+1}=\frac{6}{7}u_{n}+\frac{1}{7}\left(\frac{4r_{n}}{5}+\frac{r_{n}+2}{15}\right),  \forall n\geq 1. & \textrm{ $  $}
 \end{array}  \right.\label{exs1}
\end{equation}

\end{proof}
	\end{example}

	\newpage

\begin{table}[!htbp]
	\begin{center}
		\caption{Shows the numerical values of Example \ref{wq} Algorithm (\ref{exs}), starting with the initial values $x_{0}=5$ and $y_{0}= 5.$}
	\label{fig1}
		\begin{tabular}{|c|c|c|}
			\hline
			n & $x_{n}$&$y_{n}$\\ [0.5ex]
			\hline
			0 & 5.000000000&   5.000000000\\
			1 & 4.916472663&  4.760850019  \\
			
			2 &4.834689530  &4.537828465\\
			
			3 &4.754614179 & 4.329771078\\
		
			. &.   &. \\
			
			. &.  &. \\
			.&.&.\\
			148 &1.176058095	& 1.007392532  \\
						
			149 & 1.172381679	& 1.007122340	   \\ [1ex]
			\hline
		\end{tabular}
	\end{center}
	\end{table}

\begin{figure}[!htbp]
\caption{Shows the convergence of Example \ref{wq} Algorithm (\ref{exs}), starting with the initial value $x_{0}=5$ and $y_{0}=5.$}
	\label{figT2T}
	\centering
		\includegraphics[scale=0.5]{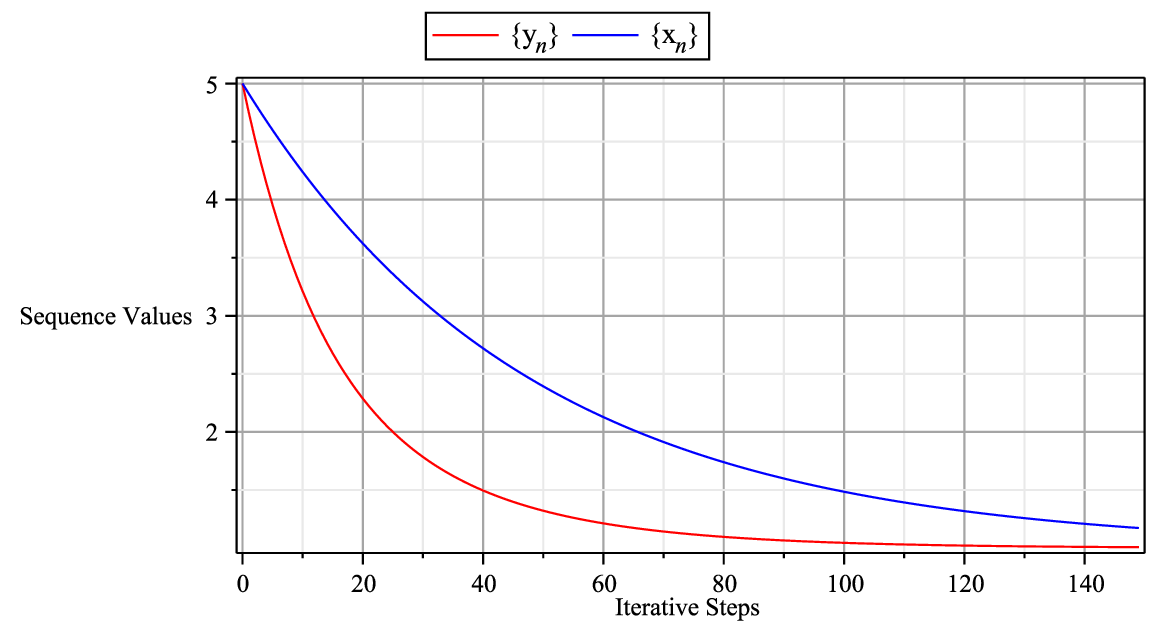}
		\end{figure}

	\newpage

\begin{table}[!htbp]
	\begin{center}
		\caption{Shows the numerical values of Example \ref{wq} Algorithm (\ref{exs}), starting with the initial values $x_{0}=-5.$ and $y_{0}= -5$}
	\label{tab:StartingWithInialValue101}
		\begin{tabular}{|c|c|c|}
			\hline
			n & $x_{n}$&$y_{n}$\\ [0.5ex]
			\hline
			0 &-5.000000000&  -5.000000000\\
			
			1 &-4.460475401 &  -4.874708995\\
			
			2 & -3.953349994 & -4.752034296\\
		
			. &-3.474475616  &-4.631921270 \\
			
			. &.  &. \\
			. &.  &. \\
			. &.  &. \\
			
			148 &1.001346412 	& 0 .7359128532 \\

			149 &1.001297344	& 0.7414274772	   \\ [1ex]
			\hline
		\end{tabular}
	\end{center}
	\end{table}
	
	\begin{figure}[!htbp]
\caption{Shows the convergence of Example \ref{wq} Algorithm (\ref{exs}), starting with the initial value $x_{0}=-5$ and $y_{0}=-5.$}\label{fig4}
	\centering
		\includegraphics[scale=0.5]{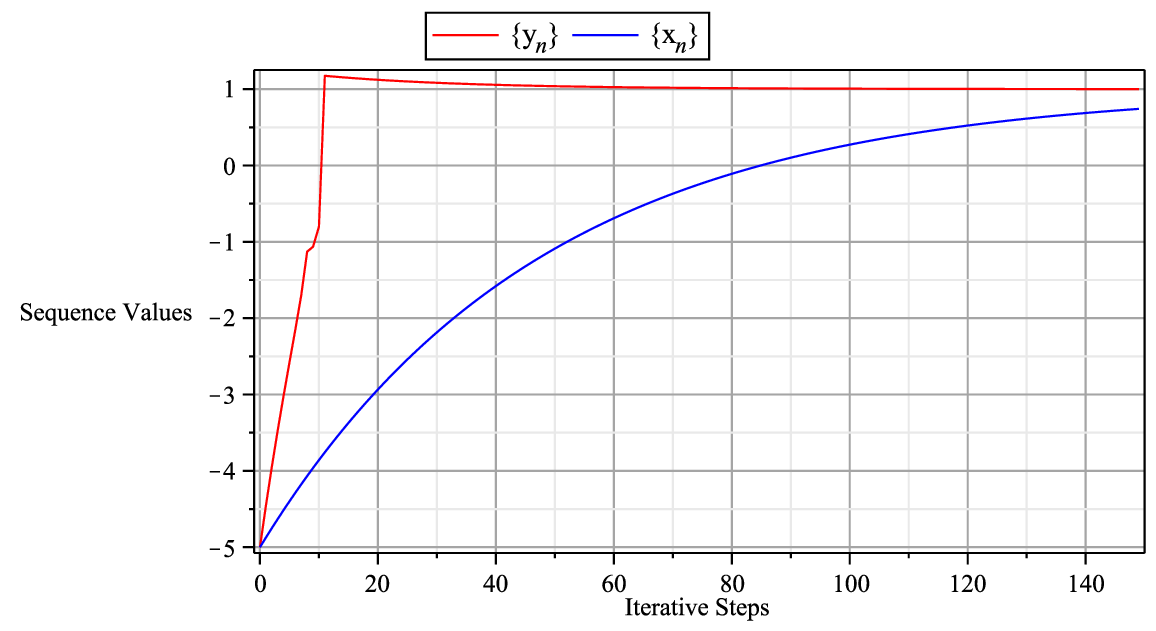}
		\end{figure}

\section{The split feasibility and fixed point  problems for quasi-nonexpansive mappings  in Hilbert spaces }
In this section, we propose  Ishikawa-type extra-gradient algorithms for solving the split feasibility and fixed point problems. Under some suitable assumptions imposed on some parameters and operators involved, we prove the strong convergence theorems of these algorithms.

\subsection{Problem Formulation}

The split feasibility and fixed point problems (SFFPP)  required to find a vector
\begin{equation}x^{*}\in C\cap Fix(T_{1}) {\rm~such~that~} Ax^{*}\in Q\cap Fix(T_{2}),\label{6.1}
\end{equation}
where $T_{1}:H_{1}\to H_{1}$ and $T_{2}:H_{2}\to H_{2}$ are  quasi-nonexpansive mappings,
and $A:H_{1}\to H_{2}$ is a bounded linear operator.\\

We denote the solution set of  Problem (\ref{6.1}) by 
\begin{equation}
\Delta=\Big\{ x^{*}\in C\cap Fix(T_{1}){\rm~such~that~} Ax^{*}\in Q\cap Fix(T_{2})\Big\}. \label{6.3}
\end{equation}
In sequel, we assume that $\Delta\neq\emptyset.$

\subsection{Preliminary Results}

The following well-known results are significant in proving the main result of this chapter.

\begin{lemma}\normalfont\label{lem6.1} Let  $C$ be a nonempty closed convex subset of a Hilbert space.
\begin{itemize}
\item[(i)] If $G$ and $S$ are two quasi-nonexpansive mappings on C, then $GS$ is also  quasi- nonexpansive mapping.
\item [(ii)] Let $G_{\alpha}=(1-\alpha)I+\alpha G,$ where $\alpha\in (0,1]$ and  $G$ is a quasi-nonexpansive mapping on C.  Then for all $x\in C$ and $q\in Fix(G),$  $G_{\alpha}$ is also a quasi-nonexpansive.
\item [(iii)]  Let $\{G_{i}\}^{N}_{i=1}: C\to C$ be N-quasi-nonexpansive mappings and $\{\alpha_{i}\}^{N}_{i=1}$ be a positive sequence in (0,1) such that $\sum^{N}_{i=1}\alpha_{i}=1.$   Suppose that $\{G_{i}\}^{N}_{i=1}$  has a fixed point. Then $$Fix\left(\sum^{N}_{i=1}{\alpha_{i}G_{i}}\right)= \bigcap^{N}_{i=1}Fix(G_{i}).$$
	\item [(iv)] Let $\{G_{i}\}^{N}_{i=1}$ and $\{\alpha_{i}\}^{N}_{i=1}$ be as in (iii) above. Then
	$\sum^{N}_{i=1}{\alpha_{i}G_{i}},$ is a quasi-nonexpansive mapping. Furthermore, if for each $i=1,2,3,...N,$ $G_{i}-I$ is demiclosed at zero, then  $\sum^{N}_{i=1}{\alpha_{i}G_{i}}-I$ is also.
			\end{itemize}
	\end{lemma}
	The proof of (i) follows trivially, for the proof of (ii) see Moudafi \cite{moudafi2011note} while the proof of (iii) and (iv) are deduce from Li and He \cite{li2015new}.
	\newpage

\section{Ishikawa-type Extra-Gradient Iterative Methods for Quasi- Nonexpansive Mappings in Hilbert Spaces}

\begin{theorem}\normalfont\label{T6.1}
Let     $T:C\to H_{1}$  and $G:Q\to H_{2}$ be two quasi nonexpansive mappings and $A:H_{1}\to H_{2}$ be a bounded linear operator  with its adjoint  $A^{*}.$   Assume that $(T-I)$ and $(GP_{Q}-I)$ are demiclosed at zero, and   $\Delta\neq\emptyset.$  Define  $\{x_{n}\}$  by
\begin{equation}
 {} \left\{ \begin{array}{ll} x_{0}\in C~ {\rm ~chosen~arbitrarily,}\label{6.10}
\\ y_{n}=P_{C}(x_{n}-\gamma_{n} A^{*}(I-GP_{Q})Ax_{n}),
\\ z_{n}=P_{C}(y_{n}-\gamma_{n} A^{*}(I-GP_{Q})Ay_{n}),
\\ w_{n}=(1-\alpha_{n})z_{n}+\alpha_{n}T\Big((1-\beta_{n})z_{n}+\beta_{n}Tz_{n}\Big),
\\C_{n+1}=\Big\{z\in C_{n}: \|w_{n}-z\|^{2}\leq \|z_{n}-z\|^{2}\leq \|y_{n}-z\|^{2}\leq \|x_{n}-z\|^{2}\Big\},
\\x_{n+1}=P_{C_{n+1}}(x_{0}),
 \forall n\geq 0, & \textrm{ $  $}
 \end{array}  \right.
\end{equation}

where $P$ is a projection operator,  $0<a<\alpha_{n}<1,$ $0<b<\beta_{n}<1,$ and  $0<c<\gamma_{n}<\frac{1}{L},$  with
 $L=\|AA^{*}\|.$ Then $x_{n}\to x^{*}\in\Delta$.
\end{theorem}
\begin{proof}

\textbf{Step 1.} First, we show that $P_{C_{n+1}}$ is well defined. To show this, it  suffices to show that  for each $n\geq 0,$ $C_{n}$ is  closed and convex. Trivially,  $C_{n}$ is  closed. \\

Next, we show that  $C_{n}$ is  convex. To show this, it suffices to show that for each  $r_{1}, r_{2}\in C_{n}$ and $\xi\in(0,1),$ $\xi r_{1}+ (1-\xi)r_{2}\in C_{n}.$\\

Now, we compute
\begin{align}
\|w_{n}-\xi r_{1}- (1-\xi)r_{2}\|^{2}&=\|\xi(w_{n}- r_{1})+(1-\xi)(w_{n}-r_{2})\|^{2}\nonumber
\\&=\xi\|w_{n}-r_{1}\|^{2}+(1-\xi)\|w_{n}-r_{2}\|^{2}-(1-\xi)\xi\|r_{1}-r_{2}\|^{2}\nonumber
\\&\leq\xi\|z_{n}-r_{1}\|^{2}+(1-\xi)\|z_{n}-r_{2}\|^{2}-(1-\xi)\xi\|r_{1}-r_{2}\|^{2}\nonumber
\\&=\|z_{n}-\xi r_{1}- (1-\xi)r_{2}\|^{2}.\label{6.11}
\end{align}
Similarly, we  obtain that
\begin{align}
\|z_{n}-\xi r_{1}- (1-\xi)r_{2}\|^{2}&\leq\|y_{n}-\xi r_{1}- (1-\xi)r_{2}\|^{2}\nonumber
\\&\leq\|x_{n}-\xi r_{1}- (1-\xi)r_{2}\|^{2}.\label{6.12}
\end{align}
Thus,   for each  $r_{1}, r_{2}\in C_{n},$  $\xi r_{1}+ (1-\xi)r_{2}\in C_{n}.$ \\

\textbf{Step 2.} Here, we show that $\Delta\subset C_{n},$ $n\geq 0.$\\

Let $q\in\Delta$ and $u_{n}=(1-\beta_{n})z_{n}+\beta_{n}Tz_{n}.$ The fact that $T$ is quasi-nonexpansive,
it follows from (\ref{6.10})  that
\begin{align}
\left\|w_{n}-q\right\|^{2}&=\left\|(1-\alpha_{n})z_{n}+\alpha_{n}Tu_{n}-q\right\|^{2}\nonumber
\\&=\left\|(1-\alpha_{n})(z_{n}-q)+\alpha_{n}(Tu_{n}-q)\right\|^{2}\nonumber
\\&=(1-\alpha_{n})\left\|z_{n}-q\right\|^{2}+\alpha_{n}\left\|Tu_{n}-q\right\|^{2}-\alpha_{n}(1-\alpha_{n})\left\|Tu_{n}-z_{n}\right\|^{2}\nonumber
\\&\leq(1-\alpha_{n})\left\|z_{n}-q\right\|^{2}+\alpha_{n}\left\|u_{n}-q\right\|^{2}-\alpha_{n}(1-\alpha_{n})\left\|Tu_{n}-z_{n}\right\|^{2}\nonumber
\\&\leq(1-\alpha_{n})\left\|z_{n}-q\right\|^{2}+\alpha_{n}\left\|(1-\beta_{n})(z_{n}-q)+\beta_{n}(Tz_{n}-q)\right\|^{2}\nonumber
\\&-\alpha_{n}(1-\alpha_{n})\left\|Tu_{n}-z_{n}\right\|^{2}\nonumber
\\&= (1-\alpha_{n})\left\|z_{n}-q\right\|^{2}+\alpha_{n} (1-\beta_{n})\left\|z_{n}-q\right\|^{2} +\alpha_{n}\beta_{n}\left\|Tz_{n}-q\right\|^{2}\nonumber
\\&-\alpha_{n}\beta_{n}(1-\beta_{n})\left\|Tz_{n}-z_{n}\right\|^{2}-\alpha_{n}(1-\alpha_{n})\left\|Tu_{n}-z_{n}\right\|^{2}\nonumber
\\&\leq (1-\alpha_{n})\left\|z_{n}-q\right\|^{2}+\alpha_{n} (1-\beta_{n})\left\|z_{n}-q\right\|^{2} +\alpha_{n}\beta_{n}\left\|z_{n}-q\right\|^{2}\nonumber
\\&-\alpha_{n}\beta_{n}(1-\beta_{n})\left\|Tz_{n}-z_{n}\right\|^{2}-\alpha_{n}(1-\alpha_{n})\left\|Tu_{n}-z_{n}\right\|^{2}\nonumber
\\&\leq \left\|z_{n}-q\right\|^{2}.\label{6.13}
\end{align}
On the other hand, since $G$ and $P_{C}$ are both quasi-nonexpansive, by Lemma (\ref{lem6.1}), we obtain that $GP_{C}$ is also quasi-nonexpansive. Thus, we have
\begin{align}
\left\|z_{n}-q\right\|^{2}&=\left\|P_{C}(y_{n}-\gamma_{n} A^{*}(I-GP_{Q})Ay_{n})-q\right\|^{2}\nonumber
\\&\leq \left\|y_{n}-\gamma_{n} A^{*}(I-GP_{Q})Ay_{n}-q\right\|^{2}\nonumber
\\&= \left\|y_{n}-q\right\|^{2}-2\gamma_{n}\left\langle  y_{n}-q, A^{*}(I-GP_{Q})Ay_{n}\right\rangle+\left\|\gamma_{n} A^{*}(I-GP_{Q})Ay_{n}\right\|^{2}\nonumber
\\&= \left\|y_{n}-q\right\|^{2}-2\gamma_{n}\left\langle  Ay_{n}-GP_{Q}Ay_{n}+GP_{Q}Ay_{n}-Aq, Ay_{n}-GP_{Q}Ay_{n}\right\rangle\nonumber
\\&+\gamma_{n}^{2}L\left\|(I-GP_{Q})Ay_{n}\right\|^{2}\nonumber
\\&= \left\|y_{n}-q\right\|^{2}+2\gamma_{n}\left\langle  Aq -GP_{Q}Ay_{n}, Ay_{n}-GP_{Q}Ay_{n}\right\rangle\nonumber
\\&-\gamma_{n}(2-\gamma_{n}L)\left\|GP_{Q}Ay_{n}-Ay_{n}\right\|^{2}\nonumber
\\&\leq \left\|y_{n}-q\right\|^{2}-\gamma_{n}(1-\gamma_{n}L)\left\|GP_{Q}Ay_{n}-Ay_{n}\right\|^{2}\label{6.14}.
\end{align}
Following the same way as in the proof of  (\ref{6.14}), we obtain that
\begin{align}
\left\|y_{n}-q\right\|^{2}&\leq \left\|x_{n}-q\right\|^{2}.\label{6.15}
\end{align}
Combine Equation (\ref{6.13}) $-$ (\ref{6.15}),
we have
\begin{align}
\|w_{n}-q\|^{2}\leq \|z_{n}-q\|^{2}\leq \|y_{n}-q\|^{2}\leq \|x_{n}-q\|^{2}.\label{6.16}
\end{align}
Thus, we have that $q\in C_{n},$ this implies that $ \Delta\subset C_{n}.$\\

Noticing that $\Delta\subset C_{n+1}\subset C_{n}$ and  $x_{n+1}=P_{C_{n+1}}(x_{0})\subset C_{n},$
 we have that
\begin{equation}
\|x_{n+1}-x_{0}\|\leq\|q-x_{0}\|, \forall n\geq 0 {\rm~and~} q\in\Delta.\label{6.17}
\end{equation}
This shows that $\{x_{n}\}$ is bounded. By Lemma (\ref{l5}) we have
\begin{align}
\|x_{n+1}-x_{n}\|^{2}+\|x_{n+1}-x_{0}\|^{2}&=\|P_{C_{n+1}}(x_{0})-x_{n}\|^{2}+\|P_{C_{n+1}}(x_{0})-x_{0}\|^{2}\nonumber
\\&\leq \|x_{n}-x_{0}\|^{2}.\label{6.18}
\end{align}
This implies that
\begin{align*}
\|x_{n+1}-x_{0}\|&\leq \|x_{n}-x_{0}\|.
\end{align*}
Thus,  $\{\|x_{n}-x_{0}\|\}$ is a non-increasing sequence and  bounded below by zero. Therefore,  $\underset{n\to\infty}\lim \|x_{n}-x_{0}\|$ exists. \\

On the other hand, for each $k>n,$ we also obtain that
\begin{align}
\|x_{k}-x_{n}\|^{2}+\|x_{n}-x_{0}\|^{2}&=\|P_{C_{n}}(x_{0})-x_{k}\|^{2}+\|P_{C_{n}}(x_{0})-x_{0}\|^{2}\nonumber
\\&\leq \|x_{k}-x_{0}\|^{2}.\label{6.19}
\end{align}
  Thus, by (\ref{6.19}) and the fact that the $\underset{n\to\infty}\lim \|x_{n+1}-x_{0}\|$ exist, we  obtain that
	$$\underset{k,n\to\infty}\lim \|x_{k}-x_{n}\|=0.$$
	This shows that $\{x_{n}\}$ is  Cauchy.\\

Since $x_{n+1}=P_{C_{n+1}}(x_{0})\in C_{n+1}\subset C_{n}$ and the fact that $\{x_{n}\}$ is a Cauchy sequence, we deduce that
\begin{align}
\|y_{n}-x_{n}\|&\leq \|y_{n}-x_{n+1}\|+\|x_{n+1}-x_{n}\|\nonumber
\\&\leq 2\|x_{n+1}-x_{n}\|,\nonumber
\end{align}
and
\begin{align}
\|z_{n}-x_{n}\|&\leq \|z_{n}-x_{n+1}\|+\|x_{n+1}-x_{n}\|\nonumber
\\&\leq 2\|x_{n+1}-x_{n}\|.\nonumber
\end{align}
Thus, as $n\to\infty,$  we deduce that
\begin{align}
\underset{n\to\infty}{\lim}\|z_{n}-x_{n}\|=0 {\rm~and~} \underset{n\to\infty}{\lim}\|y_{n}-x_{n}\|=0.\label{876}
\end{align}
On the other hand,
\begin{align*}
\|y_{n}-z_{n}\|&\leq \|y_{n}-x_{n}\|+\|x_{n}-z_{n}\|.
\end{align*}
By (\ref{876}) we obtain that
\begin{align}\underset{n\to\infty}{\lim}\|y_{n}-z_{n}\|=0.\label{6.21}
\end{align}
Similarly, we obtain that
\begin{align}
\underset{n\to\infty}{\lim}\|w_{n}-x_{n}\|= 0, \underset{n\to\infty}{\lim}\|z_{n}-y_{n}\|=0
{\rm~and~}\underset{n\to\infty}{\lim}\|w_{n}-z_{n}\|=0.\label{mnb}
\end{align}
By (\ref{6.14}),   we obtain that
\begin{align}
\left\|GP_{Q}Ay_{n}-Ay_{n}\right\|^{2}&\leq \frac{\left\|y_{n}-x^{*}\right\|^{2}-\left\|z_{n}-x^{*}\right\|^{2}}{\gamma_{n}(1-\gamma_{n}L)}\nonumber
\\&\leq \frac{\left\|y_{n}-z_{n}\right\|\Big(\left\|z_{n}-y_{n}\right\|+2\left\|z_{n}-x^{*}\right\|\Big)}{c(1-\gamma_{n}L)}.\nonumber
\end{align}
Thus, by  (\ref{6.21}) we deduce that
\begin{align*}
\underset{n\to\infty}{\lim}\left\|GP_{Q}Ay_{n}-Ay_{n}\right\|=0.
\end{align*}
Similarly, by (\ref{6.13})  and (\ref{mnb}), we deduce that
\begin{align}
\underset{n\to\infty}{\lim}\left\|Tz_{n}-z_{n}\right\|=0.\nonumber
\end{align}
  Finally, we show that  $x_{n}\to x^{*}.$\\
	
	Since $\{x_{n}\}$ is  Cauchy,  we  assume that $x_{n}\to p.$ By Equation (\ref{6.10}), we have that $z_{n}\to p,$ this implies that $z_{n}\rightharpoonup p.$ The fact that the $\underset{n\to\infty}{\lim}\|Tz_{n}-z_{n}\|=0$ together with the demiclosed of $(T-I)$ at zero, we deduce that $p\in Fix(T).$\\
	
On the other hand, since $x_{n}\to p,$  implies that $y_{n}\rightharpoonup p,$ by (\ref{6.10}), we deduce that
$p=P_{C}p$ which implies that $p\in C$
and therefore we have $p\in C\cap Fix(T).$	\\

Furthermore,  by the definition of $A,$   we have that $Ay_{n}\to Ap,$ this implies that $Ay_{n}\rightharpoonup Ap.$ The fact that the $\underset{n\to\infty}{\lim}\|GP_{Q}Ay_{n}-Ay_{n}\|=0$ together with the demiclosed of $(GP_{Q}-I)$ at zero, we deduce that $Ap\in Fix(GP_{Q}),$ this implies that $	Ap\in Q\cap Fix(G).$  Hence $p\in \Delta.$ This show that $x_{n}\to x^{*}.$ The  proof is complete.
\end{proof}\vspace{6mm}\newpage

 Next, we consider the split feasibility and fixed point problems for the class of finite family of quasi-nonexpansive mappings.

\begin{theorem}\normalfont\label{T62}
Let   $\{T_{i}\}^{M}_{i=1}:C\to H_{1}$  and $\{G_{j}\}^{N}_{j=1}:Q\to H_{2}$ be  quasi nonexpansive mappings with $\bigcap^{M}_{i=1}Fix(T_{i})\neq\emptyset$ and $\bigcap^{N}_{j=1}Fix(G_{j})\neq\emptyset.$ And also let $A:H_{1}\to H_{2}$ be a bounded linear operator with its adjoint $A^{*}.$   Assume that $(T_{i}-I),i=1,2,3,...,M$ and $(G_{j}P_{Q}-I), j=1,2,3,...,N$ are demiclosed at zero, and  $\Delta\neq\emptyset.$  Define  $\{x_{n}\}$  by
\begin{equation}
 {} \left\{ \begin{array}{ll} x_{0}\in C~ {\rm ~chosen~arbitrarily,}\label{6.22}
\\ y_{n}=P_{C}(x_{n}-\gamma_{n} A^{*}(I-\sum^{N}_{j=1}\delta_{j}G_{j}P_{Q})Ax_{n}),
\\ z_{n}=P_{C}(y_{n}-\gamma_{n} A^{*}(I-\sum^{N}_{j=1}\delta_{j}G_{j}P_{Q})Ay_{n}),
\\ w_{n}=(1-\alpha_{n})z_{n}+\alpha_{n}\sum^{M}_{i=1}\lambda_{i}T_{i}\Big((1-\beta_{n})z_{n}+\beta_{n}\sum^{M}_{i=1}\lambda_{i}T_{i}z_{n}\Big),
\\C_{n+1}=\Big\{z\in C_{n}: \|w_{n}-z\|^{2}\leq \|z_{n}-z\|^{2}\leq \|y_{n}-z\|^{2}\leq \|x_{n}-z\|^{2}\Big\},
\\x_{n+1}=P_{C_{n+1}}(x_{0}),
 \forall n\geq 0, & \textrm{ $  $}
 \end{array}  \right.
\end{equation}
where $P$ is a projection operator,  $0<a<\alpha_{n}<1,$ $0<b<\beta_{n}<1,$ and  $0<c<\gamma_{n}<\frac{1}{L}$ with $L=\|AA^{*}\|.$  Then $x_{n}\to x^{*}\in\Delta.$
\end{theorem}

\begin{proof}
\end{proof}

By Lemma \ref{lem6.1}, we deduce that
\begin{itemize}
\item [(i)] $\sum^{N}_{j=1}\delta_{j}G_{j}$ and $\sum^{M}_{i=1}\lambda_{i}T_{i}$ are quasi-nonexpansive mappings.
\item[(ii)] $\sum^{N}_{j=1}\delta_{j}(G_{j}-I)$ and $\sum^{M}_{i=1}\lambda_{i}(T_{i}-I)$ are demiclosed at zero.
\item [(iii)] $Fix\left(\sum^{M}_{i=1}{\lambda_{i}T_{i}}\right)= \bigcap^{M}_{i=1}Fix(T_{i})$ and $Fix\left(\sum^{N}_{j=1}{\delta_{j}G_{j}}\right)= \bigcap^{N}_{j=1}Fix(G_{j}).$
\end{itemize}
Thus, all the hypothesis of Theorem \ref{T62} is satisfied. Therefore, the proof of this theorem follows trivially from
Theorem \ref{T6.1}.\\

As the consequence of Theorem \ref{T62}, we immediately obtain the following corollary.

\begin{corollary}\normalfont\label{TB7.16}
Let  $\{T_{i}\}^{M}_{i=1}:C\to H_{1}$  and $\{G_{j}\}^{N}_{j=1}:Q\to H_{2}$ be  quasi nonexpansive mappings with $\bigcap^{M}_{i=1}Fix(T_{i})\neq\emptyset$ and $\bigcap^{N}_{j=1}Fix(G_{j})\neq\emptyset,$ respectively. And also let $A:H_{1}\to H_{2}$ be a bounded linear operator with its adjoint $A^{*}.$  Assume that $(T_{i}-I),i=1,2,3,...,M$ and $(G_{j}P_{Q}-I), j=1,2,3,...,N$ are demiclosed at zero and    $\Delta\neq\emptyset.$  Define $\{x_{n}\}$ by
\begin{equation}
 {} \left\{ \begin{array}{ll} x_{0}\in C~ {\rm ~chosen~arbitrarily,}\label{6.23}
\\ z_{n}=P_{C}(x_{n}-\gamma_{n} A^{*}(I-\sum^{N}_{j=1}\delta_{j}G_{j}P_{Q})Ax_{n}),
\\ w_{n}=(1-\alpha_{n})z_{n}+\alpha_{n}\sum^{M}_{i=1}\lambda_{i}T_{i}z_{n},
\\C_{n+1}=\Big\{z\in C_{n}: \|w_{n}-z\|^{2}\leq \|z_{n}-z\|^{2}\leq \|x_{n}-z\|^{2}\Big\},
\\x_{n+1}=P_{C_{n+1}}(x_{0}),
 \forall n\geq 0, & \textrm{ $  $}
 \end{array}  \right.
\end{equation}
where $P$ is a projection operator,  $0<a<\alpha_{n}<1,$ $0<b<\beta_{n}<1,$ and  $0<c<\gamma_{n}<\frac{1}{L}$ with $L=\|AA^{*}\|.$  Then $x_{n}\to x^{*}\in\Delta$.
\end{corollary}

\begin{proof}

In Algorithm (\ref{6.22}), take $y_{n}=x_{n}$ and $\beta_{n}=0,$ then, Algorithm (\ref{6.22}) reduces to  Algorithm (\ref{6.23});  therefore, all the hypothesis of Theorem \ref{T62} is satisfied. Hence, the proof of this corollary follows directly from Theorem \ref{T62}.
\end{proof}

\subsection{Application to Split Feasibility Problems}

As a special case of Problem (\ref{6.3}), we  give the following theorems for solving split feasibility Problem and the fixed point problem.

\begin{theorem}\normalfont\label{T6.3}
Let $\{T_{i}\}^{M}_{i=1}:C\to H_{1}$  and $\{G_{j}\}^{N}_{j=1}:Q\to H_{2}$ be  quasi-nonexpansive mappings with $\bigcap^{M}_{i=1}Fix(T_{i})\neq\emptyset$ and $\bigcap^{N}_{j=1}Fix(G_{j})\neq\emptyset.$ And also let $A:H_{1}\to H_{2}$ be a bounded linear operator with its adjoint $A^{*}.$  Assume that $(T_{i}-I),i=1,2,3,...,M$ and $(G_{j}-I), j=1,2,3,...,N$ are demiclosed at zero and   $\Delta\neq\emptyset.$  Define $\{x_{n}\}$  by
\begin{equation}
 {} \left\{ \begin{array}{ll} x_{0}\in C~ {\rm ~chosen~arbitrarily,}\label{6.25}
\\ y_{n}=P_{C}(x_{n}-\gamma_{n} A^{*}(I-\sum^{N}_{j=1}\delta_{j}G_{j})Ax_{n}),
\\ z_{n}=P_{C}(y_{n}-\gamma_{n} A^{*}(I-\sum^{N}_{j=1}\delta_{j}G_{j})Ay_{n}),
\\ w_{n}=(1-\alpha_{n})z_{n}+\alpha_{n}\sum^{M}_{i=1}\lambda_{i}T_{i}\Big((1-\beta_{n})z_{n}+\beta_{n}\sum^{M}_{i=1}\lambda_{i}T_{i}z_{n}\Big),
\\C_{n+1}=\Big\{z\in C_{n}: \|w_{n}-z\|^{2}\leq \|z_{n}-z\|^{2}\leq \|y_{n}-z\|^{2}\leq \|x_{n}-z\|^{2}\Big\},
\\x_{n+1}=P_{C_{n+1}}(x_{0}),
 \forall n\geq 0, & \textrm{ $  $}
 \end{array}  \right.
\end{equation}
where $P$ is a projection operator,  $0<a<\alpha_{n}<1,$ $0<b<\beta_{n}<1,$ and  $0<c<\gamma_{n}<\frac{1}{L}$ with $L=\|AA^{*}\|.$  Then $x_{n}\to x^{*} \in\Omega.$
\end{theorem}
\begin{proof}

In Algorithm (\ref{6.22}), take $P_{Q}=I$ (the identity mapping),  then, Algorithm (\ref{6.22}) reduces to  Algorithm (\ref{6.25}),  therefore, all the hypothesis of Theorem \ref{T62} is satisfied. Hence, the proof of this theorem, follows directly from Theorem \ref{T62}.
\end{proof}

\begin{theorem}\normalfont\label{T6.4}
Let  $\{T_{i}\}^{M}_{i=1}:C\to H_{1}$  be  quasi-nonexpansive mapping with \\$\bigcap^{M}_{i=1}Fix(T_{i})$$\neq$$\emptyset.$    Assume that $(T_{i}-I),i=1,2,3,...,M$ is demiclosed at zero.  Define $\{x_{n}\}$  by
\begin{equation}
 {} \left\{ \begin{array}{ll} x_{0}\in C~ {\rm ~chosen~arbitrarily,}\label{6.26}
\\ u_{n}=(1-\beta_{n})x_{n}+\beta_{n}\sum^{M}_{i=1}\lambda_{i}T_{i}x_{n},
\\ w_{n}=(1-\alpha_{n})u_{n}+\alpha_{n}\sum^{M}_{i=1}\lambda_{i}T_{i}u_{n},
\\C_{n+1}=\Big\{z\in C_{n}: \|w_{n}-z\|^{2}\leq \|u_{n}-z\|^{2}\leq \|x_{n}-z\|^{2}\Big\},
\\x_{n+1}=P_{C_{n+1}}(x_{0}),
 \forall n\geq 0, & \textrm{ $  $}
 \end{array}  \right.
\end{equation}
where $P$ is a projection operator,  $0<a<\alpha_{n}<1$ and $0<b<\beta_{n}<1.$    Then $\{x_{n}\}$ converges strongly to the solution of common fixed point of $\{T_{i}\}^{M}_{i=1}$.
\end{theorem}
\begin{proof}

In Algorithm (\ref{6.22}), take $\gamma_{n}=0$ and  $P_{C}=I$ (the identity mapping),  then, Algorithm (\ref{6.22}) reduces to  Algorithm (\ref{6.26}),  therefore; all the hypothesis of Theorem \ref{T62} is satisfied. Hence, the proof of this theorem, follows directly from Theorem \ref{T62}.
\end{proof}

\subsubsection{Conclusion}
In this section,  we have proposed Ishikawa-type extra-gradient methods for solving the split feasibility and fixed point problems for the class of quasi-nonexpansive mappings in  Hilbert spaces. Under some suitable assumptions imposed on some parameters and operators involved, we proved the strong convergence theorems of these algorithms. Furthermore, as an application, we gave the strong convergence theorem for the split feasibility problem.  The results presented in this chapter, not only extend the result of Chen et al., \cite{chen2015extra} but also   extend, improve and generalize   the results of;  Takahashi and Toyoda \cite{takahashi2003weak}, Nadezhkina and Takahashi \cite{nadezhkina2006weak}, Ceng et al., \cite{ceng2013relaxed} and Li and He \cite{li2015new}   in    the following ways:
 \begin{itemize}
\item The theorem of Chen et al., \cite{chen2015extra} gave the weak convergence results while ours gave the strong convergence results.
\item The technique of proving our results is entirely different from that of Chen et al., \cite{chen2015extra}. Furthermore, the algorithms of Chen et al., \cite{chen2015extra} involve the class of nonexpansive mappings while our algorithm includes the class of quasi-nonexpansive mappings which are more general that nonexpansive mappings.
\item The method for finding the solution of the split feasibility and fixed point problems is more general that the method of finding the solution to split feasibility problem.
\item The theorem of Li and He \cite{li2015new} gave the strong convergence results for the split feasibility problem while ours gave the strong convergence for the split feasibility and fixed point problems. Furthermore, our algorithms generalize that of Li and He  \cite{li2015new}. For instance, in Theorem \ref{T6.1} Algorithm (\ref{6.10}) take $y_{n}=x_{n}$ and $\beta_{n}=0$, hence, our algorithm reduces to that of Li and He  \cite{li2015new} Theorem 2.1 Algorithm 2.1.
\item Our theorems gave the strong convergence for the solution of the split feasibility and fixed point problems for the class of quasi-nonexpansive mappings, while the results of Ceng et al., \cite{ceng2012extragradient} gave a weak convergence result for the solution of the split feasibility and fixed point problems for the class of nonexpansive mappings.
\item The split feasibility and fixed point problems is a fascinating problem. It generalizes the split feasibility problem  (SFP) and fixed point problem (FPP). All the results and conclusions that are true for the split feasibility and fixed point problems continue to holds for these problems (SFP, FPP), and it shows the significance and the range of applicability of split feasibility and fixed point problems.
\item The novelty of our theorems gives strong convergence results while the theorem of;  \cite{takahashi2003weak} Nadezhkina and Takahashi \cite{nadezhkina2006weak}, Ceng et al., \cite{ceng2013relaxed} and Li and He \cite{li2015new} all gives weak convergence results.
 \end{itemize}

\section{Conclusion}
In this work, we have studied the split common fixed point problems and its applications. We have  suggested some  algorithms for solving  this split common fixed point problems and its variant forms for  different classes of nonlinear mappings.\\

Proceeding systematically in our work, we gave the basic definitions and results from the literature. Also, we briefly provided an overview of the split common fixed point problems and its variant forms in Section 2. In the next section, we have suggested and analysed iterative algorithms for solving the split common fixed point problems for the class of total quasi asymptotically nonexpansive mappings in Hilbert spaces. Also, we gave the strong convergence results of the proposed algorithms. Also, we considered an algorithm for solving this split common fixed point problems for the class of demicontractive mappings without any prior information on the normed on the bounded linear operator and established the strong convergence results of the proposed algorithm. \\

As a generalization of the split feasibility problem, we  proposed Ishikawa-type extra-gradient methods for solving the
split feasibility and fixed point problems for the class of quasi-nonexpansive mappings in Hilbert spaces. Under some suitable assumptions imposed on some parameters and operators involved, we proved the strong convergence theorems of these algorithms.
\\

In the end, we proposed a new problem called "Split Feasibility and Fixed Point Equality Problems (SFFPEP)" and study it for the class of quasi-nonexpansive mappings in Hilbert spaces. We also proposed new iterative methods for solving this SFFPEP and proved the convergence results of the proposed algorithms. In additions, as a generalization of SFFPEP, we consider another problem called "Split Common Fixed Equality Problems (SCFPEP)" and study it for the class of finite family of quasi-nonexpansive mappings in Hilbert spaces. Finally, We suggested some algorithms for solving this SCFPEP and proved the convergence results of the proposed algorithms.

\end{document}